\documentclass[10pt]{amsart}
      \usepackage[mathscr]{eucal}
      \usepackage{amsmath,amsfonts}
      \parskip\smallskipamount
          \newtheorem{theorem}{Theorem}[section]
      
      \newtheorem{proposition}[theorem]{Proposition}
      \newtheorem{corollary}[theorem]{Corollary}
      \newtheorem{lemma}[theorem]{Lemma}
      \newtheorem{example}[theorem]{Example}
      
      \newtheorem{remark}[theorem]{Remark}
      \makeatletter
      \@addtoreset{equation}{section}
      \makeatother

      \newcommand{\CC}{{\mathbb C}}
      \newcommand{\NN}{{\mathbb N}}

      \newcommand{\DD}{{\mathbb D}}
      
      \newcommand{\FF}{{\mathbb F}}

      \newcommand{\cA}{{\mathcal A}}
      
      \newcommand{\cC}{{\mathcal C}}
      \newcommand{\cD}{{\mathcal D}}
      \newcommand{\cE}{{\mathcal E}}
      
      \newcommand{\cG}{{\mathcal G}}
      \newcommand{\cH}{{\mathcal H}}
      \newcommand{\cK}{{\mathcal K}}
      
      \newcommand{\cM}{{\mathcal M}}
      \newcommand{\cQ}{{\mathcal Q}}
      \newcommand{\cN}{{\mathcal N}}
      \newcommand{\cP}{{\mathcal P}}
      \newcommand{\cR}{{\mathcal R}}
      \newcommand{\cS}{{\mathcal S}}
      
      \newcommand{\cU}{{\mathcal U}}
      \newcommand{\cV}{{\mathcal V}}
      \newcommand{\cY}{{\mathcal Y}}
      \newcommand{\cX}{{\mathcal X}}

      \newdimen\expt
      \def\boxit#1{\setbox0\hbox{$\displaystyle{#1}$}
            \hbox{\lower.4\expt
       \hbox{\lower3\expt\hbox{\lower\dp0
            \hbox{\vbox{\hrule height.4\expt
       \hbox{\vrule width.4\expt\hskip3\expt
            \vbox{\vskip3\expt\box0\vskip2\expt}%
       \hskip3\expt\vrule width.4\expt}\hrule height.4\expt}}}}}}
      
\begin{document}
       \pagestyle{myheadings}
      \markboth{ Gelu Popescu}{ Noncommutative Berezin transforms and model theory }

      \title [ Noncommutative Berezin transforms and model theory]
      { Noncommutative Berezin transforms and model theory
      }
        \author{Gelu Popescu}
\date{February 28, 2007}
      \thanks{Research supported in part by an NSF grant}
      \address{Department of Mathematics, The University of Texas
      at San Antonio \\ San Antonio, TX 78249, USA}
      \email{\tt gelu.popescu@utsa.edu}

\begin{abstract}
In this paper, we initiate the  study  of  a class ${\bf D}_p^m(\cH)$
of noncommutative domains
     of $n$-tuples of bounded linear operators
    on a Hilbert space $\cH$, where $m\geq 2$, $n\geq2$, and $p$ is a positive
     regular  polynomial in $n$ noncommutative indeterminates. These
     domains
      are  defined  by certain positivity conditions on $p$, i.e.,
          $$
{\bf D}_p^m(\cH):=\left\{X:=(X_1,\ldots,X_n) :\ (1-p)^k(X,X^*) \geq
0 \ \text{ for } \ 1\leq k\leq m \right\}.
$$
    Each such a domain  has a  universal model $(W_1,\ldots, W_n)$ of
weighted shifts acting on the full Fock space $F^2(H_n)$  with $n$
generators. The  study of  ${\bf D}_p^m(\cH)$ is close related to
the study of the weighted shifts $W_1,\ldots,W_n$, their joint
invariant subspaces, and the representations of the algebras they
generate: the domain algebra $\cA_n({\bf D}_p^m)$, the Hardy algebra
$F_n^\infty({\bf D}_p^m)$, and the $C^*$-algebra $C^*(W_1,\ldots,
W_n)$. A good part of this paper deals with these issues.

The main tool, which we  introduce
      here, is  a
     noncommutative Berezin type  transform associated with each
    $n$-tuple of operators in  ${\bf D}_p^m(\cH)$.
    The study of this transform and its boundary behavior leads to Fatou type results,
      functional calculi, and
   a model theory for $n$-tuples of operators in ${\bf D}_p^m(\cH)$.
These results extend to  noncommutative  varieties
$\cV_{p,\cQ}^m(\cH)\subset {\bf D}_p^m(\cH)$
    generated by   classes $\cQ$  of noncommutative polynomials.
    When $m\geq 2$, $n\geq2$, $p=Z_1+\cdots +Z_n$,  and  $\cQ=0$, the elements
    of the corresponding variety ${\cV }_{p,\cQ}^m(\cH)$ can be seen as
    multivariable noncommutative analogues of Agler's
    $m$-hypercontractions.

Our results apply, in  particular, when $\cQ$ consists of the
noncommutative polynomials
  $Z_iZ_j-Z_jZ_i$,  $i,j=1,\ldots, n$. In this case, the model space
  is a
      symmetric weighted Fock space $F_s^2({\bf D}_p^m)$,  which is  identified with
      a reproducing kernel Hilbert space of holomorphic functions on a Reinhardt domain
      in $\CC^n$, and the universal model   is
      the $n$-tuple $(M_{\lambda_1},\ldots,
      M_{\lambda_n})$  of multipliers  by the coordinate functions.
In this particular case, we obtain  a model theory for commuting
$n$-tuples of operators in ${\bf D}_p^m(\cH)$, recovering several
results already existent in the literature.
\end{abstract}

      \maketitle

\bigskip

\section*{Introduction}

Let $\FF_n^+$ be the unital free semigroup on $n$ generators
$g_1,\ldots, g_n$ and the identity $g_0$,  and consider a polynomial
$q=q(Z_1,\ldots, Z_n)=\sum c_\alpha Z_\alpha$   in noncommutative
indeterminates $Z_1,\ldots,Z_n$, where we denote $Z_\alpha:=
Z_{i_1}\ldots Z_{i_k}$ if $\alpha=g_{i_1}\ldots g_{i_k}\in \FF_n^+$,
\ $i_1,\ldots i_k\in \{1,\ldots,n\}$, and $Z_{g_0}:=I$. We associate
with $q$
  the operator
$$q(X,X^*):=\sum c_\alpha X_\alpha X_\alpha^*,
$$
where $X:=(X_1,\ldots,X_n)\in B(\cH)^n$ and $B(\cH)$ is the algebra
of all bounded linear operators  on a Hilbert space  $\cH$. Let
$p=p(Z_1,\ldots, Z_n)=\sum a_\alpha Z_\alpha$, $a_\alpha\in \CC$, be
a positive regular polynomial, i.e., $a_\alpha\geq 0$, $a_{g_0}=0$,
and $a_{g_i}>0$, $i=1,\ldots,n$. Given $m,n\in\{1,2,\ldots \}$, we
define the noncommutative domain
$$
{\bf D}_p^m(\cH):=\left\{X:=(X_1,\ldots,X_n)\in B(\cH)^n:\
(1-p)^k(X,X^*) \geq 0 \ \text{ for } \ 1\leq k\leq m \right\}.
$$
In the last fifty years, these domains have been studied in several
particular cases. Most of all, we should mention that the study of
the closed operator unit ball
$$
[B(\cH)]_1^-:=\{X\in B(\cH):\ I-XX^*\geq 0\}
$$
(which corresponds to the case $m=1$, $n=1$, and $p=Z$) has
generated the celebrated
 Sz.-Nagy--Foias theory of contractions on Hilbert spaces and has had profound
 implications  in  function theory, interpolation,
prediction theory, scattering theory, and linear system theory (see
\cite{SzF-book}, \cite{FF-book}, \cite{FFGK-book}, \cite{BaGR},
etc).
 The case when $m=1$, $n\geq 2$, and $p=Z_1+\cdots+ Z_n$, corresponds to
 the closed operator  ball
 $$
 [B(\cH)^n]_1^-:=\left\{
(X_1,\ldots, X_n)\in B(\cH)^n:\ I-X_1 X_1^*-\cdots -X_nX_n^*\geq 0
 \right\}
$$
and its study has generated a {\it free } analogue of
Sz.-Nagy--Foias theory (see \cite{F}, \cite{B}, \cite{Po-models},
\cite{Po-isometric}, \cite{Po-charact}, \cite{Po-multi},
\cite{Po-von},  \cite{Po-funct}, \cite{Po-analytic}, \cite{Po-disc},
\cite{Po-interpo}, \cite{Po-poisson},  \cite{DKS}, \cite{BV},
\cite{Po-curvature},
  \cite{Po-entropy}, \cite{Po-varieties}, \cite{Po-varieties2}, \cite{Po-unitary},
etc.)
The commutative case, which  corresponds to the subvariety of
$[B(\cH)^n]_1$ determined by the commutators $Z_iZ_j-Z_jZ_i$,
$i,j=1,\ldots,n$, was considered by Drurry \cite{Dru}, extensively
studied by Arveson \cite{Arv}, and considered by the author
\cite{Po-poisson} in connection with  noncommutative Poisson
transforms. More general  subvarieties  in $[B(\cH)^n]_1$,
determined by classes of noncommutative polynomials, were considered
by the author in
 \cite{Po-varieties} and  \cite{Po-varieties2}.
The study of the unit ball $[B(\cH)^n]_1$ was  extended, in
\cite{Po-domains}, to  noncommutative domains
  ${\bf D}_p^m(\cH)$ (resp. subvarieties) when $m=1$, $n\geq 1$, and $p$ is
   any positive regular noncommutative  polynomial
   (resp.   free holomorphic function
   in the sense of \cite{Po-holomorphic}).

   In this paper, we initiate the  study  of  noncommutative domains
    ${\bf D}_p^m(\cH)$, when $m\geq 2$,
    $n\geq 2$, and $p$ is
   any positive regular noncommutative  polynomial. What makes the case $m\geq 2$ quite
    different from the case $m=1$ is that  ${\bf D}_p^m(\cH)$ is not
     a ball-like domain, when $m\geq 2$.  This can be seen even in the single variable case
     ($n=1$) (see \cite{Ag1}, \cite{Ag2},
     \cite{O1}, \cite{O2}).
We introduce  a class of
      noncommutative Berezin transforms associated with any
    $n$-tuple of operators in  ${\bf D}_p^m(\cH)$.
    The study of these transforms and  their boundary behavior leads to Fatou type results,
      functional calculi, and
   a model theory for $n$-tuples of operators in ${\bf D}_p^m(\cH)$.
Our results extend to  noncommutative  varieties $\cV_{p,\cQ}^m(\cH)$
    generated
    by   classes $\cQ$  of noncommutative polynomials, i.e.,
    $$
\cV_{p,\cQ}^m(\cH):=\left\{(X_1,\ldots, X_n)\in {\bf D}_p^m(\cH): \
 q(X_1,\ldots, X_n)=0, \, q\in \cQ\right\}.
 $$

In   Section 1, we associate with each $m,n\in\{1,2,\ldots\}$ and
each  positive regular noncommutative
 polynomial $p=p(Z_1,\ldots, Z_n)=\sum a_\alpha Z_\alpha$,
 a noncommutative domain
${\bold D}_p^m(\cH)\subset B(\cH)^n$ and a unique $n$-tuple
$(W_1,\ldots, W_n)$  of weighted shifts acting on the full Fock
space $F^2(H_n)$ with $n$ generators.
 They will play the role of the {\it universal model }
for the elements of ${\bold D}_p^m(\cH)$.  We also introduce the
$n$-tuple $(\Lambda_1,\ldots, \Lambda_n)$ associated with ${\bold
D}_p^m(\cH)$, which  turns out to be the universal model associated
with the noncommutative domain
 ${\bold D}_{\widetilde p}^m(\cH)$, where  $\widetilde
p=\widetilde p(Z_1,\ldots, Z_n)=\sum a_{\tilde \alpha} Z_\alpha$ and
$\widetilde \alpha$ denotes the reverse of $\alpha =g_{i_1}\cdots
g_{i_k}$, i.e., $\widetilde \alpha:= g_{i_k}\cdots g_{i_1}$.

In Section 2, we introduce a {\it noncommutative Berezin transform}
${\bf B}_T$
 associated with  each  $n$-tuple of operators $T:=(T_1,\ldots,
T_n)\in {\bf D}_p^m(\cH)$ with the joint spectral radius
$r_p(T_1,\ldots, T_n)<1$.  More precisely,
  the map
${\bf B}_T:B(F^2(H_n)) \to B(\cH)$  is defined by
\begin{equation*}
 \left<{\bf B}_T[g]x,y\right>:=
\left<\left( I-\sum_{|\alpha|\geq 1} \overline{a}_{\tilde\alpha}
\Lambda_\alpha^* \otimes T_{\tilde\alpha} \right)^{-m}
 (g\otimes \Delta_{T,m,p}^2)
  \left(
I-\sum_{|\alpha|\geq 1} a_{\tilde\alpha} \Lambda_\alpha \otimes
T_{\tilde\alpha}^* \right)^{-m}(1\otimes x), 1\otimes y\right>
\end{equation*}
  where $\Delta_{T,m,p}:= [(1-p)^m(T,T^*)]^{1/2}$ and $x,y\in\cH$.
We remark that in the particular case when: $m=1$, $n=1$, $p=Z$,
     $\cH=\CC$,
 and $T=\lambda\in \DD$, we recover the Berezin transform \cite{Be} of a bounded
linear operator on the Hardy space $H^2(\DD)$, i.e.,
$$
{\bf B}_\lambda [g]=(1-|\lambda|^2)\left<g k_\lambda,
k_\lambda\right>,\quad g\in B(H^2(\DD)),
$$
where $k_\lambda(z):=(1-\overline{\lambda} z)^{-1}$ and  $z,
\lambda\in \DD$. The noncommutative Berezin  transform which will
play an important role in this paper.

 First, we show that the Berezin transform has an
extension $\widetilde{\bf B}_T:B(F^2(H_n))\to B(\cH)$  to any
$n$-tuple $T\in {\bf D}_p^m(\cH)$. This is used to prove that
 the restriction of  $\widetilde {\bf B}_T$
 to  the operator system $\cS:=\overline{\text{\rm  span}} \{ W_\alpha W_\beta^*;\
\alpha,\beta\in \FF_n^+\}$ is
    a unital completely contractive linear map
such that
 $$
\widetilde{\bf B}_T[W_\alpha W_\beta^*]=T_\alpha T_\beta^*, \quad
\alpha,\beta\in \FF_n^+,
$$
when $T:=(T_1,\ldots, T_n)\in {\bf D}_p^m(\cH)$ is a
  {\it pure}
 $n$-tuple of operators  (i.e. $p^k(T,T^*)\to 0$ strongly as $k\to \infty$).
We obtain a similar result  for $n$-tuple of operators with the
radial property, i.e., $(rT_1,\ldots, rT_n)\in {\bf D}_p^m(\cH)$ for
any $r\in (\delta,1]$ and some $\delta\in(0,1)$. In this case, we
show that
$$
\Psi(g):=\lim_{r\to 1} {\bf B}_{rT}[g],\quad g\in \cS,
$$
exists in the norm operator topology and defines a  unital
completely contractive map $\Psi:\cS\to B(\cH)$ such that
$\Psi(W_\alpha W_\beta^*)=T_\alpha T_\beta^*, \quad \alpha,\beta\in
\FF_n^+$.

In  Section 3, we introduce the Hardy algebra $F_n^\infty({\bf
D}^m_p)$ (resp. $R_n^\infty({\bf D}^m_p)$) associated with the
noncommutative domain ${\bf D}^m_p$  and prove some basic
properties. We mention that an $n$-tuple of operators $T:=(T_1,\ldots, T_n)\in
{\bf D}^m_p(\cH)$ is  called {\it completely non-coisometric}
(c.n.c.) if there is no vector $h\in \cH$, $h\neq 0$, such that
$\left<p^k(T,T^*)h,h\right>=\|h\|^2$ for any $k=1,2,\ldots$. The
main result  of Section 3 is an $F_n^\infty({\bf
D}^m_p)$--functional calculus for
 (c.n.c.) $n$-tuples of operators in
the noncommutative domain ${\bf D}^m_p(\cH)$.
 More
precisely, we show that if $T:=(T_1,\ldots, T_n)$ is a c.n.c.
$n$-tuple
   of operators   in a noncommutative domain ${\bf
 D}^m_p(\cH)$ with the radial property,
 then
 $$\Phi(g):=\text{\rm SOT-}\lim_{r\to 1} g(rT_1,\ldots, rT_n), \qquad
  g=g(W_1,\ldots,
 W_n)\in F_n^\infty({\bf D}_p^m),
 $$
 exists in the strong operator topology   and defines a map
 $\Phi:F_n^\infty({\bf D}_p^m)\to B(\cH)$ with the following
 properties:
\begin{enumerate}
\item[(i)]
$\Phi(g)=\text{\rm SOT-}\lim\limits_{r\to 1}{\bf B}_{rT}[g]$, where
${\bf B}_{rT}$ is the Berezin transform  at  $rT\in {\bf
D}_p^m(\cH)$;
\item[(ii)] $\Phi$ is    WOT-continuous (resp.
SOT-continuous)  on bounded sets;
\item[(iii)]
$\Phi$ is a unital completely contractive homomorphism.
\end{enumerate}

In Section 4,  we find all the eigenvectors  for $W_1^*,\ldots,
W_n^*$, where $(W_1,\ldots, W_n)$  is the  universal model
associated with the noncommutative domain ${\bf D}^m_p$. As
consequences, we
 identify the $w^*$-continuous multiplicative
linear functional on the Hardy algebra $F_n^\infty({\bf D}^m_p)$ and
find  the joint right spectrum of  $(W_1,\ldots, W_n)$. We introduce
the symmetric weighted Fock space $F_s^2({\bf D}^m_p)$ and identify
it with $H^2({\bf D}_{f,\circ}^1(\CC))$, the reproducing kernel
Hilbert space
 with reproducing
kernel $K_p:{\bf D}_{p,\circ}^1(\CC)\times {\bf
D}_{p,\circ}^1(\CC)\to \CC$ defined by
$$
K_p(\mu,\lambda):=\frac{1}{\left(1-\sum  a_\alpha \mu_\alpha
\overline{\lambda}_\alpha\right)^m}\quad \text{ for all }\
\lambda,\mu\in {\bf D}_{p,\circ}^1(\CC),
$$
where
$${\bf D}_{p,\circ}^1(\CC):=\left\{ \lambda=(\lambda_1,\ldots, \lambda_n)\in
\CC^n:\ \sum  a_\alpha |\lambda_\alpha|^2<1\right\}\subset {\bf
D}^m_p(\CC),
$$
  $\lambda_\alpha:=\lambda_{i_1}\cdots \lambda_{i_m}$ if
$\alpha=g_{i_1}\cdots g_{i_m}\in \FF_n^+$, and $\lambda_{g_0}$=1.

 We show that
the algebra  $H^\infty({\bf D}_{p,\circ}^1(\CC))$ of all multipliers
of the Hilbert space $H^2({\bf D}_{p,\circ}^1(\CC))$ is reflexive
and coincides with the weakly closed algebra generated by the
identity and  the multipliers $M_{\lambda_1},\ldots M_{\lambda_n}$
by the coordinate functions. Moreover, the multipliers
$M_{\lambda_1},\ldots M_{\lambda_n}$ can be identified with the
operators $L_1,\ldots, L_n$, where
$$L_i:=P_{F_s^2({\bf D}^m_p)} W_i|_{F_s^2({\bf D}^m_p)}, \qquad i=1,\ldots,
n,
$$
and $(W_1,\ldots, W_n)$  is the  universal model associated with the
noncommutative domain ${\bf D}^m_p$. Section 4 will play an
important role in connecting the  results of the present paper  to
analytic function theory on Reinhardt domains in $\CC^n$, as well
as, to model theory for commuting $n$-tuples of operators.

In Section 5, we consider noncommutative varieties
$\cV_{p,\cQ}^m(\cH)\subset {\bf D}_p^m(\cH)$ determined by  sets
$\cQ$ of noncommutative polynomials. We associate  with each such a
variety a {\it universal model} $(B_1,\ldots, B_n)\in
\cV_{p,\cQ}^m(\cN_\cQ)$, which is the compression of $(W_1,\ldots,
W_n)$ to an appropriate subspace $\cN_\cQ$ of the full Fock space
$F^2(H_n)$. We introduce the {\it constrained noncommutative Berezin
transform} ${\bf B}^c_T:B(\cN_\cQ)\to B(\cH)$ and use it to obtain
analogues of the results of Section 2, for subvarieties. We also
show that,  if the constants belong to the subspace $\cN_\cQ$, then
the $C^*$-algebra $C^*(B_1,\ldots, B_n)$ is irreducible  and all the
compacts operators in $B(\cN_\cQ)$ are contained in the operator
space
 $\overline{\text{\rm span}}\{B_\alpha B_\beta^*:\ \alpha,\beta\in \FF_n^+\}$.
 These results are vital for the development of model theory on noncommutative
  varieties.

  In Section 6, we obtain dilation and model theorems for the elements of
  the noncommutative variety $\cV_{p,\cQ}^m(\cH)$.  First, we prove that
   an $n$-tuple
  of operators $T:=(T_1,\ldots, T_n)\in B(\cH)^n$ is a pure element of
$\cV_{p,\cQ}^m(\cH)$ if and only if
$$T_i^*=(B_i^*\otimes I_\cD)|_\cH,
 \quad i=1,\ldots,n,
 $$
 where
  $\cH$ is an invariant subspace under each operator
 $B_i^*\otimes I_\cD$, $i=1,\ldots, n$,  $\cD:=\overline{\Delta_{p,m,T}\cH}$,  and
  $\Delta_{p,m,T}:=[(1-p)^m(T,T^*)]^{1/2}$.

When $(T_1,\ldots, T_n)\in \cV_{p,\cQ}^m(\cH) $
 is an $n$-tuple of operators (on a separable Hilbert space $\cH)$
 with the  radial property  and $\cQ$ is a set of
    homogenous noncommutative polynomials,
    we show that  there exists  a $*$-representation $\pi:C^*(B_1,\ldots, B_n)\to
B(\cK_\pi)$  on a separable Hilbert space $\cK_\pi$,  which
annihilates the compact operators and
$$
p(\pi(B), \pi(B)^*)=I_{\cK_\pi}, \quad \text{ where }\
 \pi(B):=(\pi(B_1),\ldots, \pi(B_n)),
$$
such that $T_i^*=V_i^*|\cH$ for $ i=1,\ldots, n$, where the operators
$$
V_i:=\left[\begin{matrix} B_i\otimes
I_{\cD}&0\\0&\pi(B_i)
\end{matrix}\right],\quad i=1,\ldots,n,
$$
are acting on the Hilbert space $\tilde\cK:=(\cN_\cQ\otimes \cD)\oplus
\cK_\pi$
and
$\cH$ is  identified with a $*$-cyclic co-invariant subspace of
$\tilde\cK$ under  each operator $V_i$, $i=1,\ldots,n$.

In
   the single variable case, when  $m\geq 2$,  $n=1$,  $p=Z$, and $\cQ=0$,  the
    corresponding variety coincides with the set of all
    $m$-hypercontractions  studied by Agler  in \cite{Ag1}, \cite{Ag2},
    and  recently by
    Olofsson \cite{O1}, \cite{O2}.\
    When $m\geq 2$, $n\geq2$, $p=Z_1+\cdots +Z_n$,  and  $\cQ=0$, the elements
    of the corresponding domain ${\bf D}_p^m(\cH)$ can be seen as
    multivariable noncommutative analogues of Agler's
    $m$-hypercontractions.

In the particular case when $\cQ_c$  coincides with the set of
polynomials
  $Z_iZ_j-Z_jZ_i$,  $i,j=1,\ldots, n$, we can combine the results of Section 4 and Section 6
  to  recover  several results concerning
   model theory for commuting $n$-tuples of operators.
   The case $m\geq 2$, $n\geq 2$,  $p=Z_1+\cdots + Z_n$, and $\cQ=\cQ_c$, was studied
   by Athavale \cite{At}, M\" uller \cite{M}, M\" uller-Vasilescu \cite{MV},
   Vasilescu \cite{Va}, and Curto-Vasilescu \cite{CV1}.
   Some  of these results concerning  model theory were extended by
   S.~Pott \cite{Pot} to positive regular  polynomials in commuting indeterminates.

We should mention that most of the results  of this paper  are
presented in a more general setting, namely, when the polynomial $p$
is replaced by a positive regular free holomorphic function (see
Section 1 for terminology).   In a future paper, we expect to use
these results   to obtain functional models for the elements of the
noncommutative domain ${\bf D}_p^m(\cH)$ (resp. subvariety
$\cV_{p,\cQ}^m(\cH)$), based on characteristic functions.

\bigskip

\section{Noncommutative domains and universal models}

In this section, we associate with each  positive regular free
holomorphic function $f$ on $[B(\cH)^n]_\rho$, $\rho>0$,  and each
$m,n\in\{1,2,\ldots\}$,
 a noncommutative domain
${\bold D}_f^m(\cH)\subset B(\cH)^n$ and a unique $n$-tuple
$(W_1,\ldots, W_n)$ of weighted shifts. This $n$-tuple of operators
will play the role of the {\it universal model} for the elements  of
${\bold D}_f^m(\cH)$. We also introduce the $n$-tuple
$(\Lambda_1,\ldots, \Lambda_n)$  associated with ${\bold
D}_f^m(\cH)$, which turns out to be the universal model  for the
elements of the  noncommutative domain ${\bf D}_{\tilde f}$.

      Let $H_n$ be an $n$-dimensional complex  Hilbert space with orthonormal
      basis
      $e_1$, $e_2$, $\dots,e_n$, where $n\in\{1,2,\dots\}$.        We consider
      the full Fock space  of $H_n$ defined by
      $$F^2(H_n):=\bigoplus_{k\geq 0} H_n^{\otimes k},$$
      where $H_n^{\otimes 0}:=\CC 1$ and $H_n^{\otimes k}$ is the (Hilbert)
      tensor product of $k$ copies of $H_n$.
      Define the left creation
      operators $S_i:F^2(H_n)\to F^2(H_n), \  i=1,\dots, n$,  by
      $$
       S_i\varphi:=e_i\otimes\varphi, \quad  \varphi\in F^2(H_n),
      $$
      and  the right creation operators
      $R_i:F^2(H_n)\to F^2(H_n)$, \  $i=1,\dots, n$,  by
      $
       R_i\varphi:=\varphi\otimes e_i$, \ $ \varphi\in F^2(H_n)$.

The    algebra   $F_n^\infty$
  and  its norm closed version,
  the noncommutative disc
 algebra  $\cA_n$,  were introduced by the author   \cite{Po-von} in connection
   with a multivariable noncommutative von Neumann inequality.
$F_n^\infty$  is the algebra of left multipliers of $F^2(H_n)$  and
can be identified with
 the
  weakly closed  (or $w^*$-closed) algebra generated by the left creation operators
   $S_1,\dots, S_n$  acting on   $F^2(H_n)$,
    and the identity.
     The noncommutative disc algebra $\cA_n$ is
    the  norm closed algebra generated by
   $S_1,\dots, S_n$,
    and the identity. For basic properties concerning
    the  noncommutative analytic Toeplitz algebra   $F_n^\infty$
  we refer to
\cite{Po-charact},  \cite{Po-multi},  \cite{Po-funct},
\cite{Po-analytic}, \cite{Po-disc}, \cite{Po-poisson},
  \cite{DP1}, \cite{DP2},   \cite{DP}, \cite{DKP},
  and \cite{ArPo2}.

Let $\FF_n^+$ be the unital free semigroup on $n$ generators
$g_1,\ldots, g_n$ and the identity $g_0$.  The length of $\alpha\in
\FF_n^+$ is defined by $|\alpha|:=0$ if $\alpha=g_0$  and
$|\alpha|:=k$ if
 $\alpha=g_{i_1}\cdots g_{i_k}$, where $i_1,\ldots, i_k\in \{1,\ldots, n\}$.
If $X:=(X_1,\ldots, X_n)\in B(\cH)^n$, where $B(\cH)$ is the algebra
of all bounded linear operators on the Hilbert space $\cH$,    we
denote $X_\alpha:= X_{i_1}\cdots X_{i_k}$  and $X_{g_0}:=I_\cH$.

We say that  $f=f(X_1,\ldots, X_n):= \sum_{\alpha\in \FF_n^+}
a_\alpha X_\alpha$, \ $a_\alpha\in \CC$,  is  a free holomorphic
function on the noncommutative ball  $[B(\cH)^n]_\rho$ for some
$\rho>0$, where
$$[B(\cH)^n]_\rho:=\{(X_1,\ldots, X_n)\in B(\cH)^n: \
\|X_1X_1^*+\cdots + X_nX_n^*\|< \rho\},
$$
if the series $\sum_{k=0}^\infty \sum_{|\alpha|=k} a_\alpha
X_\alpha$ is convergent in the operator norm topology for any
$(X_1,\ldots, X_n)\in [B(\cH)^n]_\rho$.
  According to  \cite{Po-holomorphic},  $f$  is  a free holomorphic
function on   $[B(\cH)^n]_\rho$ if and only if
\begin{equation*}
\limsup_{k\to\infty} \left( \sum_{|\alpha|=k}
|a_\alpha|^2\right)^{1/2k}\leq \frac{1}{\rho}.
\end{equation*}
Throughout this paper, we
  assume that  $a_\alpha\geq 0$ for any $\alpha\in \FF_n^+$, \ $a_{g_0}=0$,
 \ and  $a_{g_i}>0$, $i=1,\ldots, n$.
 A function $f$ satisfying  all these conditions on the coefficients is
 called a {\it positive regular free holomorphic function on}
 $[B(\cH)^n]_\rho$  for some
$\rho>0$.

\begin{lemma}\label{b-alpha}
Let $f$ be a positive regular free holomorphic function on
$[B(\cH)^n]_\rho$, $\rho>0$,  with
 the representation
 $f(X_1,\ldots, X_n):= \sum_{\alpha\in \FF_n^+} a_\alpha
X_\alpha$, \ $a_\alpha\in \CC$. Then there exists $r\in (0,1)$ such that
$\|f(rS_1,\ldots, rS_n)\|<1$ and, for any $m=1,2,\ldots$,
$$
[1-f(rS_1,\ldots, rS_n)]^{-m}=\sum_{k=0}^\infty \sum_{|\alpha|=k}
b_\alpha^{(m)} r^{|\alpha|} S_\alpha,
$$
where $b_{g_0}^{(m)}=1$ and
\begin{equation}
\label{b-al}
 b_\alpha^{(m)}= \sum_{j=1}^{|\alpha|}
\sum_{{\gamma_1\cdots \gamma_j=\alpha }\atop {|\gamma_1|\geq
1,\ldots, |\gamma_j|\geq 1}} a_{\gamma_1}\cdots a_{\gamma_j}
\left(\begin{matrix} j+m-1\\m-1
\end{matrix}\right)  \qquad
\text{ if } \ |\alpha|\geq 1.
\end{equation}
\end{lemma}

\begin{proof}
Due to the Schwartz type lemma for free holomorphic functions on the
open unit ball  $[B(\cH)^n]_1$ (see \cite{Po-holomorphic}), there
exists $r>0$ such that
    $f(rS_1,\ldots, rS_n)$ is  in the noncommutative disc algebra
  $\cA_n$ and   $\|f(rS_1,\ldots, rS_n)\|<1$.
Therefore, the operator $I-f(rS_1,\ldots, rS_n)$ is invertible with
its inverse
 $g(rS_1,\ldots, rS_n):=[I-f(rS_1,\ldots, rS_n)]^{-1}$  in $\cA_n\subset
 F_n^\infty$. Assume that $g(rS_1,\ldots, rS_n)$
  has the Fourier representation
$  \sum_{\alpha\in \FF_n^+} b_\alpha^{(1)} r^{|\alpha|} S_\alpha $
for some constants $b_\alpha^{(1)}\in \CC$. Consequently,
   using the fact that $r^{|\alpha|}b_\alpha^{(1)}=P_\CC S_\alpha^*
g(rS_1,\ldots, rS_n)(1)$, we  deduce that
\begin{equation*}
\begin{split}
g(rS_1,\ldots, rS_n)&= I+ f(rS_1,\ldots, rS_n)+ f(rS_1,\ldots, rS_n)^2+\cdots \\
&=I+\sum_{k=1}^\infty \sum_{|\alpha|=k}\left(\sum_{j=1}^{|\alpha|}
\sum_{{\gamma_1\cdots \gamma_j=\alpha }\atop {|\gamma_1|\geq
1,\ldots,
 |\gamma_j|\geq 1}} a_{\gamma_1}\cdots a_{\gamma_j} \right) r^{|\alpha|}
  S_\alpha.
\end{split}
\end{equation*}
Due to the uniqueness of the Fourier representation of the elements
in $F_n^\infty$,  we deduce relation \eqref{b-al}, when $m=1$. Now,
we proceed  by induction over $m$. Assume that relation \eqref{b-al}
holds for $m$ and let us prove it for $m+1$. Notice that
\begin{equation*}
\begin{split}
&[I- f(rS_1,\ldots, rS_n)]^{-(m+1)}\\
 &\quad=[I- f(rS_1,\ldots,rS_n)]^{-m}
[I- f(rS_1,\ldots, rS_n)]\\
&\quad=\left\{ I+\sum_{|\omega|\geq 1} \left[\sum_{j=1}^{|\omega|}
\sum_{{\xi_1\cdots \xi_j=\omega }\atop {|\xi_1|\geq 1,\ldots,
 |\xi_j|\geq 1}} a_{\xi_1}\cdots a_{\xi_j} \left(\begin{matrix} j+m-1\\m-1
\end{matrix}\right)\right]
 r^{|\omega|} S_\omega\right\}\\
 &\quad\qquad \times
\left\{ I+\sum_{|\sigma|\geq 1} \left[\sum_{k=1}^{|\sigma|}
\sum_{{\epsilon_1\cdots \epsilon_k=\sigma }\atop {|\epsilon_1|\geq
1,\ldots,
 |\epsilon_k|\geq 1}} a_{\epsilon_1}\cdots a_{\epsilon_k}  \right] r^{|\sigma|}
 S_\sigma\right\}\\
 &\quad=
 I+\sum_{|\gamma|\geq 1}
 \left[\sum_{k=1}^{|\gamma|}
\sum_{{\epsilon_1\cdots \epsilon_k=\gamma }\atop {|\epsilon_1|\geq
1,\ldots,
 |\epsilon_k|\geq 1}} a_{\epsilon_1}\cdots a_{\epsilon_k}+
\sum_{j=1}^{|\gamma|} \sum_{{\xi_1\cdots \xi_j=\gamma }\atop
{|\xi_1|\geq 1,\ldots,
 |\xi_j|\geq 1}} a_{\xi_1}\cdots a_{\xi_j} \left(\begin{matrix} j+m-1\\m-1
\end{matrix}\right)\right.\\
&\quad\qquad \qquad  \qquad +\left. \sum_{{\omega
\sigma=\gamma}\atop {|\omega|\geq 1, |\sigma|\geq 1}}
\sum_{j=1}^{|\omega|}\sum_{k=1}^{|\sigma|} \sum_{{\xi_1\cdots
\xi_j=\gamma }\atop {|\xi_1|\geq 1,\ldots,
 |\xi_j|\geq 1}}
 \sum_{{\epsilon_1\cdots \epsilon_k=\gamma }\atop {|\epsilon_1|\geq
1,\ldots,
 |\epsilon_k|\geq 1}} \left(\begin{matrix} j+m-1\\m-1
\end{matrix}\right)
 a_{\xi_1}\cdots a_{\xi_j}a_{\epsilon_1}\cdots a_{\epsilon_k}
 \right] r^{|\gamma|} S_\gamma.
\end{split}
\end{equation*}
If we look closer to the  sums in the brackets, we notice that each
product $a_{\eta_1}\cdots a_{\eta_p}$, where $\eta_1\cdots
\eta_p=\gamma$ with $\eta_1,\ldots \eta_p\in \FF_n^+$ and
$|\eta_1|\geq 1,\ldots, |\eta_p|\geq 1$, occurs $p+1$ times. This is
because
\begin{equation*}
a_{\eta_1}\cdots a_{\eta_p}=
\begin{cases} a_{\epsilon_1}\cdots a_{\epsilon_k} & \text{ if }\
(\eta_1,\ldots \eta_p)=(\epsilon_1,\ldots, \epsilon_k)  \\
a_{\xi_1}\cdots a_{\xi_j}a_{\epsilon_1}\cdots a_{\epsilon_k} &
\text{ if }\ (\eta_1,\ldots \eta_p)=(\xi_1,\ldots,
\xi_j,\epsilon_1,\ldots, \epsilon_k)   \ \text{ and } \ j=1,\ldots, p-1\\
a_{\xi_1}\cdots a_{\xi_j}
 & \text{ if }\
(\eta_1,\ldots \eta_p)=(\xi_1,\ldots, \xi_j).
\end{cases}
\end{equation*}
Moreover, at each occurrence, the product $a_{\eta_1}\cdots
a_{\eta_p}$  has a coefficient which is equal to
\begin{equation*}
\begin{cases}   \left(\begin{matrix} m-1\\m-1
\end{matrix}\right) & \text{ if }\
(\eta_1,\ldots \eta_p)=(\epsilon_1,\ldots, \epsilon_k)  \\
  \left(\begin{matrix} j+m-1\\m-1
\end{matrix}\right) &
\text{ if }\ (\eta_1,\ldots \eta_p)=(\xi_1,\ldots,
\xi_j,\epsilon_1,\ldots, \epsilon_k)   \ \text{ and } \ j=1,\ldots, p-1\\
  \left(\begin{matrix} p+m-1\\m-1
\end{matrix}\right)
 & \text{ if }\
(\eta_1,\ldots \eta_p)=(\xi_1,\ldots, \xi_j).
\end{cases}
\end{equation*}
Hence, we deduce that the coefficient of  $a_{\eta_1}\cdots
a_{\eta_p}$ is equal to
$$
\sum_{j=0}^p \left(\begin{matrix} j+m-1\\m-1
\end{matrix}\right)=\left(\begin{matrix} p+m\\m
\end{matrix}\right).
$$
The latter equality can be easily deduced using the well-known
relation $$\left(\begin{matrix} j+m\\m
\end{matrix}\right)=\left(\begin{matrix} j+m-1\\m
\end{matrix}\right)+\left(\begin{matrix} j+m-1\\m-1
\end{matrix}\right)$$ for any $j=1,\ldots,p$.
Therefore, we have $[I- f(rS_1,\ldots,
rS_n)]^{-(m+1)}=\sum_{|\gamma|\geq 1} b_\gamma^{(m+1)} r^{|\gamma|}
S_\gamma$, where
$$
b_\gamma^{(m+1)}= \sum_{p=1}^{|\gamma|} \sum_{{\eta_1\cdots
\eta_p=\gamma }\atop {|\eta_1|\geq 1,\ldots, |\eta_p|\geq 1}}
a_{\eta_1}\cdots a_{\eta_p} \left(\begin{matrix} p+m\\m
\end{matrix}\right)  \qquad
\text{ if } \ |\gamma|\geq 1.
$$
This completes the induction and the proof.
\end{proof}

\begin{lemma}
\label{relations} Let $f$ be a positive regular free holomorphic
function on $[B(\cH)^n]_\rho$, $\rho>0$,  with
 the representation
 $f(X_1,\ldots, X_n):= \sum_{\alpha\in \FF_n^+} a_\alpha
X_\alpha$, \ $a_\alpha\in \CC$, and let $g:=1-(1-f)^m$,
$m=1,2,\ldots,$ have the representation $g(X_1,\ldots, X_n):=
\sum_{\gamma\in \FF_n^+} c_\gamma^{(m)} X_\gamma$, \ $a_\gamma\in
\CC$. Then the following relations hold:
\begin{equation}
\label{b_b}
b_\beta^{(m)}=\sum\limits_{{\gamma\alpha=\beta}\atop{\alpha\in
\FF_n^+,
 |\gamma|\geq 1}}
b_\alpha^{(m)} c_\gamma^{(m)}\qquad  \text{ if } \ \ |\beta|\geq 1 \
\text{ and }\ m=1,2,\ldots,
\end{equation}
and
\begin{equation}
\label{relations2}
b_\alpha^{(m)}= b_\alpha^{(m-1)}+
\sum\limits_{{\gamma\sigma=\alpha}\atop{\sigma\in \FF_n^+,
 |\gamma|\geq 1}}
b_\sigma^{(m)} a_\gamma\qquad  \text{ if } \ \ m\geq 2 \ \text{ and }\
\alpha\in \FF_n^+.
\end{equation}
\end{lemma}
\begin{proof}
Since
$$
\left\{I- [I-f(rS_1,\ldots, rS_n)]^m\right\}[I-f(rS_1,\ldots,
rS_n)]^{-m}=[I-f(rS_1,\ldots, rS_n)]^{-m}-I
$$
and using Lemma \ref{b-alpha}, we have
$$
\left(\sum_{k=0}^\infty \sum_{|\alpha|=k}b_\alpha^{(m)} r^{|\alpha|}
S_\alpha\right)
 \left(\sum_{p=1}^\infty
\sum_{|\gamma|=p}b_\gamma^{(m)} r^{|\gamma|} S_\gamma\right)=
\sum_{q=1}^\infty\sum_{|\beta|=q} b_\beta^{(m)} r^{|\beta|} S_\beta.
$$
Hence, using the uniqueness of the Fourier representation for the
elements in $F_n^\infty$, we obtain relation \eqref{b_b}. To prove
\eqref{relations2}, assume that $m\geq 2$ and notice that
$$
[I-f(rS_1,\ldots, rS_n)]^{-m}-f(rS_1,\ldots, rS_n)[I-f(rS_1,\ldots,
rS_n)]^{-m}-I=[I-f(rS_1,\ldots, rS_n)]^{-m+1}-I.
$$
Consequently, we have
$$
\sum_{k=0}^\infty \sum_{|\alpha|=k}b_\alpha^{(m)} r^{|\alpha|}
S_\alpha= \left(\sum_{q=1}^\infty\sum_{|\gamma|=q} b_\gamma^{(m)}
r^{|\gamma|} S_\gamma\right) \left(\sum_{p=0}^\infty
\sum_{|\sigma|=p}b_\sigma^{(m)} r^{|\sigma|} S_\sigma\right) +
\sum_{k=0}^\infty \sum_{|\alpha|=k}b_\alpha^{(m-1)} r^{|\alpha|}
S_\alpha.
$$
Using  again the uniqueness of the Fourier representation for the
elements in $F_n^\infty$, we deduce relation \eqref{relations2}.
This completes the proof.
\end{proof}

According to Lemma \ref{b-alpha}, we have $b_\alpha^{(m)}>0$ for any
$\alpha\in \FF_n^+$ and $m=1,2,\ldots$. We  define  now the diagonal
operators $D_i:F^2(H_n)\to F^2(H_n)$, $i=1,\ldots, n$, by setting
$$
D_ie_\alpha:=\sqrt{\frac{b_\alpha^{(m)}}{b_{g_i \alpha}^{(m)}}}
e_\alpha,\quad
 \alpha\in \FF_n^+.
$$
Due  to Lemma \ref{relations}, we have
$$
b_{g_i\alpha}^{(m)}\geq
\sum\limits_{{\gamma\sigma=g_i\alpha}\atop{\sigma\in \FF_n^+,
 |\gamma|\geq 1}}
b_\sigma^{(m)} a_\gamma\geq a_{g_i} b_\alpha^{(m)}.
$$
Since  $a_{g_i}>0$ for each $i=1,\ldots,n$, we deduce that
$$
\|D_i\|=\sup_{\alpha\in \FF_n^+} \sqrt{\frac{b_\alpha^{(m)}}{b_{g_i
\alpha}^{(m)}}}\leq \frac{1}{\sqrt{a_{g_i}}}, \quad i=1,\ldots,n.
$$
Now we define the {\it weighted left creation  operators}
$W_i:F^2(H_n)\to F^2(H_n)$, $i=1,\ldots, n$,  associated with the
  positive regular free holomorphic $f$   by setting $W_i:=S_iD_i$, where
 $S_1,\ldots, S_n$ are the left creation operators on the full
 Fock space $F^2(H_n)$.
Therefore, we have
\begin{equation} \label{w-shift}
W_i
e_\alpha=\frac {\sqrt{b_\alpha^{(m)}}}{\sqrt{b_{g_i \alpha}^{(m)}}}
e_{g_i \alpha}, \quad \alpha\in \FF_n^+,
\end{equation}
where the coefficients  $b_\alpha^{(m)}$, $\alpha\in \FF_n^+$, are
given by relation \eqref{b-al}.

Throughout this paper, we denote by $id$ the identity map acting on
 the algebra of all bounded linear operators an a Hilbert space.

\begin{theorem}
\label{prop-shif} Let $f$ be a positive regular free
 holomorphic function on $[B(\cH)^n]_\rho$, $\rho>0$,  and  $m=1,2,\ldots$.
 The weighted left creation
  operators $W_1,\ldots, W_n$     associated with $f$ and $m$, and defined  by
 relation \eqref{w-shift} have the following properties:
 \begin{enumerate}
 \item[(i)] $\sum\limits_{|\beta|\geq 1} a_\beta W_\beta W_\beta^*\leq I$, where the
convergence is in the strong operator topology;
 \item[(ii)]
 $\left(id-\Phi_{f,W}\right)^{m}(I)=P_\CC$, where $P_\CC$ is the
 orthogonal projection of $F^2(H_n)$ on $\CC$, and the map
 $\Phi_{f,W}:B(F^2(H_n))\to
B(F^2(H_n))$    is defined  by
$$\Phi_{f,W}(X)=\sum\limits_{|\alpha|\geq 1} a_\alpha W_\alpha
XW_\alpha^*,
$$
where the convergence is in the weak operator topology;
 \item[(iii)] $\lim\limits_{p\to\infty} \Phi^p_{f,W}(I)=0$ in the strong operator
 topology;
 \item[(iv)] $\sum\limits_{\beta\in \FF_n^+} b_\beta^{(m)}
 W_\beta\left[(id-\Phi_{f,W})^m(I)\right] W_\beta^*=I$, where the
 coefficients $b_\beta^{(m)}$ are defined by \eqref{b-al}, and the
  the convergence is in the strong operator topology.
 \end{enumerate}
\end{theorem}

\begin{proof} Using relation \eqref{b-al},
a simple calculation reveals that
\begin{equation}\label{WbWb}
W_\beta e_\gamma= \frac {\sqrt{b_\gamma^{(m)}}}{\sqrt{b_{\beta
\gamma}^{(m)}}} e_{\beta \gamma} \quad \text{ and }\quad W_\beta^*
e_\alpha =\begin{cases} \frac
{\sqrt{b_\gamma^{(m)}}}{\sqrt{b_{\alpha}^{(m)}}}e_\gamma& \text{ if
}
\alpha=\beta\gamma \\
0& \text{ otherwise }
\end{cases}
\end{equation}
 for any $\alpha, \beta \in \FF_n^+$.
Due to    \eqref{WbWb}, we deduce that
\begin{equation}\label{WW*}
W_\beta W_\beta^* e_\alpha =\begin{cases}
\frac {{b_\gamma^{(m)}}}{{b_{\alpha}^{(m)}}} e_\alpha & \text{ if } \alpha=\beta\gamma\\
0& \text{ otherwise. }
\end{cases}
\end{equation}
Since the case $m=1$ was considered in \cite{Po-domains}, we assume
that $m\geq 2$. Notice that $$\left(I-\sum\limits_{1\leq |\beta|\leq
N} a_\beta W_\beta W_\beta^*\right)
 e_\alpha= \frac{1}{b_{\alpha}^{(m)}}K_{N, \alpha} e_\alpha,
$$
 where $K_{N, \alpha}=b_{\alpha}^{(m)}$ if $\alpha=g_0$, and
$$
K_{N, \alpha}=b_{\alpha}^{(m)}- \sum_{\beta\gamma=\alpha, 1\leq
|\beta|\leq N} a_\beta b_\gamma^{(m)}  \qquad \text{if } \quad
|\alpha|\geq 1.
$$
Due to relation \eqref{relations2}, if $1\leq|\alpha|\leq N$, we
have $$K_{N, \alpha}= b_\alpha^{(m-1)} \leq {b_\alpha^{(m)}}.
$$
On the other hand, since $a_\beta\geq 0$, $b_\gamma^{(m)}\geq 0$ for
any $\alpha,\gamma\in \FF_n^+$, we have $K_{N,\alpha}\leq
b_\alpha^{(m)}$
 if $|\alpha| \geq 1$.
Hence, we deduce that  $0\leq K_{N, \alpha}\leq b_\alpha^{(m)}$,
whenever $|\alpha|>N$.
  On the  other hand,
notice that if  $1\leq N_1\leq N_2\leq |\alpha|$, then
$K_{N_2,\alpha}\leq K_{N_1,\alpha}$. Consequently,
$\left\{I-\sum\limits_{1\leq |\beta|\leq N} a_\beta W_\beta
 W_\beta^*\right\}_{N=1}^\infty$ is a decreasing
sequence of positive diagonal operators which  converges in the
strong operator topology. Hence, we deduce that
 $\sum\limits_{|\beta|\geq 1} a_\beta W_\beta W_\beta^*\leq I$, where the
convergence is in the strong operator topology.

We prove now part (ii). By \eqref{WW*}, the subspaces $\CC
e_\alpha$, $\alpha\in \FF_n^+$, are invariant under $W_\beta
W_\beta^*$, $\beta\in \FF_n^+$, and, therefore, they are also
invariant under $(id-\Phi_{f,W})^m(I)$. Consequently, it is enough
to show that $(id-\Phi_{f,W})^m(I) 1=1$ and
$$
\left<(id-\Phi_{f,W})^m(I)e_\alpha, e_\alpha\right>=0
$$
for any $\alpha\in \FF_n^+$ with $|\alpha|\geq 1$. The first
equality is obvious due to \eqref{WW*}. Using Lemma \ref{relations},
we deduce that
\begin{equation*}
\begin{split}
\left<(id-\Phi_{f,W})^m(I)e_\alpha, e_\alpha \right>&= \left<
e_\alpha-\sum_{|\beta|\geq 1} c_\beta^{(m)} W_\beta W_\beta^*
e_\alpha,
e_\alpha\right>\\
&= \frac{1}{b_\alpha^{(m)}} \left( b_{\alpha}^{(m)}-
\sum_{\beta\gamma=\alpha,   |\beta|\geq 1} c_\beta^{(m)}
b_\gamma^{(m)} \right)=0
\end{split}
\end{equation*}
if $\alpha\in \FF_n^+$ with $|\alpha|\geq 1$. Therefore,
$\left(id-\Phi_{f,W}\right)^{m}(I)=P_\CC$.

To prove part (iii), notice that relation \eqref{WW*} implies
$\Phi_{f,W}^p(I) e_\alpha=0$ if $p>|\alpha|$.  This shows that
$\lim\limits_{p\to\infty} \Phi^p_{f,W}(I)e_\alpha=0$ for any
$\alpha\in \FF_n^+$.  By part (i), we have $\|\Phi_{f,W}^p(I)\|\leq
1$ for any  $p\in \NN$. Now item (iii) follows.

It remains to  prove (iv). To this end, notice that
\begin{equation}\label{projW*}
P_\CC W_\beta^* e_\alpha =\begin{cases}
\frac {1}{\sqrt{b_{\beta}}}   & \text{ if } \alpha=\beta\\
0& \text{ otherwise,}
\end{cases}
\end{equation}
and, therefore  $\sum_{\beta\in \FF_n^+} b_\beta W_\beta P_\CC
W_\beta^* e_\alpha= e_\alpha$.  Using part (ii), we complete the
proof.
\end{proof}

We can also define the {\it weighted right creation operators}
$\Lambda_i:F^2(H_n)\to F^2(H_n)$ by setting $\Lambda_i:= R_i G_i$,
$i=1,\ldots, n$,  where $R_1,\ldots, R_n$ are
 the right creation operators on the full Fock space $F^2(H_n)$ and
 each  diagonal operator $G_i$, $i=1,\ldots,n$,  is defined by
$$
G_ie_\alpha:=\sqrt{\frac{b_\alpha^{(m)}}{b_{ \alpha g_i}^{(m)}}}
e_\alpha,\quad
 \alpha\in \FF_n^+,
$$
where the coefficients $b_\alpha^{(m)}$, $\alpha\in \FF_n^+$, are
given by relation \eqref{b-al}. In this case, we have
\begin{equation}\label{WbWb-r}
\Lambda_\beta e_\gamma= \frac {\sqrt{b_\gamma^{(m)}}}{\sqrt{b_{
\gamma \tilde\beta}^{(m)}}} e_{ \gamma \tilde \beta} \quad \text{
and }\quad \Lambda_\beta^* e_\alpha =\begin{cases} \frac
{\sqrt{b_\gamma^{(m)}}}{\sqrt{b_{\alpha}^{(m)}}}e_\gamma& \text{ if
}
\alpha=\gamma \tilde \beta \\
0& \text{ otherwise }
\end{cases}
\end{equation}
 for any $\alpha, \beta \in \FF_n^+$, where $\tilde \beta$ denotes
 the reverse of $\beta=g_{i_1}\cdots g_{i_k}$, i.e.,
 $\tilde \beta=g_{i_k}\cdots g_{i_1}$.
Using   Lemma \ref{relations}  and \eqref{WbWb-r}, we deduce that
$$\left(I-\sum\limits_{1\leq |\beta|\leq N}
a_{\tilde\beta} \Lambda_\beta \Lambda_\beta^*\right)
e_\alpha=\frac{1}{b_\alpha^{(m)}} \tilde K_{N, \alpha} e_\alpha,
$$
 where $\tilde K_{N, \alpha}=b_\alpha^{(m)}$ if $\alpha=g_0$, and
$$
\tilde K_{N, \alpha}=b_\alpha^{(m)}- \sum_{\gamma\tilde\beta=\alpha,
 \leq |\tilde\beta|\leq N}  a_{\tilde\beta} b_\gamma^{(m)}
\qquad \text{if } \quad |\alpha|\geq 1.
$$
As in the case of weighted left creation operators, one can show
that
 \begin{equation}
 \label{tild-Lamb} \sum\limits_{|\beta|\geq 1}
a_{\tilde\beta} \Lambda_\beta
 \Lambda_\beta^*\leq I\quad \text{ and } \quad \left(id-\Phi_{\tilde
 f,\Lambda}\right)^m(I)=P_\CC,
 \end{equation}
 where
  $\tilde{f}(X_1,\ldots,
X_n):=\sum_{|\alpha|\geq 1} a_{\tilde \alpha} X_\alpha$, $\tilde
\alpha$ denotes the reverse of $\alpha$, and $\Phi_{\tilde f,
\Lambda}(X):=\sum_{|\alpha|\geq 1} a_{\tilde \alpha} \Lambda_\alpha
X \Lambda_\alpha^*$, $X\in B(F^2(H_n))$, with the convergence is in
the weak operator topology. Since
$$
P_\CC \Lambda_\beta^* e_\alpha =\begin{cases}
\frac {1}{\sqrt{b_{\alpha}^{(m)}}}   & \text{ if } \alpha=\tilde\beta\\
0& \text{ otherwise,}
\end{cases}
$$
we deduce that
\begin{equation*}
\sum_{\beta\in \FF_n^+}b_{\tilde\beta}^{(m)} \Lambda_\beta
 \left[ (id-\Phi_{\tilde
 f,\Lambda})^m(I)\right]   \Lambda_\beta^* =I,
\end{equation*}
where the convergence is in the strong operator topology. Therefore,
we obtain a result similar to Theorem \ref{prop-shif} for the $n$-tuple
$(\Lambda_1,\ldots, \Lambda_n)$.

A linear map $\varphi:B(\cH)\to B(\cH)$ is called power bounded if
there exists a constant $M>0$ such that $\|\varphi^k\|\leq M$ for
any $k\in \NN:=\{1,2,\ldots\}$.

\begin{lemma}\label{phi-condition}
Let $\varphi:B(\cH)\to B(\cH)$  be a power bounded, positive linear
map and let $D\in B(\cH)$ be  a positive operator. If $m\in \NN$,
then
$$(id-\varphi)^m(D)\geq 0\quad \text{ if and only if }\quad  (id-\varphi)^k(D)\geq
0, \quad k=1,2,\ldots,m.
$$
\end{lemma}

\begin{proof}
One implication is obvious. Assume that $m\geq 2$ and
$(id-\varphi)^m(D)\geq 0$.  Due to the identity
$$
(id-\varphi)^k(D)=\sum_{p=0}^k (-1)^p \left(\begin{matrix}
k\\p\end{matrix}\right) \varphi^p(D),\quad k\in \NN,
$$
and the fact that $\varphi$ is a positive linear map, we deduce that
$x_j:=\left< \varphi^j(id-\varphi)^{m-1}(D)h,h\right>$ is a real
number for any $h\in \cH$ and $j=0,1,\ldots$. Note that, we have
$$
x_j-x_{j+1}=\left< \varphi^j(id-\varphi)^{m}(D)h,h\right>\geq 0.
$$
Therefore,  $\{x_j\}_{j=0}^\infty$ is a decreasing sequence of real
numbers.

On the other hand, using the fact that $\varphi$ is a power bounded
linear map, there exists a constant $M>0$ such that
$\|\varphi^k\|\leq M$ for any $k\in \NN$. Therefore, we have
\begin{equation*}
\begin{split}
\left|\sum_{j=0}^p x_j\right|&= \left|\sum_{j=0}^p\left<
(\varphi^j-\varphi^{j+1})(id-\varphi)^{m-2}(D)h,h\right>\right|\\
&=\left|\left<
(id-\varphi)^{m-2}(D)h,h\right>-\left<\varphi^{p+1}(id-\varphi)^{m-2}(D)h,h\right>\right|\\
&\leq \left|\left<
(id-\varphi^{p+1})(id-\varphi)^{m-2}(D)h,h\right>\right|\\
&\leq (1+ M)\|(id-\varphi)^{m-2}(D)\|\|h\|^2<\infty
\end{split}
\end{equation*}
for any $p=0,1,\ldots$.  Hence, we deduce that $x_j\geq 0$ for any
$j=0,1,\ldots$. In particular, we have
$x_0:=\left<(id-\varphi)^{m-1}(D)h,h\right>\geq 0$ for any $h\in
\cH$. Therefore, $(id-\varphi)^{m-1}(D)\geq 0$. Iterating this
process, one can show that $(id-\varphi)^k(D)\geq 0$  for any
$k=1,2,\ldots,m$. The proof is complete.
\end{proof}

\begin{corollary}\label{phi-cond2} If $\varphi$ is a positive linear map on $B(\cH)$
such that $\varphi(I)\leq I$ and $(id-\varphi)^m(I)\geq 0$ for some
$m\in \NN$, then
$$
0\leq (id-\varphi)^m(I) \leq  (id-\varphi)^{m-1}(I)\leq \cdots\leq
(id-\varphi)(I)\leq I.
$$
\end{corollary}

 Given $m,n\in \{1,2,\ldots\}$ and  a positive regular free holomorphic
  function $f:=\sum_{|\alpha|\geq 1} a_\alpha X_\alpha$, we define  the noncommutative domain
$$
{\bold D}_f^m(\cH):=\left\{X:= (X_1,\ldots, X_n)\in B(\cH)^n: \
 (id-\Phi_{f,X})^k(I)\geq 0 \text{ for } \ 1\leq k\leq m\right\},
$$
where $\Phi_{f,X}:B(\cH)\to B(\cH)$ is defined by
$
\Phi_{f,X}(Y):=\sum_{|\alpha|\geq 1} a_\alpha X_\alpha
YX_\alpha^*$, \  $Y\in B(\cH),
$
and the convergence is in the week operator topology.
 For the next result,   we need to  denote by $(W_1^{(f)},\ldots, W_n^{(f)})$ the weighted left
creation operators $(W_1,\ldots, W_n)$ associated with ${\bf
D}_f^{(m)}$. The notation $(\Lambda_1^{(f)},\ldots,
\Lambda_n^{(f)})$ is now  clear.

\begin{theorem}\label{tilde-f}
Let $(W_1^{(f)},\ldots, W_n^{(f)})$ (resp. $(\Lambda_1^{(f)},\ldots,
\Lambda_n^{(f)})$) be the weighted left (resp. right) creation
operators associated with the noncommutative domain  ${\bf D}_f^m$. Then
the following statements hold:
\begin{enumerate}
\item[(i)] $(W_1^{(f)},\ldots, W_n^{(f)})\in {\bf D}_f^{m}(F^2(H_n));$
\item[(ii)] $(\Lambda_1^{(f)},\ldots,
\Lambda_n^{(f)})\in {\bf D}_{\tilde f}^{m}(F^2(H_n));$
\item[(iii)] $U^* \Lambda_i^{(f)} U=W_i^{(\tilde f)}$,
$i=1,\ldots,n$, where $U\in B(F^2(H_n))$ is the unitary operator
defined by equation $U e_\alpha:= e_{\tilde \alpha}$, $\alpha\in \FF_n^+$;
\item[(iv)] $W_i^{(f)}\Lambda_j^{(f)}=\Lambda_j^{(f)} W_i^{(f)}$ for $i,j=1,\ldots,n$.
\end{enumerate}
\end{theorem}
\begin{proof}
Items (i) and (ii) follow from  Theorem \ref{prop-shif}, Lemma \ref{phi-condition}, and
 relation \eqref{tild-Lamb}.
Using relation \eqref{WbWb} when $f$ is
replaced by $\tilde f$, we obtain
$$
W_i^{(\tilde f)} e_\gamma=\frac{\sqrt{b_{\tilde
\gamma}^{(m)}}}{\sqrt{b_{\tilde\gamma g_i}^{(m)}}} e_{g_i\gamma}.
$$
On the other hand,  due to  relation \eqref{WbWb-r}, we deduce that
$$
U^*\Lambda_i^{(f)} U e_\gamma =U^*\left(\frac{\sqrt{b_{\tilde
\gamma}^{(m)}}}{\sqrt{b_{\tilde\gamma g_i}^{(m)}}} e_{\tilde\gamma
g_i}\right)= \frac{\sqrt{b_{\tilde
\gamma}^{(m)}}}{\sqrt{b_{\tilde\gamma g_i}^{(m)}}} e_{g_i\gamma}.
$$
Therefore,  $U^* \Lambda_i^{(f)} U=W_i^{(\tilde f)}$,
$i=1,\ldots,n$.

Now, using relation \eqref{w-shift}, \eqref{WbWb-r}, we obtain
$$
\Lambda_j W_i^{(f)} e_\alpha= \frac{\sqrt{b_{\alpha}^{(m)}}}{\sqrt{b_{g_i \alpha}^{(m)}}}
\Lambda_j^{(f)}(e_{g_i\alpha})=\frac{\sqrt{b_{\alpha}^{(m)}}}{\sqrt{b_{g_i \alpha g_j}^{(m)}}}
e_{g_i \alpha g_j}
$$
for any $\alpha\in \FF_n^+$ and $i,j=1,\ldots,n$. Similar calculations reveal that $\Lambda_j^{(f)} W_i^{(f)} e_\alpha=
 W_i^{(f)} \Lambda_j^{(f)} e_\alpha$, which proves (iv).
The proof is complete.
\end{proof}

\bigskip

      \section{Noncommutative  Berezin transforms
} \label{Noncommutative}

In this section, we introduce a {\it noncommutative Berezin
transform}
 associated with  each  $n$-tuple of operators $T:=(T_1,\ldots,
T_n)$ in the noncommutative domain $ {\bf D}_f^m(\cH)$,
and present some of its basic properties.

Let $f$ be a positive regular free holomorphic function on a
noncommutative ball $[B(\cH)^n]_\rho$, $\rho>0$, with representation
$f(X_1,\ldots, X_n):=\sum_{|\alpha|\geq 1} a_\alpha X_\alpha$.   Let
$T:=(T_1,\ldots, T_n)\in B(\cH)^n$ be  an $n$-tuple of operators
 such  that the series
 $ \sum\limits_{|\alpha|\geq 1}
a_{\alpha} T_{\alpha}T_{\alpha}^* $ is WOT convergent,  and consider
the bounded linear map $\Phi_{f,T}:B(\cH)\to B(\cH)$,  given by
\begin{equation}
\label{PFT}
\Phi_{f,T}(X):=\sum_{|\alpha|\geq 1} a_\alpha T_\alpha
XT_\alpha^*,\quad X\in B(\cK),
\end{equation}
where the convergence is in the week operator topology. The joint
spectral radius of $T\in {\bf D}_f^m(\cH)$   is defined  by
$$
r_f(T_1,\ldots, T_n):=\lim_{k\to\infty}\|\Phi_{f,T}^k(I)\|^{1/2k}.
$$
We recall that the model $n$-tuple $(\Lambda_1, \ldots, \Lambda_n)$
associated with ${\bf D}_f^m$
was defined in Section 1.
  According to the results of that section, the series   $\sum\limits_{|\alpha|\geq 1} a_{\tilde \alpha}
\Lambda_\alpha \Lambda_\alpha^*$ is SOT convergent and, therefore,
so is  the series $\sum\limits_{|\alpha|\geq 1} a_{\tilde\alpha}
\Lambda_\alpha \otimes T_{\tilde\alpha}^*$. Notice also that
$$
\left\|\sum_{|\alpha|\geq 1} a_{\tilde\alpha} \Lambda_\alpha \otimes
T_{\tilde\alpha}^*\right\|\leq \left\|\sum_{|\alpha|\geq 1}
a_{\tilde\alpha} \Lambda_\alpha
\Lambda_\alpha^*\right\|\left\|\sum_{|\alpha|\geq 1}
a_{\tilde\alpha} T_{\tilde\alpha}T_{\tilde\alpha}^*\right\|.
$$
and
\begin{equation}
\label{RRR}
 \left\|\left(\sum_{|\alpha|\geq 1} a_{\tilde\alpha}
\Lambda_\alpha \otimes T_{\tilde\alpha}^*\right)^k\right\|\leq
\left\|\Phi^k_{\tilde f,\Lambda}(I)\right\|^{1/2}
\left\|\Phi_{f,T}^k(I)\right\|^{1/2},\quad k\in \NN,
\end{equation}
where $\tilde f(X_1,\ldots, X_n):=\sum\limits_{|\alpha|\geq 1}
a_{\tilde \alpha}X_\alpha$ and $\Phi_{\tilde
f,\Lambda}(Y):=\sum\limits_{|\alpha|\geq 1}
a_{\tilde\alpha}\Lambda_\alpha Y \Lambda_\alpha^*$. Hence, we deduce
that
\begin{equation*}
r\left(\sum_{|\alpha|\geq 1} a_{\tilde\alpha} \Lambda_\alpha \otimes
T_{\tilde\alpha}^*\right)\leq r_{\tilde f}(\Lambda_1,\ldots,
\Lambda_n)r_f(T_1,\ldots, T_n),
\end{equation*}
where $r(A)$ denotes the usual spectral radius of an operator $A$.
Due to the results of Section 1, we have $\left\|\Phi_{\tilde
f,\Lambda}(I)\right\|\leq 1$, which implies $r_{\tilde
f}(\Lambda_1,\ldots, \Lambda_n)\leq 1$. Consequently, we have
\begin{equation*}
  r\left(\sum_{|\alpha|\geq 1} a_{\tilde\alpha}
\Lambda_\alpha \otimes T_{\tilde\alpha}^*\right)\leq r_f(T_1,\ldots,
T_n).
\end{equation*}
Consequently, if $r_f(T_1,\ldots, T_n)<1$, then
  the operator
\begin{equation}\label{form-1}
 \left(
I-\sum_{|\alpha|\geq 1} a_{\tilde\alpha} \Lambda_\alpha \otimes
T_{\tilde\alpha}^* \right)^{-1}= \sum_{k=0}^\infty
\left(\sum_{|\alpha|\geq 1} a_{\tilde\alpha} \Lambda_\alpha \otimes
T_{\tilde\alpha}^* \right)^k
\end{equation} is well-defined,
where the convergence is in the operator norm topology.

For each $T:=(T_1,\ldots, T_n)\in {\bf D}_f^m(\cH)$ with
$r_f(T_1,\ldots, T_n)<1$, we introduce the {\it noncommutative Berezin
transform}  at $T$ as the map  ${\bf B}_T:B(F^2(H_n)) \to  B(\cH)$ defined by

\begin{equation}
\label{Berezin}
 \left<{\bf B}_T[g]x,y\right>:=
\left<\left(
I-\sum_{|\alpha|\geq 1} \overline{a}_{\tilde\alpha} \Lambda_\alpha^* \otimes
T_{\tilde\alpha} \right)^{-m}
 (g\otimes \Delta_{T,m,f}^2)
  \left(
I-\sum_{|\alpha|\geq 1} a_{\tilde\alpha} \Lambda_\alpha \otimes
T_{\tilde\alpha}^* \right)^{-m}(1\otimes x), 1\otimes y\right>
\end{equation}
  where $\Delta_{T,m,f}:= [(id-\Phi_{f,T})^m(I)]^{1/2}$ and $x,y\in \cH$.
    We remark that in the particular case when: $n=1$, $m=1$, $f(X)=X$,
     $\cH=\CC$,
 and $T=\lambda\in \DD$, we recover the Berezin transform of a bounded
linear operator on the Hardy space $H^2(\DD)$, i.e.,
$$
{\bf B}_\lambda [g]=(1-|\lambda|^2)\left<g k_\lambda,
k_\lambda\right>,\quad g\in B(H^2(\DD)),
$$
where $k_\lambda(z):=(1-\overline{\lambda} z)^{-1}$ and  $z,
\lambda\in \DD$.

The noncommutative Berezin  transform  will play an important role
in this paper. We will present some of its basic properties in this section.
First, we need a few
preliminary results about positive linear maps on $B(\cH)$.

\begin{lemma}\label{identity}
Let $\varphi:B(\cH)\to B(\cH)$ be a linear map and $k,q\in \NN$.
Then
\begin{equation}\label{ident}
\sum_{p=0}^q\left(\begin{matrix} p+k-1\\k-1
\end{matrix}\right) \varphi^p(id-\varphi)^k=id-\sum_{j=0}^{k-1}
 \left(\begin{matrix} q+j\\ j
\end{matrix}\right) \varphi^{q+1} (id-\varphi)^j.
\end{equation}
\end{lemma}
\begin{proof}
  Since  $\sum_{p=0}^q \varphi^p(id-\varphi)=id-\varphi^{q+1}$,
 equation \eqref{ident} holds for $k=1$. We proceed now by induction
 over $k$. Assume
that \eqref{ident} holds   for $k=m$. Since
$\varphi(id-\varphi)=(id-\varphi)\varphi$ and $$\left(\begin{matrix}
p+m\\ m
\end{matrix}\right)-\left(\begin{matrix}
p+m-1\\ m
\end{matrix}\right)=\left(\begin{matrix}
p+m-1\\ m-1
\end{matrix}\right),
$$
we have
\begin{equation*}
\begin{split}
\sum_{p=0}^q&\left(\begin{matrix} p+m\\m
\end{matrix}\right) \varphi^p(id-\varphi)^{m+1}\\
&= \sum_{p=0}^q\left(\begin{matrix} p+m\\m
\end{matrix}\right) \varphi^p(id-\varphi)^{m}-\sum_{p=0}^q\left(\begin{matrix} p+m\\m
\end{matrix}\right)\varphi^{p+1}
(id-\varphi)^{m}\\
&=(id-\varphi)^m+\sum_{p=1}^q\left[\left(\begin{matrix} p+m\\m
\end{matrix}\right)-\left(\begin{matrix} p+m-1\\m
\end{matrix}\right)\right] \varphi^p(id-\varphi)^{m}-\left(\begin{matrix} q+m\\m
\end{matrix}\right)\varphi^{q+1}(id-\varphi)^{m}\\
&=\sum_{p=0}^q\left(\begin{matrix} p+m-1\\ m-1
\end{matrix}\right)\varphi^p(id-\varphi)^{m}-\left(\begin{matrix} q+m\\m
\end{matrix}\right)\varphi^{q+1}(id-\varphi)^{m}.
\end{split}
\end{equation*}
Using the induction hypothesis, we complete the proof.
\end{proof}

\begin{lemma}\label{ineq-lim}
Let $\varphi:B(\cH)\to B(\cH)$ be a power bounded, positive  linear
map and  let $ D\in B(\cH)$ be a positive operator such that
$(id-\varphi)^m(D)\geq 0$ for some $m\in \NN$.
 Then the following  limit exists for  any \ $h\in \cH$ and $k=0,1,\ldots,
 m-1$, and
$$
\lim_{p\to\infty} p^k\left< \varphi^p(id-\varphi)^k(D)h,h\right>=
\begin{cases} \lim\limits_{p\to\infty}\left< \varphi^p (D)h,h\right> &
\quad \text{ if } \ k=0\\
0 & \quad \text{ if } \ k=1,2,\ldots, m-1.
\end{cases}
$$
\end{lemma}
\begin{proof}
For each $h\in \cH$, $p=0,1,\ldots$, and $r=0,1,\ldots, m$, denote \
$x_p^{(r)}:=\left< \varphi^p(id-\varphi)^r(D)h,h\right>$ and notice
that, due to Lemma \ref{phi-condition}, $x_p^{(r)}\geq 0$. When
$k=0,1,\ldots, m-1$, using the same lemma, we obtain
$$
x_p^{(k)}-x_{p+1}^{(k)}=\left<
\varphi^p(id-\varphi)^{k+1}(D)h,h\right>\geq 0.
$$
Therefore, $\{x_p^{(k)}\}_{p=0}^\infty$ is a decreasing sequence of
positive numbers. In particular, when $k=0$,  we deduce that $
\lim\limits_{p\to\infty}\left< \varphi^p (D)h,h\right>$ exists.

It remains to prove that
\begin{equation}
\label{lim1} \lim_{p\to\infty} p^k\left<
\varphi^p(id-\varphi)^k(D)h,h\right>=0
\end{equation}
for any $k=1,\ldots, m-1$. As an intermediate step, we will also
prove that
\begin{equation}\label{ser1}
\sum_{p=1}^\infty p^{r-1}x_p^{(r)}<\infty
\end{equation}
for $r=1,\ldots,m$. Notice that this relation holds true if $r=1$, due
to the fact that the series
$$\sum_{p=1}^\infty \left<
\varphi^p(id-\varphi)(D)h,h\right>=\left<Dh,h\right>-\lim\limits_{p\to\infty}\left<
\varphi^p (D)h,h\right> $$
 is convergent.
 We proceed  now by induction over $r$. Assume that $1\leq N\leq m-1$ and
 that relation \eqref{ser1} holds for $r=N$, i.e,
 $\sum_{p=1}^\infty p^{N-1}x_p^{(N)}<\infty$.
 We shall prove first that relation \eqref{lim1} holds for $k=N$.
 Due to the Cauchy criterion, we have
 $$
 y_q:=q^{N-1} x_q^{(N)}+(q+1)^{N-1} x_{q+1}^{(N)}+\cdots +(2q-1)^{N-1}
 x_{2q-1}^{(N)}\to 0,\quad \text{ as } \ q\to\infty.
 $$
Since $\{x_q^{(N)}\}_{q=1}^\infty$ is a  decreasing sequence of
positive numbers, we have $q^Nx_{2q-1}^{(N)}\leq y_q$. Now, it is
clear that $(2q-1)^Nx_{2q-1}^{(N)}\to 0$ as $q\to \infty$. On the
other hand,  since $(2q)^Nx_{2q}^{(N)}\leq (2q)^N x_{2q-1}^{(N)}$,
we have $(2q)^Nx_{2q}^{(N)}\to 0$ as $q\to\infty$. Consequently,
relation \eqref{lim1} holds for $k=N$.

Now, we prove that if \eqref{lim1} holds for $k=N$ (where $1\leq
N\leq m-1$) and  relation \eqref{ser1} holds for $r=N$, then
\eqref{ser1} holds also for $r=N+1$. Notice that
\begin{equation*}
\begin{split}
\sum_{p=1}^q  p^Nx_{p}^{(N+1)}&=\sum_{p=1}^q  p^N \left<
\varphi^p(id-\varphi)^{N+1}(D)h,h\right>\\
&=\sum_{r=1}^q r^N x_r^{(N)}-\sum_{p=1}^q p^N x_{p+1}^{(N)}\\
&=x_1^{(N)}+\sum_{p=1}^q \left[(p+1)^N-p^N\right]
x_{p+1}^{(N)}-(q+1)^N x_{q+1}^{(N)}\\
&\leq x_1^{(N)} +N\sum_{p=1}^q (p+1)^{N-1} x_{p+1}^{(N)}-(q+1)^N
x_{q+1}^{(N)}.
\end{split}
\end{equation*}
Using our assumptions, we conclude that \eqref{ser1} holds  for
$r=N+1$. This completes the proof.
\end{proof}

Let $f$ be a positive regular free holomorphic function on a
noncommutative ball $[B(\cH)^n]_\rho$, $\rho>0$.
In what follows we introduce the noncommutative Berezin  kernel associated with
any $n$-tuple  of operators $T:=(T_1,\ldots, T_n)$ in the noncommutative domain ${\bf D}_f^m(\cH)$, and
 present some of its basic properties.

\begin{lemma}\label{Berezin-lemma}
Let $T:=(T_1,\ldots, T_n)\in {\bf D}_f^m(\cH)$ and let
$K_{f,T}^{(m)}:\cH\to F^2(H_n)\otimes
\overline{\Delta_{f,m,T}(\cH)}$  be the map defined by
\begin{equation}
\label{Po-ker}
 K_{f,T}^{(m)}h:=\sum_{\alpha\in \FF_n^+} \sqrt{b^{(m)}_\alpha}
e_\alpha\otimes \Delta_{f,m,T} T_\alpha^* h,\quad h\in \cH,
\end{equation}
where $\Delta_{f,m,T}:=\left[(I- \Phi_{f,T})^m(I) \right]^{1/2}$, the positive map
$\Phi_{f,T}$ is defined by \eqref{PFT} and the coefficients $b_{\alpha}^{(m)}$ are given by
\eqref{b-al}.
Then \begin{enumerate}
\item[(i)] $ {K_{f,T}^{(m)}}^* K_{f,T}^{(m)}=I_\cH-Q_{f,T}$, where
$Q_{f,T}:=\text{\rm SOT-}\lim_{k\to\infty} \Phi^k_{f,T}(I)$;
\item[(ii)] $K_{f,T}^{(m)} T_i^*=(W_i^*\otimes
I_\cH)K_{f,T}^{(m)}$,\ $ i=1,\ldots, n$, where $(W_1,\ldots, W_n)$
is the $n$-tuple of weighted left creation operators
  associated with the noncommutative domain ${\bf D}^m_f$.
\end{enumerate}
\end{lemma}
\begin{proof}
Since
$\Phi_{f,T}(I)\leq I$ and $\Phi_{f,T}(\cdot)$ is a positive linear
map,
 it is easy to see that $\{\Phi_{f,T}^p(I)\}_{p=1}^\infty$ is a decreasing
  sequence of positive operators and, consequently,
$Q_{f,T}:=\text{\rm SOT-}\lim\limits_{p\to\infty} \Phi_{f,T}^p(I)$
exists.
Due to relation \eqref{b-al} and using Lemma \ref{identity}, we
deduce that
\begin{equation*}
\begin{split}
&\left< \sum_{\beta\in \FF_n^+} b_\beta^{(m)} T_\beta
\Delta_{f,m,T}^2T_\beta^*h, h\right>
 =
\left<\Delta_{f,m,T}^2 h,h\right>+ \sum_{m=1}^\infty
\sum_{|\beta|=m}\left<
b_\beta^{(m)} T_\beta \Delta_{f,m,T}^2T_\beta^*h, h\right>\\
& =\left<\Delta_{f,m,T}^2 h,h\right>+ \sum_{m=1}^\infty
\sum_{|\beta|=m}\left<\left(
\sum_{j=1}^{|\beta|}\left(\begin{matrix} j+m-1\\ m-1
\end{matrix}\right)\sum_{{\gamma_1\cdots
\gamma_j=\beta}\atop{|\gamma_1|\geq 1,\ldots, |\gamma_j|\geq 1}}
a_{\gamma_1}\cdots a_{\gamma_j}\right) T_{\gamma_1\cdots \gamma_j}
\Delta_{f,m,T}^2 T_{\gamma_1\cdots \gamma_j}^*h,h\right>\\
&=\left<\Delta_{f,m,T}^2 h,h\right>+\sum_{p=1}^\infty
\left(\begin{matrix} p+m-1\\ m-1
\end{matrix}\right)
\sum_{|\alpha_1|\geq 1,\ldots, |\alpha_p|\geq 1} a_{\alpha_1}\cdots
a_{\alpha_p}T_{\alpha_1}\cdots T_{\alpha_p} \Delta_{f,m,T}^2
T_{\alpha_p}^*\cdots T_{\alpha_1}^*\\
& = \lim_{k\to\infty} \sum_{p=0}^k \left(\begin{matrix} p+m-1\\ m-1
\end{matrix}\right)
\left<\left\{\Phi_{f,T}^p[(I-\Phi_{f,T})^m](I)\right\}h,h\right>\\
& =\|h\|^2 -\lim_{k\to\infty} \sum_{j=0}^{m-1} \left(\begin{matrix}
k+j\\ j
\end{matrix}\right)
\left<\Phi_{f,T}^{k+1}[(I-\Phi_{f,T})^j](I) h,h\right>.
\end{split}
\end{equation*}
for any $h\in \cH$.  Now, applying  Lemma \ref{ineq-lim} to
$\Phi_{f,T}$, we deduce that
\begin{equation}\label{I-Q}
\sum_{\beta\in \FF_n^+} b_\beta^{(m)} T_\beta
\Delta_{f,m,T}^2T_\beta^*=I_\cH-Q_{f,T}.
\end{equation}
Due to the above calculations, we have
$$
\|K_{f,T}^{(m)}h\|^2= \sum_{\beta\in \FF_n^+} b_\beta^{(m)}
\left<T_\beta \Delta_{f,m,T}^2T_\beta^*h,
h\right>=\|h\|^2-\|Q_{f,T}^{1/2}h\|^2
$$
for any $h\in \cH$. Therefore,  $K_{f,T}^{(m)}$ is a contraction and
\begin{equation}
\label{K*K} {K_{f,T}^{(m)}}^* K_{f,T}^{(m)}=I_\cH-Q_{f,T}.
\end{equation}
 On the other hand, one can show  that
\begin{equation}
\label{ker-inter} K_{f,T}^{(m)} T_i^*=(W_i^*\otimes
I_\cH)K_{f,T}^{(m)},\quad i=1,\ldots, n,
\end{equation}
where $(W_1,\ldots, W_n)$ is the $n$-tuple of weighted left creation
operators
  associated with the noncommutative domain ${\bf D}^m_f$.
  Indeed,  notice that, due to relation \eqref{WbWb}, we have
  \begin{equation*}
  W_i^*
e_\alpha =\begin{cases} \frac
{\sqrt{b_\gamma^{(m)}}}{\sqrt{b_{\alpha}^{(m)}}}e_\gamma& \text{ if
}
\alpha=g_i\gamma \\
0& \text{ otherwise.}
\end{cases}
  \end{equation*}
  Hence, we deduce that
  \begin{equation*}
  \begin{split}
  (W_i^*\otimes I_\cH) K_{f,T}^{(m)}h &=\sum_{\alpha\in \FF_n^+} \sqrt{b^{(m)}_\alpha}
W_i^*e_\alpha\otimes \Delta_{f,m,T} T_\alpha^* h\\
&= \sum_{\gamma\in \FF_n^+} \sqrt{b^{(m)}_{g_i\gamma}}
\frac{\sqrt{b^{(m)}_\gamma}}{\sqrt{b^{(m)}_{g_i\gamma}}}
e_\gamma\otimes
\Delta_{f,m,T} T_{g_i\gamma}^* h\\
&=K_{f,T}^{(m)} T_i^*h
  \end{split}
  \end{equation*}
for any $h\in \cH$ and $i=1,\ldots,n$, which proves our assertion.
\end{proof}

We can define now the {\it extended noncommutative Berezin transform}
$\widetilde{\bf B}_T$
 at any $T\in {\bf D}_f^m(\cH)$ by setting
 \begin{equation}
 \label{def-Be2}
 \widetilde{\bf B}_T[g]:= {K_{f,T}^{(m)}}^* (g\otimes I_\cH)K_{f,T}^{(m)},
 \quad g\in B(F^2(H_n)),
 \end{equation}
where the {\it noncommutative Berezin kernel} \  $K_{f,T}^{(m)}:\cH \to F^2(H_n)\otimes
\cH$  is defined
  by
 \begin{equation}
\label{Be-ker}
 K_{f,T}^{(m)}h=\sum_{\alpha\in \FF_n^+} \sqrt{b^{(m)}_\alpha}
e_\alpha\otimes \Delta_{f,m,T} T_\alpha^* h,\quad h\in \cH,
\end{equation}
the defect operator $\Delta_{T,m,f}:= [(id-\Phi_{f,T})^m(I)]^{1/2}$, and
the coefficients $b_\alpha^{(m)}$, $\alpha\in \FF_n^+$, are given by
relation   \eqref{b-al}.

 \begin{proposition}\label{Berezin2}
 The nocommutative Berezin transforms $\widetilde{\bf B}_T$ and
  ${\bf B}_T$ coincide for any $n$-tuple of operators
   $T:=(T_1, \ldots, T_n)\in{\bf D}_f^m(\cH)$ with joint spectral radius
   $r_f(T_1,\ldots, T_n)<1$.
 \end{proposition}
\begin{proof}
 Due to  Lemma \ref{b-alpha} and relation \eqref{form-1},
     the operator
    $\left(
I-\sum_{|\alpha|\geq 1} a_{\tilde\alpha} \Lambda_\alpha \otimes
T_{\tilde\alpha}^* \right)^{-m}$ has the   Fourier representation
    is
     $
     \sum_{\beta\in \FF_n^+} (\Lambda_\beta\otimes b_{\tilde
    \beta}T_{\tilde \beta}^*)
    $.
Consequently, using relations \eqref{WbWb-r} and  \eqref{Be-ker}, we obtain
\begin{equation*}
\begin{split}
 K_{f,T}^{(m)}h&=\sum_{\alpha\in \FF_n^+} \sqrt{b^{(m)}_\alpha}
e_\alpha\otimes \Delta_{f,m,T} T_\alpha^* h\\
&= 1\otimes\Delta_{f,m,T} h+\sum_{|\beta|\geq 1} b_{\tilde
    \beta}^{(m)}\Lambda_\beta (1)\otimes \Delta_{f,m,T}T_{\tilde \beta}^*h \\
    &=
(I_{F^2(H_n)} \otimes \Delta_{T,m,f})\left( I-\sum_{|\alpha|\geq 1}
a_{\tilde\alpha} \Lambda_\alpha \otimes  T_{\tilde\alpha}^*
\right)^{-m}(1\otimes h)
\end{split}
\end{equation*}
for any $h\in \cH$. Taking into account relations \eqref{Berezin} and \eqref{def-Be2},
 we complete the proof.
\end{proof}

Let us recall some definitions concerning completely bounded maps
 on operator spaces.
We identify $M_k(B(\cH))$, the set of
$k\times k$ matrices with entries in $B(\cH)$, with
$B( \cH^{(k)})$, where $\cH^{(k)}$ is the direct sum of $k$ copies
of $\cH$.
If $\cX$ is an operator space, i.e., a closed subspace
of $B(\cH)$, we consider $M_k(\cX)$ as a subspace of $M_k(B(\cH))$
with the induced norm.
Let $\cX, \cY$ be operator spaces and $u:\cX\to \cY$ be a linear map. Define
the map
$u_k:M_k(\cX)\to M_k(\cY)$ by
$$
u_k ([x_{ij}]_{k}):=[u(x_{ij})]_{k}.
$$
We say that $u$ is completely bounded   if
$
\|u\|_{cb}:=\sup_{k\ge1}\|u_k\|<\infty.
$
When  $\|u\|_{cb}\leq1$
(resp. $u_k$ is an isometry for any $k\geq1$) then $u$ is completely
contractive (resp. isometric). We call $u$ completely positive
  if $u_k$ is positive for all $k\geq 1$.
   For more information  on completely bounded maps
  and  the classical von Neumann inequality \cite{vN}, we refer
 to \cite{Pa-book} and \cite{Pi}.

Let $f(X_1,\ldots, X_n):=\sum_{|\alpha|\geq 1} a_\alpha X_\alpha$ be
a positive regular free holomorphic function on $[B(\cH)^n]_\rho$,
$\rho>0$,
 and let $T:=(T_1,\ldots, T_n)$ be an $n$-tuple of operators in the
noncommutative domain $ {\bf D}^m_f(\cH)$.
 Recall that the positive linear map
$\Phi_{f,T}:B(\cH)\to
B(\cH)$  is defined   by
$\Phi_{f,T}(X)=\sum\limits_{|\alpha|\geq 1} a_\alpha T_\alpha
XT_\alpha^*,
$
where the convergence is in the weak operator topology. In the proof of
Lemma \ref{Berezin-lemma}, we saw that $Q_{f,T}:=\text{\rm SOT-}\lim_{k\to\infty}
 \Phi_{f,T}^k(I)$ exists.
We call an $n$-tuple  $T$ {\it pure} (or   of class $C_{\cdot 0}$)  if
~$Q_{f,T}=0$.
We remark that  if $\|\Phi_{f,T}(I)\|<1$, then $T$ is of class
$C_{\cdot 0}$. This is due to the fact that $\|\Phi_{f,T}^k(I)\|\leq
\|\Phi_{f,T}(I)\|^k$.
Note also that, due to Theorem \ref{tilde-f},
 the model  $n$-tuple  $W:=(W_1,\ldots, W_n)$ is in the noncommutative domain
  ${\bf D}^m_f(F^2(H_n))$ and, due to  Theorem \ref{prop-shif},
   it  is of class $C_{\cdot 0}$.

We introduce the  domain  algebra $\cA_n({\bf D}^m_f)$ associated
with the noncommutative domain ${\bf D}^m_f$ to be the norm closure
of all polynomials in the weighted left creation operators
$W_1,\ldots, W_n$ and the identity. Using the weighted right
creation operators
 associated with ${\bf D}^m_f$, one can   define    the corresponding
  domain algebra ${\cR}_n({\bf D}^m_f)$.

\begin{theorem}\label{purecase}
Let $T:=(T_1,\ldots, T_n)$ be a pure  $n$-tuple of operators in the
 noncommutative domain $ {\bf D}_f^m(\cH)$.
 Then the restriction of the noncommutative Berezin transform $\widetilde {\bf B}_T$
 to  $\overline{\text{\rm  span}} \{ W_\alpha W_\beta^*;\
\alpha,\beta\in \FF_n^+\}$ is
    a unital completely contractive linear map
such that
 $$
\widetilde{\bf B}_T[W_\alpha W_\beta^*]=T_\alpha T_\beta^*, \quad
\alpha,\beta\in \FF_n^+.
$$
In particular, the restriction of \ $\widetilde {\bf B}_T$ to the
domain algebra  $\cA_n({\bf D}_f^m)$ is a  completely contractive
homomorphism.
\end{theorem}
\begin{proof}
According to Lemma \ref{Berezin-lemma}, $K_{f,T}^{(m)}$ is an isometry if and
only if $T:=(T_1,\ldots, T_n)\in {\bf D}^m_f(\cH)$ is a pure $n$-tuple. Part (ii), of the same
lemma, and relation \eqref{def-Be2} imply
$$
\widetilde{\bf B}_T[W_\alpha W_\beta^*]={K_{f,T}^{(m)}}^*[W_\alpha W_\beta^*\otimes
I_\cH)K_{f,T}^{(m)}=T_\alpha T_\beta^*, \quad \alpha,\beta\in
\FF_n^+.
$$
Now, one can easily deduce that $\widetilde{\bf B}_T$ is a unital
 completely contractive linear map. This completes the proof.
\end{proof}

We say that an  $n$-tuple of operators $X:=(X_1,\ldots, X_n)\in {\bf
D}_f^m(\cH)$ has the radial property with respect to ${\bf
D}_f^m(\cH)$ if  there exists $\delta\in (0,1)$ such that
$rX:=(rX_1,\ldots,
 rX_n)\in  {\bf D}_f^m(\cH)$ for any  $ r\in (\delta, 1)$.

\begin{proposition}
\label{radial} Any noncommutative domain ${\bf D}_f^m(\cH)$ contains
a ball  \ $[B(\cH)^n]_\gamma$, $\gamma>0$, and, therefore,
$n$-tuples of operators with the radial property.
\end{proposition}
\begin{proof}
Since $f=\sum_{|\alpha|\geq 1} a_\alpha X_\alpha$ is a free
holomorphic function on  a certain ball $[B(\cH)^n]_\delta$,
$\delta>0$, we have $\limsup_{k\to\infty}\left(\sum_{|\alpha|=k}
|a_\alpha|^2\right)^{1/2k}<\infty$. Consequently, there exists  a
constant $M>0$ such that $|a_\alpha|\leq M^k$ for any $\alpha\in
\FF_n^+$ with $|\alpha|=k$. Let $r\in(0,1)$ be such that $Mr<1$ and let
$(X_1,\ldots, X_n)\in [B(\cH)^n]_r$. Then we have
$$
\|\Phi_{f,X}(I)\|\leq \sum_{k=1}^\infty M^k\left\|\sum_{|\alpha|=k}
X_\alpha X_\alpha^*\right\|\leq\sum_{k=1}^\infty M^{k}
r^{2k}=\frac{r^2M}{1-r^2M},
$$
which converges to  zero as $r\to 0$. Since $f^k$, $k=1,\ldots,m$,
is a free holomorphic function with $f^k(0)=0$, a similar result
holds. Therefore, there exists a ball $[B(\cH)^n]_\gamma$,
$\gamma>0$, such that $\|\Phi_{f,X}(I)\|,\ldots,
\|\Phi_{f^m,X}(I)\|$ are as small as needed for any $X\in
[B(\cH)^n]_\gamma$. On the other hand, we have
$\Phi_{f,X}^k(I)=\Phi_{f^k,X}(I)$ and
$$
(id-\Phi_{f,X})^m(I)=I-\sum_{k=1}^m(-1)^{k-1}\left(\begin{matrix}
m\\k
\end{matrix}\right) \Phi_{f^k,X}(I).
$$
Now, it is clear that $(I-\Phi_{f,X})^m(I)\geq 0$ for any $X$ in an
appropriate ball $[B(\cH)^n]_\gamma$, $\gamma>0$. The proof is
complete.
\end{proof}

We remark that one can easily prove that if $p$ is a positive regular
 noncommutative polynomial and $T:=(T_1,\ldots, T_n)\in B(\cH)^n$  is such that
 $(id-\Phi_{p,T})^k\geq cI$ for some $c>0$ and any $1\leq k\leq m$, then
 $(T_1,\ldots, T_n)\in {\bf D}_p^m(\cH)$ has the radial property.

The next result  extends Theorem 3.7 and Theorem 3.8  from
\cite{Po-poisson} to our more general setting. We only sketch the proof.

\begin{theorem}\label{Poisson-C*}
Let $T:=(T_1,\ldots, T_n)$ be an $n$-tuple of operators with
  the radial property in the
 noncommutative domain $ {\bf D}_f^m(\cH)$   and let
  $\cS:=\overline{\text{\rm  span}} \{ W_\alpha W_\beta^*;\
\alpha,\beta\in \FF_n^+\}$.
   Then there is
    a unital completely contractive linear map
$
\Psi_{f,m,T}:\cS \to B(\cH)
$
such that
\begin{equation}\label{Po-trans}
\Psi_{f,m,T}(g)=\lim_{r\to 1} {\bf B}_{rT}[g],\qquad g\in \cS,
\end{equation}
where the limit exists in the norm topology of $B(\cH)$, and
$$
\Psi_{f,m,T}(W_\alpha W_\beta^*)=T_\alpha T_\beta^*, \quad
\alpha,\beta\in \FF_n^+. $$
 In particular, the restriction of
$\Psi_{f,m,T}$ to the domain algebra $\cA_n({\bf D}_f^m)$ is a
completely contractive homomorphism.
 If, in addition,
 $T$ is a  pure $n$-tuple,
then  $$\lim_{r\to 1} {\bf B}_{rT}[g]=\widetilde{\bf B}_T[g],\qquad g\in \cS.$$
\end{theorem}

\begin{proof} Since $0<r<1$,  $(rT_1,\ldots,
 rT_n)\in  {\bf D}_f^m(\cH)$ is a pure $n$-tuple. Indeed, it is enough to see that
  $\Phi_{f,rT}^k(I)\leq r^k \Phi_{f,T}^k(I)\leq r^k I$ for $k\in \NN$.
Therefore, due to relation \eqref{K*K},
  $K_{f,rT}$ is an isometry.  Now,  Lemma
 \ref{Berezin-lemma}
 implies
\begin{equation}
\label{K-r}
{K_{f,rT}^{(m)}}^*[W_\alpha W_\beta^*\otimes
I_\cH)K_{f,rT}^{(m)}=r^{|\alpha|+|\beta|}T_\alpha T_\beta^*, \quad \alpha,\beta\in
\FF_n^+.
\end{equation}
Hence, we deduce that
\begin{equation}
\label{vn1}
\left\|\sum_{\alpha,\beta\in \Lambda} c_{\alpha,\beta} T_\alpha T_\beta^*\right\|
\leq \left\|\sum_{\alpha,\beta\in \Lambda} c_{\alpha,\beta}
 W_\alpha W_\beta^*\right\|
 \end{equation}
 for any
  finite  set $\Lambda\subset \FF_n^+$ and $c_{\alpha,\beta}\in \CC$.
  For each $g\in \cS $, let
$\{q_k(W_i, W_i^*)\}_{k=0}^\infty$ be a sequence of polynomials of the form
$\sum_{\alpha,\beta\in \Lambda} c_{\alpha,\beta}
 W_\alpha W_\beta^*$ which converges to $g$, as $k\to\infty$. Define
 $\Psi_{f,m,T}(g):=\lim_{k\to\infty} q_k(T_i,T_i^*)$. The von Neumann type inequality
 \eqref{vn1} shows that $\Psi_{f,m,T}(g)$ is  well-defined and
 $\|\Psi_{f,m,T}(g)\|\leq \|g\|$. Using the matrix version on \eqref{K-r},
  we deduce that $\Psi_{f,m,T}$ is  a unital completely contractive linear map.
  To prove the second part of the theorem, one has to use the relation
  $${\bf B}_{rT}[g]={K_{f,rT}^{(m)}}^*(g\otimes I_\cH)K_{f,rT}^{(m)},\qquad g\in \cS,
  $$
  and standard approximation arguments (see \cite{Po-poisson}).
\end{proof}

We say that a   noncommutative domain ${\bf D}_f^m $ has the radial
property if each  $n$-tuple $X\in {\bf D}_f^m(\cH)$  has the radial
property, where $\cH$ is a separable infinite dimensional Hilbert
space. Notice that, if $m=1$, then  the noncommutative domain ${\bf
D}_f^1$ has always the radial property. When $m\geq 1$, we have the
following   class of noncommutative domains with the radial
property.

\begin{example}\label{ex-radial}
If $p(X_1,\ldots, X_n):=a_1X_1+\cdots +a_n X_n$, $a_i>0$, then the
noncommutative domain ${\bf D}_p^m(\cH)$, $m=1,2\ldots$,  has the
radial property. Indeed,  let $X:=(X_1,\ldots, X_n)\in {\bf
D}_p^m(\cH)$, $0<r\leq 1$,  and note that
\begin{equation*}
\begin{split}
(id-\Phi_{p,rX})^k(I)&=\left[(id-\Phi_{p,X})+(1-r)\Phi_{p,X}\right]^k(I)\\
&= \sum_{j=0}^k \left(\begin{matrix} k\\j
\end{matrix}\right) (1-r)^{k-j}\Phi_{p,X}^{k-j}(id-\Phi_{p,X})^j(I)
\end{split}
\end{equation*}
for any $k=1,\ldots,m$. By Corollary \ref{phi-cond2}, we have
$(id-\Phi_{p,X})^j(I)\geq 0$ for $j=1,\ldots,m$.  Now, using the
fact that  $\Phi_{p,X}^j$ is a positive linear map, we deduce that
$(id-\Phi_{p,rX})^k(I)\geq 0$ for $j=1,\ldots,m$ and $r\in (0,1]$,
which proves our assertion.
\end{example}

 Assume that $p$ is a regular positive noncommutative polynomial and
${\bf D}^m_p$ is a noncommutative domain with the radial property.
Under these conditions, one  can prove the following.

\begin{corollary}
An $n$-tuple of operators $(T_1,\ldots, T_n)\in B(\cH)^n$  is in the
noncommutative domain ${\bf D}^m_p(\cH)$ if and only if  there
exists a completely positive linear map $\Psi:C^*(W_1,\ldots,
W_n)\to B(\cH)$ such that  $\Psi(W_\alpha W_\beta^*)=T_\alpha
T_\beta^*$, $\alpha,\beta\in \FF_n^+$. In particular, the result
holds if $p=a_1X_1+\cdots + a_n X_n$ with $ a_i>0$.
\end{corollary}
\begin{proof}
The direct implication is due to Theorem \ref{Poisson-C*} and
Arveson's extension theorem \cite{Arv-acta}. For the converse, use
  Theorem \ref{tilde-f}, and  notice that
$ \Psi\left[ (I-\Phi_{p,W})^k(I)\right]=(I-\Phi_{f,p})^k(I)$ for
$k=1,\ldots,m. $
\end{proof}

\bigskip

\section{The Hardy algebra $F_n^\infty({\bf D}^m_f)$ and a
functional calculus}  \label{Hardy}

 In this section, we introduce the Hardy algebra
$F_n^\infty({\bf D}^m_f)$ (resp. $R_n^\infty({\bf D}^m_f)$)
associated with the noncommutative domain ${\bf D}^m_f$  and present
some basic properties.  The main result is an $F_n^\infty({\bf
D}^m_f)$- functional calculus for completely noncoisometric
$n$-tuples of operators in the noncommutative domain ${\bf D}^m_f$.

Let $f$ be a positive regular free holomorphic function on a
noncommutative ball $[B(\cH)^n]_\rho$, $\rho>0$, with representation
$f(X_1,\ldots, X_n):=\sum_{|\alpha|\geq 1} a_\alpha X_\alpha$. As preliminaries, we need
 some  inequalities concerning the coefficients
$b_\alpha^{(m)}$ associated with $f$ (see Section 1).
According  to  Lemma \ref{b-alpha},   if   $ |\alpha|\geq 1$ and
$|\beta|\geq 1$, then
 we have
\begin{equation*}
b_\alpha^{(m)} b_\beta^{(m)}= \sum_{j=1}^{|\alpha|}
 \sum_{k=1}^{|\beta|}
\left(\begin{matrix} j+m-1\\m-1
\end{matrix}\right)
\left(\begin{matrix} k+m-1\\m-1
\end{matrix}\right)
 \left[
\sum_{{\gamma_1\cdots \gamma_j=\alpha }\atop {|\gamma_1|\geq
1,\ldots, |\gamma_j|\geq 1}} \sum_{{\sigma_1\cdots \sigma_k=\beta
}\atop {|\sigma_1|\geq 1,\ldots, |\sigma_k|\geq 1}}
a_{\gamma_1}\cdots a_{\gamma_j} a_{\sigma_1}\cdots
a_{\sigma_k}\right]
\end{equation*}
  and

\begin{equation*}
b_{\alpha \beta}^{(m)}= \sum_{p=1}^{|\alpha|+|\beta|}
\left(\begin{matrix} p+m-1\\m-1
\end{matrix}\right)
\left[ \sum_{{\epsilon_1\cdots \epsilon_p=\alpha\beta }\atop
{|\epsilon_1|\geq 1,\ldots, |\epsilon_p|\geq 1}}
a_{\epsilon_1}\cdots a_{\epsilon_p} \right].
\end{equation*}

Note that, for any $j=1,\ldots, |\alpha|$ and $k=1,\ldots,|\beta|$,
$$
\left(\begin{matrix} j+m-1\\m-1
\end{matrix}\right)
\left(\begin{matrix} k+m-1\\m-1
\end{matrix}\right)
\leq M_{|\beta|, m} \left(\begin{matrix} j+k+m-1\\m-1
\end{matrix}\right),
$$
where $M_{|\beta|, m}:=\left(\begin{matrix} |\beta|+m-1\\m-1
\end{matrix}\right)$.
A closer look at the above-mentioned equalities reveals that

\begin{equation}
\label{Mb} b_\alpha^{(m)} b_\beta^{(m)}\leq M_{|\beta|, m} b_{\alpha
\beta}^{(m)}, \quad \alpha\in \FF_n^+.
\end{equation}
Similarly,  we obtain
$$
 b_\alpha^{(m)} b_\beta^{(m)}\leq M_{|\alpha|, m} b_{\alpha \beta}^{(m)},\quad
\beta \in \FF_n^+.
$$

Let $\varphi(W_1,\ldots, W_n)=\sum\limits_{\beta\in \FF_n^+} c_\beta
W_\beta $ be a formal sum with the property  that $\sum_{\beta\in
\FF_n^+} |c_\beta|^2 \frac{1}{b_\beta^{(m)}}<\infty$, where the
coefficients $b_\beta$, $\beta\in \FF_n^+$, are given by relation
 \eqref{b-al}.
Using relations   \eqref{WbWb} and \eqref{Mb}, one can see that
$\sum\limits_{\beta\in \FF_n^+} c_\beta W_\beta (p)\in F^2(H_n)$ for
any $p\in \cP$, where $\cP$ is the set of all polynomial in
$F^2(H_n)$. Indeed, for each $\gamma\in \FF_n^+$,  we have
$\sum\limits_{\beta\in \FF_n^+} c_\beta W_\beta(e_\gamma)=
\sum\limits_{\beta\in \FF_n^+}
c_\beta\sqrt{\frac{b_\gamma^{(m)}}{b_{\beta\gamma}^{(m)}}} e_{\beta
\gamma}$ and, due to inequality \eqref{Mb}, we deduce that
$$
\sum_{\beta\in \FF_n^+}|c_\beta|^2
\frac{b_\gamma^{(m)}}{b_{\beta\gamma}^{(m)}} \leq M_{|\gamma|,
m}\sum_{\beta\in \FF_n^+} |c_\beta|^2
\frac{1}{b_\beta^{(m)}}<\infty.
$$
If
$$
\sup_{p\in\cP, \|p\|\leq 1} \left\|\sum\limits_{\beta\in \FF_n^+}
c_\beta W_\beta (p)\right\|<\infty,
$$
then there is a unique bounded operator acting on $F^2(H_n)$, which
we denote by $\varphi(W_1,\ldots, W_n)$, such that
$$
\varphi(W_1,\ldots, W_n)p=\sum\limits_{\beta\in \FF_n^+} c_\beta
W_\beta (p)\quad \text{ for any } \ p\in \cP.
$$
The set of all operators $\varphi(W_1,\ldots, W_n)\in B(F^2(H_n))$
satisfying the above-mentioned properties is denoted by
$F_n^\infty({\bf D}^m_f)$. When $f=X_1+\cdots +X_n$ and $m=1$, $F_n^\infty({\bf D}^m_f)$
coincides with the noncommutative analytic Toeplitz algebra
$F_n^\infty$, which was introduced in \cite{Po-von} in connection
with a noncommutative multivariable von Neumann inequality. As in
this particular case, one can prove that $F_n^\infty({\bf D}^m_f)$
is a Banach algebra, which we call Hardy algebra associated with the
noncommutative domain ${\bf D}^m_f$.

 In a similar manner,  using the weighted right creation
operators $(\Lambda_1,\ldots, \Lambda_n)$
 associated with ${\bf D}^m_f$, one can   define    the corresponding
     the Hardy algebra $R_n^\infty({\bf D}^m_f)$.
More precisely, if $g(\Lambda_1,\ldots,
\Lambda_n)=\sum\limits_{\beta\in \FF_n^+} c_{\tilde\beta
}\Lambda_\beta $ is a formal sum with the property  that
$\sum_{\beta\in \FF_n^+} |c_\beta|^2
\frac{1}{b_\beta^{(m)}}<\infty$, where the
 coefficients $b_\alpha^{(m)}$, $\alpha\in \FF_n^+$, are given by relation
 \eqref{b-al}, and such that
$$
\sup_{p\in\cP, \|p\|\leq 1} \left\|\sum\limits_{\beta\in \FF_n^+}
 c_{\tilde\beta} \Lambda_\beta (p)\right\|<\infty,
$$
then there is a unique bounded operator on $F^2(H_n)$, which we
 denote by $g(\Lambda_1,\ldots, \Lambda_n)$, such that
$$
g(\Lambda_1,\ldots, \Lambda_n)p=\sum\limits_{\beta\in \FF_n^+}
c_{\tilde\beta} \Lambda_\beta (p)\quad \text{ for any } \ p\in \cP.
$$
The set of all operators $g(\Lambda_1,\ldots, \Lambda_n)\in
B(F^2(H_n))$
satisfying the above-mentioned properties is denoted by $R_n^\infty({\bf D}^m_f)$.

\begin{proposition}\label{tilde-f2}
Let $f$ be a positive regular free holomorphic function on a
noncommutative ball $[B(\cH)^n]_\rho$, $\rho>0$, and let ${\bf D}^m_f$ be the associated
noncommutative domain.
 Then the following statements hold:
\begin{enumerate}
\item[(i)] $F_n^\infty({\bf D}^m_f)'=U^*(F_n^\infty({\bf D}^m_{\tilde
f}))U=R_n^\infty({\bf D}^m_f)$, where $'$ stands for the commutant
and $U\in B(F^2(H_n))$ is the unitary operator defined by ~$U
e_\alpha=e_{\tilde\alpha}$, $\alpha\in \FF_n^+$;
\item[(ii)] $F_n^\infty({\bf D}^m_f)''=F_n^\infty({\bf D}^m_f)$ and
$R_n^\infty({\bf D}^m_f)''=R_n^\infty({\bf D}^m_f)$.
\end{enumerate}
\end{proposition}
\begin{proof}
Let ~$(W_1^{(f)},\ldots, W_n^{(f)}) ($resp.
$(\Lambda_1^{(f)},\ldots, \Lambda_n^{(f)}))$ be the weighted left
(resp. right) creation operators associated with the noncommutative
domain  ${\bf D}^m_f$.
Due to Theorem \ref{tilde-f}, part (iii), we have $U^*(F_n^\infty({\bf
D}^m_{\tilde f}))U=R_n^\infty({\bf D}^m_f)$. On the other hand, since $W_i^{(f)}
\Lambda_j^{(f)}=\Lambda_j^{(f)} W_i^{(f)}$ for any $i,j=1,\ldots,n$,
it is clear that $R_n^\infty({\bf D}^m_f)\subseteq F_n^\infty({\bf
D}^m_f)'$. To prove the reverse inclusion, let $A\in F_n^\infty({\bf
D}^m_f)'$. Since $A(1)\in F^2(H_n)$, we have $A(1)=\sum_{\beta\in
\FF_n^+} c_{\tilde \beta}\frac{1}{\sqrt{b_{\tilde \beta}^{(m)}}}
e_{\tilde \beta}$ for some coefficients $\{c_\beta\}_{\beta\in\FF_n^+}$ with
$\sum_{\beta\in\FF_n^+} |c_\beta|^2 \frac{1}{b_\beta^{(m)}}<\infty$.
On the other hand, since $ AW_i^{(f)}=W_i^{(f)}A$ for
$i=1,\ldots,n$, relations \eqref{WbWb} and \eqref{WbWb-r} imply
\begin{equation*}
\begin{split}
Ae_\alpha
&=\sqrt{b_\alpha^{(m)}}AW_\alpha(1)=\sqrt{b_\alpha^{(m)}}W_\alpha
A(1)\\
&=\sum_{\beta\in \FF_n^+} c_{\tilde \beta}
\frac{\sqrt{b^{(m)}_\alpha}}{\sqrt{b^{(m)}_{\alpha \tilde \beta}}}
e_{\alpha \tilde \beta} =\sum_{\beta\in \FF_n^+} c_{\tilde \beta}
\Lambda_\beta(e_\alpha)
\end{split}
\end{equation*}
for any $\alpha\in \FF_n^+$.
Therefore, $A(q)=\sum_{\beta\in \FF_n} c_{\tilde \beta}
\Lambda_\beta(q)$ for any polynomial $q$ in in the full Fock space
$F^2(H_n)$. Since $A$ is a bounded operator, $g(\Lambda_1,\ldots,
\Lambda_n):=\sum_{\beta\in \FF_n} c_{\tilde \beta} \Lambda_\beta$ is
in $R_n^\infty({\bf D}^m_f)$ and $A=g(\Lambda_1,\ldots, \Lambda_n)$.
Therefore, $R_n^\infty({\bf D}^m_f)= F_n^\infty({\bf D}^m_f)'$. The
item (ii) follows easily applying part (i). This completes the
proof.
\end{proof}

An obvious consequence of Proposition \ref{tilde-f2} is that
$F_n^\infty({\bf D}^m_f)$  is WOT-closed (resp. $w^*$-closed) in
$B(F^2(H_n))$.

Let $Q_k$, $k\geq 0$,  be the orthogonal projection of $F^2(H_n)$ on the the
subspace  $\text{\rm span}\,\{e_\alpha: \ |\alpha|=k\}$. For each
integer $j$, define the completely contractive projection
$\Phi_j:B(F^2(H_n))\to B(F^2(H_n))$ by
$$
\Phi_j(A):=\sum_{k\geq\max\{0,-j\}} Q_kAQ_{k+j}.
$$
According to Lemma 1.1 from \cite{DP}, the Cesaro operators on
$B(F^2(H_n))$ defined by
$$
\Sigma_k(A):=\sum_{|j|<k}\left(1-\frac{|j|}{k}\right)
\Phi_j(A),\quad k\geq 1,
$$
 are completely contractive and $\Sigma_k(A)$ converges to $A$ in
 the strong operator topology.
 Now, let $A\in F_n^\infty({\bf D}^m_f)$ have the Fourier
 representation $\sum_{\alpha\in \FF_n^+} a_\alpha W_\alpha$.  Due to the
 definition of the weighted left  creation operators (see
 \eqref{w-shift}), one can check that
 $$
 Q_{k+j} AQ_j=\left(\sum_{|\alpha|=k} a_\alpha W_\alpha \right)
 Q_j,\quad k\geq 0, j\geq 0,
 $$
 and $Q_j A Q_{k+j}=0$ if $k\geq 1$  and  $j\geq 0$. Therefore,
 $$
\Sigma_k(A)=\sum_{|\alpha|\leq
k-1}\left(1-\frac{|\alpha|}{k}\right)a_\alpha W_\alpha
$$
converges to $A$, as $ k\to \infty$, in the strong operator
topology. Therefore,  we have proved the following result.
\begin{proposition}\label{density-pol}
 The algebra $F_n^\infty({\bf D}^m_f)$ is the sequential  SOT-(resp.~WOT-, $w^*$-) closure of all polynomials in $W_1,\ldots, W_n$, and
the identity.
\end{proposition}

Now, we have all the ingredients to extend the corresponding results
from   \cite{DP1} and \cite{Po-domains},  to our more general
setting.  More precisely, one can similarly prove that
 the following statements hold:
\begin{enumerate}
\item[(i)]The Hardy algebra $F_n^\infty({\bf D}^m_f)$ is inverse closed.
\item[(ii)] The only normal elements in $F_n^\infty({\bf D}^m_f)$ are the
scalars.
\item[(iii)] Every element  $A\in F_n^\infty({\bf D}^m_f)$
has
 its spectrum
$\sigma (A)\neq \{0\}$ and it is  injective.
\item[(iv)] The algebra $F_n^\infty({\bf D}^m_f)$  contains no non-trivial
idempotents and no non-zero quasinilpotent elements.
\item[(v)] The algebra $F_n^\infty({\bf D}^m_f)$ is semisimple.
\item[(vi)] If  $A\in F_n^\infty({\bf D}^m_f)$, $n\geq 2$, then
$\sigma (A)=\sigma_e (A)$.
\end{enumerate}

We recall that an $n$-tuple  $(T_1,\ldots, T_n)\in {\bf
 D}^m_f(\cH)$ has the radial property with respect to ${\bf
 D}^m_f(\cH)$ if there exists a constant $\delta\in (0,1)$ such that
  $(rT_1,\ldots, rT_n)\in {\bf
 D}^m_f(\cH)$ for any $r\in (\delta, 1)$.

\begin{lemma}
\label{gKKg} Let  $(T_1,\ldots, T_n)$be an $n$-tuple of operators
with  the radial property in the noncommutative domain ${\bf
 D}^m_f(\cH)$. Then
 \begin{equation}
\label{frf} g(rT_1,\ldots, rT_n) {K_{f,T}^{(m)}}^*
={K_{f,T}^{(m)}}^*(g(rW_1,\ldots, rW_n )\otimes I_\cH) \qquad \text{
for any } \ r\in (\delta,1)
\end{equation}
and $g(W_1,\ldots,W_n)=\sum_{\beta\in \FF_n^+} c_\beta W_\beta$ in $F_n^\infty({\bf D}_f^m)$,
where $ g(rT_1,\ldots, rT_n):=\sum_{k=0}^\infty \sum_{|\alpha|=k} c_\alpha r^{|\alpha|}
T_\alpha$,
with the convergence in the operator norm topology.
\end{lemma}
\begin{proof}
According to  relations \eqref{WW*} and   \eqref{Mb},
the operators $\{W_\beta\}_{|\beta|=k}$ have orthogonal ranges and

$$\|W_\beta x\|\leq \frac{1}{\sqrt{b^{(m)}_\beta}} M_{|\beta|,m}\|x\|,
\qquad x\in F^2(H_n),
$$
where $M_{|\beta|, m}:=\left(\begin{matrix} |\beta|+m-1\\m-1
\end{matrix}\right)$.
Consequently,  we deduce that
\begin{equation}
\label{ine-k} \left\|\sum\limits_{|\beta|=k} b_\beta W_\beta
W_\beta^*\right\|\leq \left(\begin{matrix} k+m-1\\m-1
\end{matrix}\right)\quad \text{  for any } \quad k=0,1,\ldots.
\end{equation}
Since $g(W_1,\ldots,W_n) \in F_n^\infty({\bf D}_f^m)$, we have
  $\sum\limits_{\beta\in \FF_n^+}
|c_\beta|^2\frac{1}{b_\beta}<\infty$. Hence and using \eqref{ine-k},
 we deduce that, for $0<t<1$,
\begin{equation*}
 \begin{split}
 \sum_{k=0}^\infty t^k \left\|\sum_{|\beta|=k} c_\beta W_\beta\right\|
 &\leq
 \sum_{k=0}^\infty t^k \left(\sum_{|\beta|=k} |c_\beta|^2\frac{1}
 {b^{(m)}_\beta}\right)^{1/2}\left\|\sum_{|\beta|=k} b^{(m)}_\beta
  W_\beta W_\beta^*\right\|^{1/2}  \\
  &\leq
  \sum_{k=0}^\infty\left(\sum_{|\beta|=k} |c_\beta|^2\frac{1}
 {b^{(m)}_\beta}\right)^{1/2} t^k \left(\begin{matrix}
k+m-1\\m-1 \end{matrix}\right)^{1/2}  \\
 &\leq
    \left(\sum_{\beta\in \FF_n^+} |c_\beta|^2\frac{1}{b_\beta}
   \right)^{1/2}\left(\sum_{k=0}^\infty t^{2k}\left(\begin{matrix}
k+m-1\\m-1 \end{matrix}\right)\right)^{1/2} < \infty,
 \end{split}
 \end{equation*}
which proves that
 \begin{equation}
\label{limm1}
g(tW_1,\ldots, tW_n):=
 \lim_{k\to \infty}\sum_{p=0}^k\sum_{|\alpha|=p}t^{|\alpha|} c_\alpha W_\alpha
 \end{equation}
 is in the  noncommutative domain  algebra
$\cA_n({\bf D}^m_f)$,
 where the convergence is in the operator norm.
Consequently,  Theorem \ref{Poisson-C*} implies that
$g(rT_1,\ldots, rT_n):=\sum\limits_{k=0}^\infty \sum\limits_{|\alpha|=k} c_\alpha r^{|\alpha|}
T_\alpha $ is convergent in the operator norm topology.
 On the other hand, due to Lemma \ref{Berezin-lemma}, we have
  $T_i {K_{f,T}^{(m)}}^*={K_{f,T}^{(m)}}^*(W_i\otimes I_\cH)$ for any
$i=1,\ldots, n$. Now, one can  deduce \eqref{frf}. This completes the proof.
\end{proof}

In what follows we show that the restriction of the noncommutative
Berezin transform to the Hardy algebra $F_n^\infty({\bf D}_f^m)$
provides a functional calculus  associated with each   pure
$n$-tuple of operators in the noncommutative domain ${\bf
D}_f^m(\cH)$.

\begin{theorem}
\label{funct-calc} Let $T:=(T_1,\ldots, T_n)$ be a pure
 $n$-tuple of operators in the noncommutative domain ${\bf
 D}^m_f(\cH)$ and define the map
$$\Psi_T:F_n^\infty({\bf D}^m_f)\to B(\cH)\quad \text{ by} \quad
 \Psi_T(g):= \widetilde{\bf B}_T[g],$$
 where $\widetilde{\bf B}_T$ is the noncommutative Berezin transform  at $T\in {\bf
 D}_f^m(\cH)$.
 Then
\begin{enumerate}
\item[(i)] $\Psi_T$ is    WOT-continuous (resp.
SOT-continuous)  on bounded sets;
\item[(ii)]
$\Psi_T$ is a unital completely contractive homomorphism and
$\Psi_T(W_\alpha)=T_\alpha$ for $\alpha\in \FF_n^+$.
\end{enumerate}
If, in addition, the universal model $(W_1,\ldots, W_n)$ has the
radial property with respect to ${\bf
 D}^m_f(F^2(H_n))$,
 then
 $$
\Psi_{T}(g)=\text{\rm SOT-}\lim_{r\to 1} g(rT_1,\ldots,rT_n)
$$
for any\  $g:=\sum\limits_{\beta\in \FF_n^+}
 c_\beta W_\beta$  in $F_n^\infty({\bf D}^{m}_f)$, where
 $g(rT_1,\ldots, rT_n):=\sum_{k=0}^\infty \sum_{|\alpha|=k} c_\alpha
 r^{|\alpha|} T_\alpha$ and the convergence  is in the operator norm
 topology.
 \end{theorem}

\begin{proof}
According  to Section 2 (see relation \eqref{ker-inter}), we have
\begin{equation}\label{def-be}
\Psi_T(g)={K_{f,T}^{(m)}}^* (g\otimes I) K_{f,T}^{(m)}, \qquad g\in
F_n^\infty({\bf D}_f^m),
\end{equation}
 where the noncommutative Berezin kernel
$K_{f,T}^{(m)}$ is given by relation \eqref{Be-ker}. Using  standard
facts in functional analysis, we deduce part (i).

 Now, we prove part (ii). Since $T$ is a pure $n$-tuple of operators,
 by Lemma \ref{Berezin-lemma},
    $K^{(m)}_{f,T}$ is an isometry. Consequently, relation
    \eqref{def-be}
    implies
  $$
   \left\|\left[\Psi_T(g_{ij}) \right]_{k}\right\|
   \leq
\left\|\left[g_{ij} \right]_{k}\right\|
$$
for any operator-valued matrix $\left[g_{ij} \right]_{k}$ in
$M_k(F_n^\infty({\bf D}^{m}_f))$, which proves that $\Psi_T$ is a
unital completely contractive  linear map. Due to Theorem
\ref{purecase}, $\Psi_T$ is a homomorphism  on polynomials in
$F_n^\infty({\bf D}^m_f)$.  By Proposition \ref{density-pol}, the
polynomials in $W_1,\ldots, W_n$ and the identity  are
 sequentially
 WOT-dense in $F_n^\infty({\bf D}^m_f)$. On the other hand, due to part (i),  $\Psi_{T}$  is
 WOT- continuous on bounded sets. Now, one can use the principle of
 uniform boundedness  to  deduce that $\Psi_T$  is also a homomorphism on
 $F_n^\infty({\bf D}^m_f)$.

Now, we prove the last part of this theorem. Assume that the model
$n$-tuple  $(W_1,\ldots, W_n)$ has the radial property with respect
to ${\bf
 D}^m_f(F^2(H_n))$. First, we show that
\begin{equation}\label{c0sot}
g(W_1,\ldots,W_n)=\text{\rm SOT-}\lim_{t\to 1} g(tW_1,\ldots, tW_n)
\end{equation}
for any   ~$g(W_1,\ldots, W_n):=\sum\limits_{\beta\in \FF_n^+}
 c_\beta W_\beta\in F_n^\infty({\bf D}_f)$.
 According to Lemma \ref {gKKg},
 \begin{equation}
\label{limm}
 g(tW_1,\ldots, tW_n):=\lim_{k\to \infty}\sum_{k=0}^k\sum_{|\alpha|=p}t^{|\alpha|}
  c_\alpha W_\alpha
\end{equation}
 is in the  noncommutative domain  algebra
$\cA_n({\bf D}^m_f)$,
 where the convergence is in the operator norm topology.
 Fix now  $\gamma, \sigma, \epsilon\in \FF_n^+$ and consider the polynomial
$p(W_1,\ldots,W_n):=\sum\limits_{\beta\in \FF_n^+, |\beta|\leq
|\gamma|} c_\beta W_\beta$. Since $W_\beta^* e_\gamma =0$ for any
$\beta\in \FF_n^+$ with $|\beta|>|\gamma|$,  we have
$$
g(rW_1,\ldots, rW_n)^* e_\alpha =p(rW_1,\ldots, rW_n)^* e_\alpha$$
for any $\alpha\in \FF_n^+$ with $|\alpha|\leq |\gamma|$ and any
$r\in [0,1]$. On the other hand,  since $rW:=(rW_1,\ldots, rW_n)\in
 {\bf D}_f^m(F^2(H_n))$
for $r\in(\delta, 1)$, Lemma \ref{Berezin-lemma} implies
$$
K_{f,rW}^{(m)}p(rW_1,\ldots, rW_n)^*=[p(W_1,\ldots, W_n)^*\otimes
I_{F^2(H_n)}]K_{f,rW}^{(m)}
$$
for any $r\in (\delta,1)$. Using all these facts,  careful
calculations reveal that
\begin{equation*}
\begin{split}
\left<K_{f,rW}^{(m)}g(rW_1,\right.&\left.\ldots, rW_n)^*e_\gamma,
e_\sigma\otimes e_\epsilon\right> \\
&=\left<K^{(m)}_{f,rW}p(rW_1,\ldots, rW_n)^*e_\gamma,
e_\sigma\otimes e_\epsilon\right>\\
&=\left<[(p(W_1,\ldots, W_n)^*\otimes
I_{F^2(H_n)})]K_{f,rW}^{(m)}e_\gamma,
e_\sigma\otimes e_\epsilon\right>\\
&=\sum_{\beta\in \FF_n^+} r^{|\beta|}
\sqrt{b_\beta^{(m)}}\left<p(W_1,\ldots, W_n)^*
e_\beta,e_\sigma\right>
\left< W_\beta^*e_\gamma,  \Delta_{f,rW} e_\epsilon\right>\\
&=\sum_{\beta\in \FF_n^+} r^{|\beta|}
\sqrt{b_\beta^{(m)}}\left<g(W_1,\ldots, W_n)^*
e_\beta,e_\sigma\right>
\left< W_\beta^*e_\gamma,  \Delta_{f,rW} e_\epsilon\right>\\
&= \left<[g(W_1,\ldots, W_n)^*\otimes I_{F^2(H_n)}] K_{f,rW}^{(m)}
e_\gamma, e_\sigma\otimes e_\epsilon\right>
\end{split}
\end{equation*}
for any $r\in (\delta,1)$ and $\gamma, \sigma,\epsilon\in \FF_n^+$.
Hence, since $g(rW_1,\ldots, rW_n)$ and  $g(W_1,\ldots, W_n)$  are
bounded operators, we deduce that
$$
K_{f,rW}^{(m)}g(rW_1,\ldots, rW_n)^*=[g(W_1,\ldots, W_n)^*\otimes
I_{F^2(H_n)}]K_{f,rW}^{(m)}.
$$
  Since  the $n$-tuple $rW:=(rW_1,\ldots, rW_n)\in
{\bf D}^{(m)}_f(F^2(H_n))$ is  pure, the Berezin  kernel
$K^{(m)}_{f,rW}$ is an isometry and, therefore, the equality above
implies
\begin{equation}
\label{gr} \|g(rW_1,\ldots, rW_n)\|\leq \|g(W_1,\ldots, W_n)\|\quad
\text{ for any } r\in (\gamma,1).
\end{equation}
 Hence, and due to the fact that
$g(W_1,\ldots,W_n)e_\alpha= \lim\limits_{r\to 1}g(rW_1,\ldots,
rW_n)e_\alpha$ for any  $\alpha\in \FF_n^+$, an approximation
argument  implies relation \eqref{c0sot}.

According to Lemma \ref{gKKg}, we have
\begin{equation}
\label{gKg}
 g(rT_1,\ldots, rT_n) {K_{f,T}^{(m)}}^*
={K_{f,T}^{(m)}}^*(g(rW_1,\ldots, rW_n )\otimes I_\cH) \qquad \text{
for any } \ r\in (\delta,1).
\end{equation}
On the other hand, since the map $Y\mapsto Y\otimes I_\cH$ is
SOT-continuous on bounded sets, relations \eqref{c0sot} and
\eqref{gr} imply that
\begin{equation}
\label{SOT-lim}
\text{\rm SOT-}\lim_{r\to 1}[g(rW_1,\ldots, rW_n )\otimes
I_\cH]=g(W_1,\ldots, W_n )\otimes I_\cH.
\end{equation}
 Hence,  using relation
\eqref{gKg} and that $K_{f,T}^{(m)}$ is an isometry, we deduce that
\begin{equation}
\label{gKKg2} \text{\rm SOT-}\lim_{r\to 1}g(rT_1,\ldots,
rT_n)={K_{f,T}^{(m)}}^*[g(W_1,\ldots, W_n)\otimes
I_\cH]K_{f,T}^{(m)}=\widetilde{\bf B}_T[g].
\end{equation}
This completes the proof.
\end{proof}

We need now  the following technical result concerning the Berezin transform and the radial
property.

\begin{lemma}\label{BB} If $T:=(T_1,\ldots, T_n)\in {\bf
 D}^m_f(\cH)$   and the universal model $(W_1,\ldots, W_n)$ have  the radial property,
   then  there exists  $\delta\in(0,1)$ such that the
 noncommutative
  Berezin
 kernel satisfies the
  relation
\begin{equation}\label{Kfrt}
 {K_{f,rT}^{(m)}}^* (g(W_1,\ldots, W_n)\otimes I_\cH)=g(rT_1,\ldots, rT_n)
  {K_{f,rT}^{(m)}}^*
 \end{equation}
 for any $g(W_1,\ldots, W_n) \in F_n^\infty({\bf D}^{m}_f)$ and $r\in(\delta,1)$.

 If, in addition,  $T:=(T_1,\ldots, T_n)\in {\bf
 D}^m_f(\cH)$ is a pure
 $n$-tuple of operators, then
$$
{\bf B}_{rT}[g]=\widetilde{\bf B}_T[g_r],\quad   r\in (\delta, 1),
$$
 where $g_r(W_1,\ldots,
W_n):=g(rW_1,\ldots, rW_n)$.
\end{lemma}
\begin{proof}

First, notice that Lemma \ref{Berezin-lemma} implies
 \begin{equation}\label{eq-ker2}
 {K^{(m)}_{f,rT}}^*[p(W_1,\ldots, W_n)\otimes I_\cH]=p(rT_1,\ldots, rT_n)
  {K^{(m)}_{f,rT}}^*
 \end{equation}
for any polynomial $p(W_1,\ldots, W_n)$ and $r\in (\delta,1)$.
  Since $rT:=(rT_1,\ldots, rT_n)\in {\bf D}^{(m)}_f(\cH)$, relation \eqref{limm} and
 Theorem \ref{Poisson-C*}   imply
  $$
 \lim_{k\to \infty} \sum_{|\alpha|\leq k}
 t^{|\alpha|} r^{|\alpha|} c_\alpha T_\alpha
   =g_t(rT_1,\ldots, rT_n) \quad \text{ for any } t\in [0,1), r\in (\delta, 1),
 $$
 where the convergence is in the operator norm topology.
  Using  relation \eqref{eq-ker2}, when
 $p(W_1,\ldots, W_n):=
 \sum\limits_{q=0}^k \sum\limits_{|\alpha|=q}t^{|\alpha|} c_\alpha
 W_\alpha$,
 and taking the limit as
 $k\to \infty$, we get
 \begin{equation}
 \label{eq-ker3}
 {K^{(m)}_{f,rT}}^* [g_t(W_1,\ldots, W_n)\otimes I_\cH]=g_t(rT_1,\ldots, rT_n)
 {K^{(m)}_{f,rT}}^*\quad \text{ for } \ r\in (\delta, 1).
 \end{equation}
 On the other hand, let us prove  that
\begin{equation}\label{lim-t}
 \lim_{t\to 1} g_t(rT_1,\ldots, rT_n)=g(rT_1,\ldots, rT_n),
 \end{equation}
 where the convergence is in the operator norm topology.
  Notice that, if  $\epsilon>0$,  there is   $m_0\in \NN$
 such that $\sum\limits_{k=m_0}^\infty r^k \left(\begin{matrix} k+m-1\\m-1
\end{matrix}\right)  <\frac {\epsilon} {4 M}$, where $M:=\|g(W_1,\ldots,
  W_n)(1)\|$.
  Since $(T_1,\ldots, T_n)\in {\bf D}^{(m)}_f(\cH)$,  Theorem \ref{Poisson-C*}
  and relation \eqref{ine-k} imply
  $$\left\| \sum_{|\beta|=k} b_\alpha T_\beta T_\beta^*\right\|\leq
  \left\| \sum_{|\beta|=k} b_\alpha W_\beta W_\beta^*\right\|\leq
  \left(\begin{matrix}
  k+m-1\\m-1
\end{matrix}\right).
  $$
  Now, we can deduce that
\begin{equation*}
 \begin{split}
 \sum_{k=m_0}^\infty r^k \left\|\sum_{|\beta|=k} c_\beta T_\beta\right\|
 &\leq
 \sum_{k=m_0}^\infty r^k \left(\sum_{|\beta|=k} |c_\alpha|^2\frac{1}
 {b_\beta}\right)^{1/2}\left\|\sum_{|\beta|=k} b_\beta T_\beta
 T_\beta^*\right\|^{1/2}  \\
 &\leq
  M \sum_{k=m_0}^\infty r^k \left(\begin{matrix} k+m-1\\m-1
\end{matrix}\right)
  <\frac {\epsilon} {4}.
 \end{split}
 \end{equation*}
 Consequently,   there exists $0< d<1$ such that
 \begin{equation*}
 \begin{split}
 \left\|\sum_{k=0}^\infty \sum_{|\alpha|=k}
  t^{|\alpha|} r^{|\alpha|} c_\alpha T_\alpha \right.&-\left.
  \sum_{k=0}^\infty \sum_{|\alpha|=k}
    r^{|\alpha|} c_\alpha T_\alpha\right\|\\
    &\leq  \frac{\epsilon}{2} +
    \left\| \sum_{k=1}^{m_0-1} r^k(t^k-1)\sum_{|\beta|
    =k}c_\beta T_\beta\right\|\\
    &\leq  \frac{\epsilon}{2} +M \sum_{k=1}^{m_0-1} r^k(t^k-1) \left(\begin{matrix} k+m-1\\m-1
\end{matrix}\right)<  \epsilon.
\end{split}
 \end{equation*}
   for any  $ t\in (d, 1)$.
  Hence, we deduce   \eqref{lim-t}. Using relations
   \eqref{SOT-lim}, \eqref{lim-t}, and  taking  the limit in
 \eqref{eq-ker3}, as $t\to 1$,  we obtain   \eqref{Kfrt}.
 Now, assume that $T$ is a pure $n$-tuple.
Based on  Proposition \ref{Berezin2} and  relations \eqref{def-Be2}, \eqref{frf},
 and \eqref{Kfrt}, we deduce that ${\bf
B}_{rT}[g]={\bf B}_T[g_r]$ for $r\in (\delta, 1)$.
The proof is complete.
\end{proof}

Using Theorem \ref{funct-calc} and Lemma \ref{BB}, we can deduce  the following
Fatou type result.

\begin{corollary} Let $T:=(T_1,\ldots, T_n)\in {\bf
 D}^m_f(\cH)$ be a pure $n$-tuple of operators and assume that $(W_1,\ldots, W_n)$
 has the radial property.
 Then
 $$
 \text{\rm SOT-}\lim_{r\to 1} {\bf B}_{rT}[g]=\widetilde{\bf B}_T[g]
 \quad \text{ for any } \ g\in F_n^\infty({\bf D}_f^m).
 $$
\end{corollary}
\begin{proof}
Recall that $\Phi_{f,T}(X):=\sum_{|\alpha|\geq 1} a_\alpha T_\alpha X T_\alpha^*$,
where the series is WOT-convergent.
Since the sequence
 $\sum_{1\leq |\alpha|\leq k} a_\alpha r^{2|\alpha|} W_\alpha W_\alpha^*$
is bounded and SOT-convergent to $\Phi_{f,rW}(I)$, as $k\to\infty$,
the proof of theorem \ref{purecase} implies
$$
\Phi_{f,rT}(I)={K_{f,T}^{(m)}}^*[\Phi_{f,rW}(I)\otimes I_\cH]K_{f,T}^{(m)}
$$
and, consequently,
$$
(I-\Phi_{f,rT})^m(I)=
{K_{f,T}^{(m)}}^*[(I-\Phi_{f,rW})^m(I)\otimes I_\cH]K_{f,T}^{(m)}.
$$
Since $(W_1,\ldots, W_n)$ has the radial property, so does $(T_1,\ldots, T_n)$.
Using  now Theorem \ref{funct-calc} and Lemma \ref{BB}, we can complete the proof.
\end{proof}

An $n$-tuple $T:=(T_1,\ldots, T_n)\in {\bf
 D}^m_f(\cH)$ is called completely non-coisometric (c.n.c) with respect
to the noncommutative domain ${\bf
 D}^m_f(\cH)$ if there is no vector $h\in \cH$, $h\neq 0$ such that
 $\left< \Phi_{f,T}^k(I)h,h\right>=\|h\|^2$ for any $k=1,2,\ldots$.
 Due to  relation  \eqref{K*K}, we
 $$
 \|K_{f,T}^{(m)}h\|^2=\|h\|^2-\|Q_{f,T}^{1/2} h\|^2,\quad h\in \cH,
 $$
 where $Q_{f,T}:=\text{\rm SOT-}\lim_{k\to\infty} \Phi_{f,T}^k(I)$.
Notice that $T$ is c.n.c. if and only if the noncommutative Berezin  kernel
$K_{f,T}^{(m)}$ is one-to-one.

Now, we can present an $F_n^\infty({\bf D}_f^m)$-functional calculus for c.n.c.
 $n$-tuples of operators
in the noncommutative domain ${\bf D}^m_f(\cH)$.

\begin{theorem} \label{funct-calc2}
Let ${\bf D}_f^m$ be a noncommutative domain such that the universal model
 $(W_1,\ldots, W_n)$ has the radial property.
  If $T:=(T_1,\ldots, T_n)\in {\bf
 D}^m_f(\cH)$ is a completely
non-coisometric $n$-tuple
   of operators  with the radial property,
 then
 $$\Phi(g):=\text{\rm SOT-}\lim_{r\to 1} g_r(T_1,\ldots, T_n), \qquad
  g=g(W_1,\ldots,
 W_n)\in F_n^\infty({\bf D}_f^m),
 $$
 exists in the strong operator topology   and defines a map
 $\Phi:F_n^\infty({\bf D}_f^m)\to B(\cH)$ with the following
 properties:
\begin{enumerate}
\item[(i)]
$\Phi(g)=\text{\rm SOT-}\lim\limits_{r\to 1}{\bf B}_{rT}[g]$, where ${\bf
B}_{rT}$ is the noncommutative  Berezin transform  at  $rT\in {\bf D}_f^m(\cH)$;
\item[(ii)] $\Phi$ is    WOT-continuous (resp.
SOT-continuous)  on bounded sets;
\item[(iii)]
$\Phi$ is a unital completely contractive homomorphism.
\end{enumerate}
 \end{theorem}

\begin{proof} Let $\delta\in (0,1)$ be such that $(rT_1,\ldots, rT_n)\in {\bf
 D}^m_f(\cH)$ and $(rW_1,\ldots, rW_n)\in {\bf
 D}^m_f(F^2(H_n))$ for any $r\in (\delta, 1)$.
 Due to  \eqref{gr} and
taking the limit in relation \eqref{frf}, as $r\to 1$ ,  we deduce that the map $G:
\text{\rm range} \,{K_{f,T}^{(m)}}^*\to \cH$ given by
$Gy:=\lim\limits_{r\to 1}g_r(T_1,\ldots,T_n)y$, $y\in \text{\rm
range} \,{K_{f,T}^{(m)}}^*$,
 is well-defined, linear, and
$$
\|G{K_{f,T}^{(m)}}^*\varphi\|\leq \limsup_{r\to 1} \|g_r(W_1,\ldots,
W_n)\|\|{K_{f, T}^{(m)}}^*\varphi\|\leq \|g(W_1,\ldots, W_n)\|\|{K_{f,
T}^{(m)}}^*\varphi\|
$$
   for any $\varphi\in
 F^2(H_n)\otimes \cH$.

Now, assume that $T=(T_1,\ldots, T_n)\in \cD_f(\cH)$ is c.n.c..
Since  the Berezin  kernel $K_{f,T}^{(m)}$ is
one-to-one, its  range  is dense in $\cH$.
Consequently, the map $G$ has a unique extension to a bounded linear
operator on $\cH$, denoted also by $G$, with  $\|G\|\leq
\|g(W_1,\ldots, W_n)\|$.
Let us show that
\begin{equation}\label{A}
\lim_{r\to 1} g_r(T_1,\ldots, T_n)h =Gh\quad \text{ for any }\ h\in
\cH.
\end{equation}
Let $\{y_k\}_{k=1}^\infty$  be a sequence of vectors
in the range of $K_{f,T}^*$, which converges to $y$.
According to Theorem \ref{Poisson-C*} and relations \eqref{limm}, \eqref{gr},  we have
$$
\|g_r(T_1,\ldots, T_n)\|\leq\|g_r(W_1,\ldots, W_n)\|\leq
\|g(W_1,\ldots, W_n)\|
$$
for any $r\in(\delta, 1)$.
Let $\{y_k\}_{k=1}^\infty$  be a sequence of vectors
in the range of ${K_{f,T}^{(m)}}^*$, which converges to $y$, and notice that
\begin{equation*}
\begin{split}
\|Gh-g_r(T_1,\ldots, T_n)h\|&\leq \|Gh-G
y_k\|+\|Gy_k-g_r(T_1,\ldots, T_n)y_k\|\\
&\qquad \qquad  +\|g_r(T_1,\ldots,
T_n)y_k-g_r(T_1,\ldots, T_n)h\|\\
&\leq 2\|g(W_1,\ldots, W_n)\| \|h-y_k\|+\|Gy_k-g_r(T_1,\ldots,
T_n)y_k\|.
\end{split}
\end{equation*}
Since $\lim\limits_{r\to 1} g_r(T_1,\ldots, T_n)y_k =Gy_k$, relation \eqref{A}
 follows.
 Due to Lemma \ref{BB},   we have
\begin{equation}
\label{anot} g_r(T_1,\ldots, T_n)={K_{f,rT}^{(m)}}^* [g(W_1,\ldots,
W_n)\otimes I_\cH]K_{f,rT}^{(m)},
\end{equation}
which together with \eqref{A} imply part (i) of the theorem.

Now let us prove part (ii). Due to relation \eqref{anot}, we have
$\|g_r(T_1,\ldots, T_n)\|\leq \|g(W_1,\ldots, W_n)\|$ and, therefore,
$\|\Phi(g)\|\leq \|g\|$ for  $g\in F_n^\infty({\bf D}_f^{m})$.
Taking $r\to 1$ in  relation \eqref{frf} of Lemma \ref{gKKg} and using part (i), we obtain
\begin{equation}
\label{Phi-Kf}
\Phi(g){K_{f,T}^{(m)}}^*={K_{f,T}^{(m)}}^*(g\otimes I),\quad
g\in F_n^\infty({\bf D}_f^{m}).
\end{equation}
Let $\{g_i\}$
be a bounded net in $F_n^\infty({\bf D}_f^{m})$ such that
 $g_i\to g\in F_n^\infty({\bf D}_f^{m})$ in the weak (resp. strong)
 operator topology. Then $g_i\otimes I$ converges to $ g\otimes I$ in the same topologies.
 By \eqref{Phi-Kf}, we have
  $\Phi(g_i){K_{f,T}^{(m)}}^*={K_{f,T}^{(m)}}^*(g_i\otimes I)$.
Since the range of ${K_{f,T}^{(m)}}^*$ is dense in $\cH$ and
 $\{\Phi(g_i)\}$ is bounded, an approximation argument shows that
 $\Phi(g_i)\to \Phi(g)$ in the weak (resp. strong)
 operator topology.

 To prove (iii), note that \eqref{anot} and the fact that  $K_{f,rT}$
 is an isometry for $r\in(\delta, 1)$ imply
$$
\|[g_{ij}(rT_1,\ldots, rT_n)]_k\|\leq \|[g_{ij}]_k\|
$$
for any operator-valued matrix $[g_{ij}]_k\in M_k(F_n^\infty({\bf D}_f^{m}))$
and $r\in (\delta, 1)$.
Hence, and due to the fact that
 $\Phi(g_{ij})=\text{\rm SOT-}\lim_{r\to 1} g_{ij}(rT_1,\ldots, rT_n)$,
we deduce that $\Phi$ is completely contractive map.
On the other hand, due to Theorem \ref{Poisson-C*}, $\Phi$ is a homomorphism
 on polynomials in $W_1,\ldots, W_n$ and the identity.
 Since these polynomials are sequentially WOT-dense in
 $F_n^\infty({\bf D}_f^{m})$ (see Proposition \ref{density-pol})
 and $\Phi$ is WOT-continuous  on bounded sets, we deduce part (iii).
  The proof is complete.
\end{proof}

Consider the particular case when the domain ${\bf D}_p^m$, $m\geq 1$, is determined by
the noncommutative polynomial $p=a_1Z_1+\cdots + a_nZ_n$, $a_i>0$. Due
to Example \ref{ex-radial}, ${\bf D}_p^m$ has the radial property. Therefore, according to
Theorem \ref{funct-calc2}, there is an $F_n^\infty({\bf D}_p^{m})$-functional
calculus for any c.n.c. $n$-tuple of operators
 in ${\bf D}_p^m(\cH)$. When $m\geq 2$, $n=1$, and $p=Z$,
 we obtain a functional calculus for Agler's $m$-hypercontractions.
On the other hand, if $m=1$, $n=1$, and $p=Z_1+\cdots + Z_n$, we obtain the
$F_n^\infty$-functional calculus for row contractions \cite{Po-funct}. Moreover, if
$m=1$, $n=1$, and $p=Z$, we obtain the
 Nagy-Foias $H^\infty$-functional calculus for c.n.c contractions. We remark that the
   $H^\infty$-functional calculus works for a larger class of contractions
   (see  \cite{SzF-book}).

\bigskip

\section{Weighted shifts, symmetric weighted Fock spaces, and  multipliers}
 \label{Symmetric}

In  this section,  we find all the eigenvectors  for $W_1^*,\ldots,
W_n^*$, where $(W_1,\ldots, W_n)$  is the  universal model
associated with the noncommutative domain ${\bf D}^m_f$. As
consequences, we
 identify the $w^*$-continuous multiplicative
linear functional on the Hardy algebra $F_n^\infty({\bf D}^m_f)$ and
find  the joint right spectrum of  $(W_1,\ldots, W_n)$. We introduce
the symmetric weighted Fock space $F_s^2({\bf D}^m_f)$ and identify
it with a  reproducing kernel Hilbert space $H^2({\bf D}_{f,\circ}^1(\CC))$.
We also show that the
algebra of all its multipliers is  reflexive. This section  plays an
important role in connecting the  results of the present paper to
analytic function theory on Reinhardt domains in $\CC^n$, as well
as, to model theory for commuting $n$-tuples of operators.

 Let $f=\sum_{|\alpha|\geq 1} a_\alpha X_\alpha$ be a
positive regular free holomorphic function on $[B(\cH)^n]$,
$\rho>0$, and define
$${\bf D}_{f,\circ}^1(\CC):=\left\{ \lambda=(\lambda_1,\ldots, \lambda_n)\in
\CC^n:\ \sum_{|\alpha|\geq 1} a_\alpha
|\lambda_\alpha|^2<1\right\}\subset {\bf D}^m_f(\CC),
$$
where  $\lambda_\alpha:=\lambda_{i_1}\cdots \lambda_{i_m}$ if
$\alpha=g_{i_1}\cdots g_{i_m}\in \FF_n^+$, and $\lambda_{g_0}$=1.

\begin{theorem}\label{eigenvectors}
Let $(W_1,\ldots, W_n)$ (resp. $(\Lambda_1,\ldots, \Lambda_n)$) be
the weighted left (resp. right)  creation operators associated with
the noncommutative domain ${\bf D}^m_f$. The eigenvectors for
$W_1^*,\ldots, W_n^*$ (resp. $\Lambda_1^*,\ldots, \Lambda_n^*$) are
precisely the vectors
$$
z_\lambda:= \left( I-\sum_{|\alpha|\geq 1} a_{\tilde \alpha}
\overline{\lambda}_\alpha \Lambda_\alpha\right)^{-m}(1)
 \in F^2(H_n)\quad \text{ for }\
\lambda=(\lambda_1,\ldots, \lambda_n)\in {\bf D}_{f,\circ}^1(\CC),
$$
where $\tilde \alpha$ denotes the reverse of $\alpha$.
 They satisfy the equations
$$
W_i^*z_\lambda=\overline{\lambda}_i z_\lambda, \quad
\Lambda_i^*z_\lambda=\overline{\lambda}_i z_\lambda \qquad \text{
for } \ i=1,\ldots,n,
$$
 and each vector $z_\lambda$ is cyclic for $R_n^\infty({\bf D}^m_f)$.

If $\lambda:=(\lambda_1,\ldots, \lambda_n)\in {\bf
D}_{f,\circ}^1(\CC)$ and $\varphi(W_1,\ldots, W_n):=\sum_{\beta\in
\FF_n^+} c_\beta W_\beta$ is in the Hardy algebra $F_n^\infty({\bf
D}^m_f)$, then $\sum_{\beta\in \FF_n^+}
|c_\beta||\lambda_\beta|<\infty$ and the map
$$\Phi_\lambda:F_n^\infty({\bf D}_f)\to \CC, \qquad
 \Phi_\lambda(\varphi(W_1,\ldots, W_n)):=\varphi(\lambda),
$$
is $w^*$-continuous and multiplicative. Moreover,
$\varphi(W_1,\ldots,
W_n)^*z_\lambda=\overline{\varphi(\lambda)}z_\lambda$ and
$$
\varphi(\lambda)=\left<\varphi(W_1,\ldots, W_n)1,
z_\lambda\right>=\left< \varphi(W_1,\ldots, W_n)u_\lambda,
u_\lambda\right>,
$$
where $u_\lambda:=\frac{z_\lambda}{\|z_\lambda\|}$.
\end{theorem}

\begin{proof}
Since $\sum_{|\alpha|\geq 1} a_{\tilde \alpha}
  \Lambda_\alpha \Lambda_\alpha^*$ is SOT-convergent   and, for any $\lambda=(\lambda_1,\ldots, \lambda_n)\in {\bf
D}_{f,\circ}^1(\CC)$,

  \begin{equation*}
  \begin{split}
  \left\|\sum_{|\alpha|\geq 1} a_{\tilde \alpha}
  \overline{\lambda}_{\tilde\alpha} \Lambda_\alpha\right\|&\leq
\left\|\sum_{|\alpha|\geq 1} a_{\tilde \alpha}
  \Lambda_\alpha \Lambda_\alpha^*\right\|
  \left(\sum_{|\alpha|\geq 1} a_\alpha
  |\lambda_\alpha|^2\right)^{1/2}\\
  &\leq\left(\sum_{|\alpha|\geq 1} a_\alpha
  |\lambda_\alpha|^2\right)^{1/2}<1,
\end{split}
  \end{equation*}
the operator $\left( I-\sum_{|\alpha|\geq 1} a_{\tilde \alpha}
\overline{\lambda}_{\tilde\alpha} \Lambda_\alpha\right)^{-m}$ is
well-defined. Due to the results of Section 1 (see Lemma \ref{b-alpha}), we
have
$$
\left( I-\sum_{|\alpha|\geq 1} a_{\tilde \alpha}
\overline{\lambda}_{\tilde\alpha} \Lambda_\alpha\right)^{-m}=
\sum_{\beta\in \FF_n^+}  b^{(m)}_{\tilde
    \beta}\overline{\lambda}_{\tilde \beta} \Lambda_\beta,
    $$
    where the coefficients $b_\beta$, $\beta\in \FF_n^+$, are defined by
relation \eqref{b-al}.
    Hence, and using relation \eqref{WbWb-r}, we obtain
    \begin{equation}\label{z-l}
    z_\lambda=
    \left( I-\sum_{|\alpha|\geq 1} a_{\tilde \alpha}
\overline{\lambda}_\alpha \Lambda_\alpha\right)^{-m}(1)=\sum_{\beta\in
\FF_n^+}  b^{(m)}_{\tilde
    \beta}\overline{\lambda}_{\tilde \beta} \Lambda_\beta (1)=
    \sum_{\beta\in
\FF_n^+}  \sqrt{b^{(m)}_{
    \beta}}\overline{\lambda}_{ \beta} e_\beta.
    \end{equation}
 The fact  that $z_\lambda\in F^2(H_n)$ is a cyclic  vector for
$R_n^\infty(\cD_f)$ is obvious.

Now, notice that if $\lambda=(\lambda_1,\ldots, \lambda_n)\in {\bf
D}_{f,\circ}^1(\CC)$, then $\lambda$ is  of class $C_{\cdot 0}$ with
respect to ${\bf D}_{f,\circ}^1(\CC)$. Using relation \eqref{I-Q} in
our particular case, we get
$$
\left(1-\sum_{|\alpha|\geq 1} a_\alpha
|\lambda_\alpha|^2\right)^m\left(\sum_{\beta\in \FF_n^+}
b_\beta^{(m)} |\lambda_\beta|^2\right)= 1.
$$
Consequently,  we have
\begin{equation}
\label{norm-z}
\|z_\lambda\|=\frac{1}{\sqrt{\left(1-\sum_{|\alpha|\geq 1} a_\alpha
|\lambda_\alpha|^2\right)^m}}.
\end{equation}
Due to  relation \eqref{WbWb}, we have
$$
W_i^* e_\alpha =\begin{cases} \frac
{\sqrt{b_\gamma^{(m)}}}{\sqrt{b_{\alpha}^{(m)}}}e_\gamma& \text{ if
}
\alpha=g_i\gamma \\
0& \text{ otherwise. }
\end{cases}
$$
A simple computation shows that $W_i^*
z_\lambda=\overline{\lambda}_{i} z_\lambda$ for $i=1,\ldots, n$.
Similarly, one can use relation \eqref{WbWb-r} to prove that
$\Lambda_i^*z_\lambda=\overline{\lambda}_i z_\lambda$ for
$i=1,\ldots, n$.

Conversely, let $z=\sum_{\beta\in \FF_n^+} c_\beta e_\beta \in
F^2(H_n)$ and assume that $W_i^*z=\overline{\lambda}_i z$,
$i=1,\ldots,n$, for some $n$-tuple $(\lambda_1,\ldots, \lambda_n)\in
\CC^n$.  Using the definition of the weighted left creation
operators $W_1,\ldots, W_n$, we deduce that
\begin{equation*}
\begin{split}
c_\alpha&=\left< z, e_\alpha\right>=\left< z, \sqrt{b_\alpha^{(m)}}
W_\alpha(1)\right>\\
&=\sqrt{b_\alpha^{(m)}}\left< W_\alpha^* z,1\right>=
\sqrt{b_\alpha^{(m)}}\overline{\lambda}_\alpha\left< z,1\right>\\
&=c_0\sqrt{b_\alpha^{(m)}}\overline{\lambda}_\alpha
\end{split}
\end{equation*}
 for any $\alpha\in \FF_n^+$, whence
  $z= a_0\sum_{\beta\in \FF_n^+} \sqrt{b_\beta^{(m)}}
\overline{\lambda}_\beta e_\beta$.
 Since $z\in F^2(H_n)$, we must have $\sum_{\beta\in \FF_n^+}
 b_\beta^{(m)}
|\lambda_\beta|^2<\infty$. On the other hand, relation \eqref{b-al}
implies
$$
\left(\sum_{j=0}^k \left(\sum_{|\alpha|\geq 1} a_\alpha
|\lambda_\alpha|^2\right)^j\right)^m\leq \sum_{\beta\in \FF_n^+}
b_\beta^{(m)} |\lambda_\beta|^2<\infty
$$
for any $k\in \NN$. Letting $k\to\infty$ in the relation above, we
must have $\sum_{|\alpha|\geq 1} a_\alpha |\lambda_\alpha|^2<1$,
whence $(\lambda_1,\ldots, \lambda_n)\in {\bf D}_{f,\circ}^1(\CC)$.
A similar result can be proved for the weighted right creation
operators $\Lambda_1,\ldots, \Lambda_n$ if one uses relation
\eqref{WbWb-r}.

Now, let us prove the last part of the theorem. Since
$\varphi(W_1,\ldots, W_n)=\sum_{\beta\in \FF_n^+} c_\beta W_\beta$
is in the Hardy algebra $F_n^\infty({\bf D}^m_f)$, we have
$\sum_{\beta\in \FF_n^+} |c_\beta|^2 \frac{1}{b_\beta^{(m)}}<\infty$
(see Section \ref{Hardy}). As shown above, if
$\lambda=(\lambda_1,\ldots, \lambda_n)\in {\bf D}_{f,\circ}^1(\CC)$,
then $\sum_{\beta\in \FF_n^+} b_\beta^{(m)}
|\lambda_\beta|^2<\infty$. Applying Cauchy's inequality, we have
$$
\sum_{\beta\in \FF_n^+} |c_\beta||\lambda_\beta|\leq
\left(\sum_{\beta\in \FF_n^+} |c_\beta|^2
\frac{1}{b^{(m)}_\beta}\right)^{1/2} \left(\sum_{\beta\in \FF_n^+}
b_\beta^{(m)} |\lambda_\beta|^2\right)^{1/2}<\infty.
$$
Note  also that
\begin{equation*}
\begin{split}
\left<\varphi(W_1,\ldots, W_n)1, z_\lambda\right>&= \left<
\sum_{\beta\in \FF_n^+}c_\beta \frac{1}{\sqrt{b^{(m)}_\beta}}
e_\beta, \sum_{\beta\in \FF_n^+} \sqrt{b^{(m)}_\beta}
\overline{\lambda}_\beta
e_\beta \right>\\
&=\sum_{\beta\in\FF_n^+} c_\beta \lambda_\beta
=\varphi(\lambda_1,\ldots, \lambda_n).
\end{split}
\end{equation*}
Now, for each $\beta\in \FF_n^+$, we have
\begin{equation*}
\begin{split}
\left<\varphi(W_1,\ldots, W_n)^* z_\lambda,
\frac{1}{\sqrt{b^{(m)}_\alpha}}e_\beta\right>&= \left< z_\lambda,
\varphi(W_1,\ldots, W_n) W_\beta (1)\right>\\
&=\overline {\lambda_\beta \varphi(\lambda)} =\left<\overline {
\varphi(\lambda)} z_\lambda, \frac{1}{\sqrt{b^{(m)}_\alpha}}
e_\beta\right>.
\end{split}
\end{equation*}
Hence, we deduce that
\begin{equation}
\label{fi*z} \varphi(W_1,\ldots, W_n)^* z_\lambda=\overline {
\varphi(\lambda)} z_\lambda.
\end{equation}
One can easily see that
\begin{equation*}
\begin{split}
\left< \varphi(W_1,\ldots, W_n)u_\lambda, u_\lambda\right>&=
\frac{1}{\|z_\lambda\|^2}\left< z_\lambda,\varphi(W_1,\ldots, W_n)^*
z_\lambda\right>\\
&=\frac{1}{\|z_\lambda\|^2}\left< z_\lambda,\overline {
\varphi(\lambda)} z_\lambda\right>= \varphi(\lambda).
\end{split}
\end{equation*}
The fact that  the map $\Phi_\lambda$ is multiplicative and
$w^*$-continuous is now obvious.  This completes the proof.
\end{proof}

As in \cite{DP1}, in the particular case when $m=1$ and
$f=X_1+\cdots + X_n$,  one can similarly prove (using Theorem
\ref{eigenvectors}) the following.

 \begin{proposition}
 A  map $\varphi:
F_n^\infty({\bf D}^m_f)\to \CC$ is a $w^*$-continuous multiplicative
linear functional  if and only if there exists $\lambda\in {\bf
D}_{f,\circ}^1(\CC)$ such that
$$
\varphi(A)=\varphi_\lambda(A):=\left<Au_\lambda,
u_\lambda\right>,\quad A\in F_n^\infty(\cD_f),
$$
where $u_\lambda:=\frac{z_\lambda}{\|z_\lambda\|}$.
\end{proposition}

 We recall   that the joint  right spectrum
  $\sigma_r(T_1,\ldots, T_n)$ of an   $n$-tuple
 $(T_1,\ldots, T_n)$ of operators
   in $B(\cH)$ is the set of all $n$-tuples
    $(\lambda_1,\ldots, \lambda_n)$  of complex numbers such that the
     right ideal of $B(\cH)$  generated by the operators
     $\lambda_1I-T_1,\ldots, \lambda_nI-T_n$ does
      not contain the identity operator.
We recall \cite{Po-unitary} that
  $(\lambda_1,\ldots, \lambda_n)\notin \sigma_r(T_1,\ldots, T_n)$
  if and only if  there exists $\delta>0$ such that
  $\sum\limits_{i=1}^n (\lambda_iI-T_i)
  (\overline{\lambda}_iI-T_i^*)\geq \delta I$.

 Theorem \ref{eigenvectors} implies the following result. Since the
 proof is similar to the proof of Theorem 5.1 from \cite{Po-disc}, we
 shall omit it.

\begin{proposition}\label{right-spec} If ~$(W_1,\ldots, W_n)$~ are the
weighted left creation operators associated with the noncommutative
domain ${\bf D}^m_f$, then the right joint spectrum
$\sigma_r(W_1,\ldots, W_n)$ coincide with ${\bf D}^{1}_f(\CC)$.
\end{proposition}

\bigskip

Now, we define the symmetric weighted Fock space associated with the
noncommutative domain ${\bf D}^m_f$. We need a few definitions. For
each $\lambda=(\lambda_1,\ldots, \lambda_n)\in \CC^n$ and each $n$-tuple
${\bf k}:=(k_1,\ldots, k_n)\in \NN_0^n$, where $\NN_0:=\{0,1,\ldots
\}$, let $\lambda^{\bf k}:=\lambda_1^{k_1}\cdots \lambda_n^{k_n}$.
For each ${\bf k}\in \NN_0$, we denote
$$
\Lambda_{\bf k}:=\{\alpha\in \FF_n^+: \ \lambda_\alpha =\lambda^{\bf
k} \text{ for all } \lambda\in \CC^n\}.
$$
For each ${\bf k}\in \NN_0^n$, define the vector
$$
w^{\bf k}:=\frac{1}{\gamma^{(m)}_{\bf k}} \sum_{\alpha \in
\Lambda_{\bf k}} \sqrt{b_\alpha^{(m)}} e_\alpha\in F^2(H_n), \quad
\text{ where } \ \gamma^{(m)}_{\bf k}:=\sum_{\alpha\in \Lambda_{\bf
k}} b^{(m)}_\alpha
$$
 and the
coefficients $b^{(m)}_\alpha$, $\alpha\in \FF_n^+$, are defined by
relation \eqref{b-al}. Note that the set  $\{w^{\bf k}:\ {\bf k}\in
\NN_0^n\}$ consists  of orthogonal vectors in $F^2(H_n)$ and
$\|w^{\bf k}\|=\frac{1}{\sqrt{\gamma^{(m)}_{\bf k}}}$. We denote by
$F_s^2({\bf D}^m_f)$ the closed span of these vectors, and call it
the symmetric weighted  Fock space associated with the
noncommutative domain ${\bf D}^m_f$.

If $\cQ$ is a set of noncommutative polynomials, we define the subspace $\cM_\cQ$ of $F^2(H_n)$ by
setting
$$\cM_\cQ:=\overline{\text{\rm span}}\{W_\alpha q(W_1,\ldots, W_n) W_\beta(1): \ q\in \cQ,
 \alpha, \beta\in \FF_n^+\}.
 $$

\begin{theorem}\label{symm-Fock} Let $f=\sum_{|\alpha|\geq 1} a_\alpha X_\alpha$ be a
positive regular free holomorphic function on $[B(\cH)^n]_\rho$,  $\rho>0$,  and  let
$\cQ_c$ be the  set of all polynomials of the form
$$Z_iZ_j-Z_j Z_i,\qquad i,j=1,\ldots, n.
$$
 Then the following statements hold:
 \begin{enumerate}
 \item[(i)]
 $
 F_s^2({\bf D}^m_f)=\overline{\text{\rm span}}\{z_\lambda: \ \lambda\in
{\bf D}_{f,\circ}^1(\CC)\}=\cN_{\cQ_c}:=F^2(H_n)\ominus
\cM_{\cQ_c}$.
\item[(ii)] The symmetric weighted Fock space $F_s^2({\bf D}^m_f)$ can be
identified with the Hilbert space $H^2({\bf D}_{f,\circ}^1(\CC))$ of
all functions $\varphi:{\bf D}_{f,\circ}^1(\CC)\to \CC$ which admit
a power series representation $\varphi(\lambda)=\sum_{{\bf k}\in
\NN_0} c_{\bf k} \lambda^{\bf k}$ with
$$
\|\varphi\|_2=\sum_{{\bf k}\in \NN_0}|c_{\bf
k}|^2\frac{1}{\gamma^{(m)}_{\bf k}}<\infty.
$$
More precisely, every  element  $\varphi=\sum_{{\bf k}\in \NN_0}
c_{\bf k} w^{\bf k}$ in $F_s^2({\bf D}^m_f)$  has a functional
representation on ${\bf D}_{f,\circ}^1(\CC)$ given by
$$
\varphi(\lambda):=\left<\varphi, z_\lambda\right>=\sum_{{\bf k}\in
\NN_0} c_{\bf k} \lambda^{\bf k}, \quad \lambda=(\lambda_1,\ldots,
\lambda_n)\in {\bf D}_{f,\circ}^1(\CC),
$$
and
$$
|\varphi(\lambda)|\leq
\frac{\|\varphi\|_2}{\sqrt{\left(1-\sum_{|\alpha|\geq 1} a_\alpha
|\lambda_\alpha|^2\right)^m}},\quad \lambda=(\lambda_1,\ldots,
\lambda_n)\in {\bf D}_{f,\circ}^1(\CC).
$$
\item[(iii)]
 The mapping $K_f:{\bf D}_{f,\circ}^1(\CC)\times
{\bf D}_{f,\circ}^1(\CC)\to \CC$ defined by
$$
K_f(\mu,\lambda):=\frac{1}{\left(1-\sum_{|\alpha|\geq 1} a_\alpha
\mu_\alpha \overline{\lambda}_\alpha\right)^m}\quad \text{ for all
}\ \lambda,\mu\in {\bf D}_{f,\circ}^1(\CC)
$$
is positive definite, and $K_f(\mu,\lambda)= \left<z_\lambda,
z_\mu\right>$.
\end{enumerate}
\end{theorem}

\begin{proof} First, we prove that
$$\overline{\text{\rm span}}\{z_\lambda: \ \lambda\in
{\bf D}_{f,\circ}^1(\CC)\}\subseteq F_s^2({\bf D}^m_f)\subseteq
\cN_{\cQ_c}.
$$
Notice that the first inclusion is due to that fact that
$z_\lambda=\sum_{{\bf k}\in \NN_0^n} \overline{\lambda}^{\bf k}
\gamma_{\bf k} w^{\bf k}$ for  $\lambda \in {\bf
D}_{f,\circ}^1(\CC)$. To prove the second inclusion, note that, due
to
 relation \eqref{WbWb}, we have
\begin{equation*}
 \begin{split}
 \left<w^{\bf k},
W_\gamma(W_jW_i-W_iW_j)W_\beta(1)\right>&= \frac{1}{\gamma_{\bf
k}}\left<\sum_{\alpha \in \Lambda_{\bf k}} \sqrt{b^{(m)}_\alpha}
e_\alpha, \frac{1}{\sqrt{b^{(m)}_{\gamma g_jg_i\beta}}} e_{\gamma
g_jg_i\beta}- \frac{1}{\sqrt{b^{(m)}_{\gamma g_ig_j\beta}}}
e_{\gamma g_ig_j\beta}\right>=0
\end{split}
\end{equation*}
for any ${\bf k}\in \NN_0^n$, $\alpha, \beta\in \FF_n^+$,
$i,j=1,\ldots, n$. This shows that $w^{\bf k}\in \cN_{\cQ_c}$ and
proves our assertion. To complete the proof of part (i), it is
enough to show that
$$
\cN_{\cQ_c}\subseteq \overline{\text{\rm span}}\{z_\lambda: \
\lambda\in {\bf D}_{f,\circ}^1(\CC)\}.
$$
To this end, assume that there is a vector $ x:=\sum_{\beta\in
\FF_n^+} c_\beta e_\beta\in \cN_{\cQ_c}$ and $x\perp z_\lambda$ for
all $\lambda\in {\bf D}_{f,\circ}^1(\CC)$. Then, using \eqref{z-l}, we obtain
\begin{equation*}
\left<\sum_{\beta\in \FF_n^+} c_\beta e_\beta, z_\lambda\right>
=\sum_{{\bf k}\in \NN_0^n}\left(\sum_{\beta\in\Lambda_{\bf k}}
c_\beta \sqrt{b_\beta^{(m)}}\right)\lambda^{\bf k}=0
\end{equation*}
for any $\lambda\in {\bf D}_{f,\circ}^1(\CC)$. Since ${\bf
D}_{f,\circ}^1(\CC)$ contains an open ball in $\CC^n$, we deduce
that
\begin{equation}\label{sigma=0}
\sum_{\beta\in \Lambda_{\bf k}} c_\beta \sqrt{b_\beta^{(m)}}=0 \quad
\text{ for all } \ {\bf k}\in \NN_0^n.
\end{equation}
Fix $\beta_0\in \Lambda_{\bf k}$ and let $\beta\in \Lambda_{\bf k}$
be such that $\beta$ is obtained from $\beta_0$ by transposing just
two generators. So we can assume that $\beta_0=\gamma g_j g_i\omega$
and $\beta=\gamma g_i g_j\omega$ for some $\gamma,\omega\in \FF_n^+$
and $i\neq j$, $i,j=1,\ldots,n$. Since $x\in
\cN_{\cQ_c}=F^2(H_n)\ominus \cM_{\cQ_c}$, we must have
$$
\left<x,W_\gamma(W_jW_i-W_iW_j)W_\omega(1)\right>=0,
$$
which implies
$\frac{c_{\beta_0}}{\sqrt{b^{(m)}_{\beta_0}}}=\frac{c_\beta}{\sqrt{b^{(m)}_{\beta}}}$.
Since any element $\gamma\in \Lambda_{\bf k}$ can be obtained from
$\beta_0$  by successive  transpositions, repeating the above
argument, we deduce that
$$
\frac{c_{\beta_0}}{\sqrt{b^{(m)}_{\beta_0}}}=\frac{c_\gamma}{\sqrt{b^{(m)}_{\gamma}}}
\quad \text{ for all }\ \gamma\in \Lambda_{\bf k}.
$$
Setting $t:=\frac{c_{\beta_0}}{\sqrt{b^{(m)}_{\beta_0}}}$, we have
$c_\gamma=t\sqrt{b^{(m)}_\gamma}$, $\gamma\in \Lambda_{\bf k}$, and
relation  \eqref{sigma=0} implies  $t=0$ (remember that
$b_\beta>0$). Therefore, $c_\gamma=0$ for any $\gamma\in
\Lambda_{\bf k}$ and ${\bf k}\in \NN_0^n$, so $x=0$. Consequently,
we have $\overline{\text{\rm span}}\{z_\lambda: \ \lambda\in {\bf
D}_{f,\circ}^1(\CC)\}=\cN_{\cQ_c}. $

Now, let us prove part (ii) of the theorem. Since
 the set  $\{w^{\bf k}:\ {\bf k}\in
\NN_0^n\}$ consists  of orthogonal vectors in $F^2(H_n)$ with
$\|w^{\bf k}\|=\frac{1}{\sqrt{\gamma^{(m)}_{\bf k}}}$, and
$F_s^2({\bf D}^m_f)$ the closed span of these vectors, any
$\varphi\in F_s^2({\bf D}^m_f)$ has a unique representation
$\varphi=\sum_{{\bf k}\in \NN_0} c_{\bf k} w^{\bf k}$ with $
\|\varphi\|_2=\sum_{{\bf k}\in \NN_0}|c_{\bf
k}|^2\frac{1}{\gamma^{(m)}_{\bf k}}<\infty$.
 Note that
\begin{equation*}
\left<w^{\bf k},z_\lambda\right>=\frac{1}{\gamma_{\bf k}} \left<
\sum_{\beta\in\Lambda_{\bf k}}\sqrt{b^{(m)}_\beta} e_\beta,
z_\lambda\right> =\frac{1}{\gamma_{\bf k}}\sum_{\beta\in\Lambda_{\bf
k}} b^{(m)}_\beta \lambda_\beta =\lambda^{\bf k}
\end{equation*}
for any $\lambda\in {\bf D}_{f,\circ}^1(\CC)$ and ${\bf k}\in
\NN_0^n$. Hence, every element $\varphi=\sum_{{\bf k}\in \NN_0}
c_{\bf k} w^{\bf k}$ in $F_s^2({\bf D}^m_f)$  has a functional
representation on ${\bf D}_{f,\circ}^1(\CC)$ given by
$$
\varphi(\lambda):=\left<\varphi, z_\lambda\right>=\sum_{{\bf k}\in
\NN_0} c_{\bf k} \lambda^{\bf k}, \quad \lambda=(\lambda_1,\ldots,
\lambda_n)\in {\bf D}_{f,\circ}^1(\CC),
$$
and, due to \eqref{norm-z},
$$|\varphi(\lambda)|\leq \|\varphi\|_2 \|z_\lambda\|=
\frac{\|\varphi\|_2}{\sqrt{\left(1-\sum_{|\alpha|\geq 1} a_\alpha
|\lambda_\alpha|^2\right)^m}}.
$$
 The identification of $F_s^2({\bf
D}^m_f)$ with $H^2({\bf D}_{f,\circ}^1(\CC))$  is now clear.

As in the proof of Theorem \ref{eigenvectors}, we deduce that
$$
\left( I-\sum_{|\alpha|\geq 1} a_{\tilde \alpha}
\overline{\lambda}_{\tilde\alpha} \Lambda_\alpha\right)^{-m}=
\sum_{\beta\in \FF_n^+}  b^{(m)}_{\tilde
    \beta}\overline{\lambda}_{\tilde \beta} \Lambda_\beta
    $$
    if $
 (\lambda_1,\ldots, \lambda_n)\in
{\bf D}_{f,\circ}^1(\CC)$.
  Similarly, if
$(\mu_1,\ldots, \mu_n)\in{\bf D}_{f,\circ}^1(\CC)= {\bf D}_{f,\circ}^1(\CC)\cap  {\bf
D}_{\tilde f,\circ}^1(\CC)$, we deduce that
$$
\sum_{\beta\in \FF_n^+} b_\beta \mu_\beta\overline{\lambda}_\beta =
\left( I-\sum_{|\alpha|\geq 1} a_{ \alpha}
 \mu_\alpha\overline{\lambda}_{\alpha}\right)^{-m}.
$$
Since
$$K_f(\mu,\lambda)=
\sum_{\beta\in \FF_n^+} b^{(m)}_\beta
\mu_\beta\overline{\lambda}_\beta= \left<z_\lambda, z_\mu\right>,
$$
the result in part (iii) follows. The proof is complete.
\end{proof}

  Let
$J_c$ be the $w^*$-closed two-sided ideal of the Hardy algebra
$F_n^\infty({\bf D}^m_f)$ generated by the commutators
$$W_iW_j-W_j W_i,\qquad i,j=1,\ldots, n.
$$
Since $W_iW_j-W_jW_i\in J_c$ and every permutation of $k$ objects is
a product of transpositions, it is clear that $W_\alpha
W_\beta-W_\beta W_\alpha\in J_c$ for any $\alpha,\beta\in \FF_n^+$.
Consequently, $W_\gamma(W_\alpha W_\beta-W_\beta
W_\alpha)W_\omega\in J_c$ for any $\alpha,\beta, \gamma,\omega\in
\FF_n^+$. Since the polynomials in $W_1,\ldots, W_n$ are $w^*$ dense
in $F_n^\infty({\bf D}^m_f)$, we deduce that
  $J_c$ coincides with the $w^*$-closure  of the
 commutator ideal  of $F_n^\infty({\bf D}^m_f)$.

 Define the operators
on  $F_s^2({\bf D}^m_f)$ by
$$L_i:=P_{F_s^2(\cD_f)} W_i|_{F_s^2(\cD_f)}, \qquad i=1,\ldots, n,
$$
  where $W_1,\ldots, W_n$ are the weighted left
creation operators associated with ${\bf D}^m_f$. Let
$\varphi(W_1,\ldots, W_n)\in F_n^\infty({\bf D}^m_f)$ and denote
$M_\varphi:=P_{F_s^2({\bf D}^m_f)} \varphi(W_1,\ldots,
W_n)|_{F_s^2({\bf D}^m_f)}$.
According to Theorem \ref{eigenvectors} and Theorem \ref{symm-Fock}, the vector
 $z_\lambda$ is in $ F_s^2({\bf D}^m_f)$
for $\lambda\in {\bf D}_{f,\circ}^1(\CC)$, and $\varphi(W_1,\ldots,
W_n)^*z_\lambda= \overline{\varphi(\lambda)} z_\lambda$. Consequently,  we have
\begin{equation*}
\begin{split}
[M_\varphi \psi](\lambda)&=\left<M_\varphi\psi, z_\lambda\right>\\
&=\left<\varphi(W_1,\ldots,
W_n)\psi, z_\lambda\right>\\
&=\left< \psi,\varphi(W_1,\ldots, W_n)^*z_\lambda\right>\\
&=\left< \psi,\overline{\varphi(\lambda)} z_\lambda\right>
=\varphi(\lambda)\psi(\lambda)
\end{split}
\end{equation*}
for any $\psi\in F_s^2({\bf D}^m_f)$ and $\lambda\in {\bf
D}_{f,\circ}^1(\CC)$. Therefore, the operators in $P_{F_s^2({\bf
D}^m_f)} F_n^\infty({\bf D}^m_f) |_{F_s^2({\bf D}^m_f)}$ are
``analytic'' multipliers of $F_s^2({\bf D}^m_f)$. Moreover,
$$
\|M_\varphi\|=\sup\{\|\varphi f\|_2:\ f\in F_s^2({\bf D}^m_f), \,
\|f\|\leq 1\}.
$$
In particular, for each $i=1,\ldots,n$,  $L_i$ is is the multiplier
$M_{\lambda_i}$ by the coordinate function.
Let $H^\infty({\bf D}_{f,\circ}^1(\CC))$ be the algebra of all
multipliers of the Hilbert space $H^2({\bf D}_{f,\circ}^1(\CC))$.
In what follows, we show that the algebra $H^\infty({\bf D}_{f,\circ}^1(\CC))$ is reflexive.

First, we need to recall some definitions.
If $~A\in B(\cH)~$ then the set of all
invariant subspaces of $~A~$ is denoted by $\text{\rm Lat~}A$. For
any $~\cU\subset B(\cH)~$ we define
$$
\text{\rm Lat~}\cU=\bigcap_{A\in\cU}\text{\rm Lat~}A.
$$
 If $~\cS~$ is any collection of subspaces of $~\cH$,
then we define $\text{\rm Alg~}\cS$ by setting $$\text{\rm
Alg~}\cS:=\{A\in B(\cH):\ \cS\subset\text{\rm Lat~}A\}.$$ We recall
that the algebra $~\cU\subset B(\cH)~$ is reflexive if
$\cU=\text{\rm Alg Lat~}\cU.$

\begin{theorem}\label{reflexivity}
The algebra  $H^\infty({\bf D}_{f,\circ}^1(\CC))$ is reflexive and
coincides with the weakly closed algebra generated by the operators
$L_1,\ldots, L_n$ and the identity.
\end{theorem}
\begin{proof} First we show that $H^\infty({\bf
D}_{f,\circ}^1(\CC))$ is  included in the weakly closed algebra
generated by the operators $L_1,\ldots, L_n$ and the identity.
Suppose that $g=\sum_{{\bf k}\in \NN_0} c_{\bf k} w^{\bf k}$ is a
bounded multiplier, i.e., $M_g\in B(F_s^2(\cD_f))$. As in  Section
3,  using Cesaro means,  one can find a sequence of polynomials
$p_m=\sum c_{\bf k}^{(m)} w^{\bf k}$ such that $M_{p_m}$ converges
to $M_g$ in the strong operator topology and, consequently, in  the
$WOT$-topology.   Since $M_{p_m}$ is a polynomial in $L_1,\ldots,
L_n$ and the identity,  our assertion follows.

 Now, let $X\in B(F_s^2({\bf D}^m_f))$ be an operator   that leaves
invariant all the invariant subspaces under each operator
$L_1,\ldots, L_n$. Due to Theorem \ref{eigenvectors}, we have
$L_i^*z_\lambda=\overline{\lambda}_i z_\lambda$ for any $\lambda\in
{\bf D}_{f,\circ}^1(\CC)$ and $i=1,\ldots,n$. Since $X^*$ leaves
invariant all the invariant subspaces under $L_1^*,\ldots, L_n^*$,
the vector $z_\lambda$ must be an eigenvector for $X^*$.
Consequently, there is a function $\varphi:{\bf
D}_{f,\circ}^1(\CC)\to \CC$ such that
$X^*z_\lambda=\overline{\varphi(\lambda)} z_\lambda$ for any
$\lambda\in {\bf D}_{f,\circ}^1(\CC)$. Notice that, if $f\in
F_s^2({\bf D}^m_f)$, then, due to Theorem \ref{symm-Fock}, $Xf$
has the functional representation
$$
(Xf)(\lambda)=\left<Xf,z_\lambda\right>=\left<f,X^*z_\lambda\right>=
\varphi(\lambda)f(\lambda)\quad \text{ for all }\ \lambda\in {\bf
D}_{f,\circ}^1(\CC).
$$
In particular, if $f=1$, then  the the functional representation  of
$X(1)$  coincide with $\varphi$. Consequently, $\varphi$ admits a
power series representation  on ${\bf D}_{f,\circ}^1(\CC)$ and can
be identified with $X(1)\in F_s^2({\bf D}^m_f)$. Moreover, the
equality above shows that $\varphi f\in H^2({\bf
D}_{f,\circ}^1(\CC))$ for any $f\in F_s^2({\bf D}^m_f)$. This shows
that $\varphi$ is in $H^\infty({\bf D}_{f,\circ}^1(\CC))$ and
completes the proof of reflexivity. Hence, $H^\infty({\bf
D}_{f,\circ}^1(\CC))$ is a WOT-closed algebra containing
$L_1,\ldots, L_n$ and the identity. This implies the second part of
the theorem.
\end{proof}

\bigskip

\section{  Noncommutative  Varieties,  Berezin transforms, and Universal Models }

In  this section, we consider noncommutative varieties
$\cV_{f,\cQ}^m(\cH)\subset {\bf D}_f^m(\cH)$ determined by  sets
$\cQ$ of noncommutative polynomials, and  associate  with each such
a variety a universal model $(B_1,\ldots,
B_n)\in\cV_{f,\cQ}^m(\cN_\cQ)$, where $\cN_\cQ$ is an appropriate
subspace of the full Fock  space. We introduce a {\it constrained
noncommutative Berezin transform} and use it to obtain analogues of
the results of Section 2, for subvarieties. We also show that, under
a natural  condition, the $C^*$-algebra $C^*(B_1,\ldots, B_n)$ is
irreducible and all the compacts operators in $B(\cN_\cQ)$ are
contained in the operator space
 $\overline{\text{\rm span}}\{B_\alpha B_\beta^*:\ \alpha,\beta\in \FF_n^+\}$.
 These results are vital for the development of a model theory on noncommutative
  varieties.

Let $f:=\sum_{|\alpha|\geq1} a_\alpha X_\alpha$ be a positive
regular free holomorphic function on $[B(\cH)^n]_\rho$, $\rho>0$,
and
 let $W_1,\ldots, W_n$ be the weighted left creation operators
 associated with  the noncommutative domain ${\bf D}_f^m$.
 Let $\cQ$ be a family of  noncommutative  polynomials   and define the noncommutative variety
 $$
\cV_{f,\cQ}^m(\cH):= \left\{(X_1,\ldots, X_n)\in {\bf D}_f^m(\cH):\
q(X_1,\ldots, X_n)=0\quad \text{ for any }\quad q\in \cQ\right\}.
$$
We associate with $\cV_{f,\cQ}^m$ the   the  operators $B_1,\ldots,
B_n$   defined as follows. Consider  the subspaces
 \begin{equation}
 \label{MQ}
 \cM_\cQ:=\overline{\text{\rm span}}\{W_\alpha q(W_1,\ldots, W_n) W_\beta(1): \ q\in \cQ,
 \alpha, \beta\in \FF_n^+\}
 \end{equation}
 and $\cN_\cQ:=F^2(H_n)\ominus \cM_\cQ$. We assume that $\cN_\cQ\neq
 \{0\}$. It is easy to see that $\cN_\cQ$ is invariant under each
 operator $W_1^*,\ldots, W_n^*$ and $\Lambda^*_1,\ldots, \Lambda^*_n$. Define $B_i:=P_{\cN_\cQ}
 W_i|_{\cN_\cQ}$ and $C_i:=P_{\cN_\cQ} \Lambda_i|\cN_\cQ $ for $i=1,\ldots, n$, where
 $P_{\cN_\cQ}$ is the orthogonal projection of $F^2(H_n)$ onto $\cN_{\cQ}$.
 Notice that  $q(B_1,\ldots,B_n)=0$ for any $q\in \cQ$.
 By taking the compression to the subspace $\cN_\cQ$, in Theorem
 \ref{prop-shif}, we obtain similar results, where the universal
 model  $(W_1,\ldots, W_n)$ is replaced by the $n$-tuple
 $(B_1,\ldots, B_n)$. In particular, we deduce that
   $(B_1,\ldots, B_n)\in \cV_{f,\cQ}^m(\cN_\cQ)$ is a pure $n$-tuple of operators
   which will play the role of universal model
   for the noncommutative variety $\cV_{f,\cQ}^m$.

For each  $n$-tuple $T:=(T_1,\ldots, T_n)\in \cV_{f,\cQ}^m(\cH)$ with
$r_f(T_1,\ldots, T_n)<1$, we introduce the {\it constrained
noncommutative Berezin transform}  at $T$ as the map ${\bf
B}^c_T:B(\cN_\cQ) \to B(\cH)$ defined by

\begin{equation}
\label{Berezin-c}
 \left<{\bf B}^c_T[g]x,y\right>:=
\left<\left( I-\sum_{|\alpha|\geq 1} \overline{a}_{\tilde\alpha}
C_\alpha^* \otimes T_{\tilde\alpha} \right)^{-m}
 (g\otimes \Delta_{T,m,f}^2)
  \left(
I-\sum_{|\alpha|\geq 1} a_{\tilde\alpha} C_\alpha \otimes
T_{\tilde\alpha}^* \right)^{-m}(1\otimes x), 1\otimes y\right>
\end{equation}
  where $\Delta_{T,m,f}:= [(id-\Phi_{f,T})^m(I)]^{1/2}$ and $x,y\in \cH$.
We  define  the {\it extended  constrained noncommutative Berezin
transform} $\widetilde{\bf B}^c_T$
 at any $T\in  \cV_{f,\cQ}^m(\cH)$ by setting
 \begin{equation}
 \label{def-Be2c}
 \widetilde{\bf B}^c_T[g]:= {K_{f,T,\cQ}^{(m)}}^* (g\otimes I_\cH)K_{f,T,\cQ}^{(m)},
 \quad g\in B( \cN_\cQ),
 \end{equation}
where   the {\it constrained noncommutative  Berezin kernel}
associated with the $n$-tuple $T\in \cV_{f,\cQ}^m(\cH)$ is  the
bounded operator  \ $K_{f,T, \cQ}^{(m)}:\cH\to \cN_\cQ\otimes
\overline{\Delta_{f,m,T} \cH}$ defined by
$$K_{f,T,\cQ}^{(m)}:=(P_{\cN_\cQ}\otimes I_{\overline{\Delta_{f,m,T} \cH}})K_{f,T}^{(m)},
$$
where $K_{f,T}^{(m)}$ is the Berezin kernel associated with $T\in
{\bf D}_f^m(\cH)$.

Using the results from Section 2 (see Proposition \ref{Berezin2}), one
can show that the constrained nocommutative Berezin transforms
$\widetilde{\bf B}^c_T$ and
  ${\bf B}^c_T$ coincide for any $n$-tuple of operators
   $T:=(T_1, \ldots, T_n)\in  \cV_{f,\cQ}^m(\cH)$ with joint spectral radius
   $r_f(T_1,\ldots, T_n)<1$.

\begin{theorem}\label{vN1-variety}
Let $f$ be a positive regular free holomorphic function on
$[B(\cH)^n]_\rho$, $\rho>0$,  and  let $\cQ$
  be a family of  noncommutative polynomials such that  $\cN_\cQ\neq
 \{0\}$.
  If $T:=(T_1,\ldots, T_n)$  is a pure
$n$-tuple of operators in the noncommutative variety
$\cV_{f,\cQ}^m(\cH)$, then the restriction of the constrained
noncommutative Berezin transform $\widetilde {\bf B}^c_T$
 to  $\overline{\text{\rm  span}} \{ B_\alpha B_\beta^*:\
\alpha,\beta\in \FF_n^+\}$ is
    a unital completely contractive linear map
such that
 $$
\widetilde{\bf B}^c_T(B_\alpha B_\beta^*)=T_\alpha T_\beta^*, \quad
\alpha,\beta\in \FF_n^+.
$$
\end{theorem}

\begin{proof}
Using  Lemma \ref{Berezin-lemma},
  we have
$$
\left<  K_{f,T}^{(m)}x, W_\alpha q(W_1,\ldots, W_n)
W_\beta(1)\otimes y\right>=\left<x,T_\alpha q(T_1,\ldots,T_n)T_\beta
{K_{f,T}^{(m)}}^*(1\otimes y\right>=0
$$
for any $x\in \cH$, $y\in \overline{\Delta_{f,m,T}\cH}$, and $q\in
\cQ$.  Hence, we deduce  that
\begin{equation}\label{range}
\text{\rm range}\,K_{f,T}^{(m)}\subseteq \cN_\cQ\otimes
\overline{\Delta_{f,m,T}\cH}.
\end{equation}
 Due to the definition of the constrained Berezin kernel associated with
the $n$-tuple $T\in \cV_{f,\cQ}^m(\cH)$, and using  Lemma \ref{Berezin-lemma}
and relation \eqref{range},  we obtain
\begin{equation}\label{KJ}
K_{f,T,\cQ}^{(m)}T_\alpha^*=(B_\alpha^*\otimes
I_\cH)K_{f,T,\cQ}^{(m)},\quad \alpha\in \FF_n^+.
\end{equation}
Since \eqref{range} holds  and $K_{f,T}^{(m)}$ is an isometry,   so is
$K_{f,T,\cQ}^{(m)}$. Consequently, using relation \ref{KJ}, we
deduce that
$$
\widetilde{\bf B}^c_T(B_\alpha
B_\beta^*)={K_{f,T,\cQ}^{(m)}}^*(B_\alpha B_\beta^*\otimes
I_\cH)K_{f,T,\cQ}^{(m)} =T_\alpha T_\beta^*, \quad \alpha,\beta\in
\FF_n^+.
$$
Now, one can easily deduce that  $\widetilde {\bf B}^c_T$ is a
unital completely contractive linear map on $\overline{\text{\rm
span}} \{ B_\alpha B_\beta^*:\ \alpha,\beta\in \FF_n^+\}$. The proof
is complete.
\end{proof}

We recall that an $n$-tuple of operators $T:=(T_1,\ldots, T_n)\in
\cV_{f,\cQ}^m(\cH) $ has the radial property with respect to the
noncommutative variety
 $\cV_{f,\cQ}^m(\cH)$  if there is $\delta\in (0,1)$ such that
 $rT:=(rT_1,\ldots,
 rT_n)\in \cV_{f,\cQ}^m(\cH)$  for any $ r\in (\delta, 1)$.

 \begin{theorem}\label{vN2-variety}
 Let
 $f$ be  a positive regular free holomorphic function on $[B(\cH)^n]_\rho$, $\rho>0$,
  and let $\cQ$ be a set of
    homogenous polynomials. Let
 $T:=(T_1,\ldots, T_n) $  be an $n$-tuple of operators  with the radial property
 in the noncommutative variety
 $\cV_{f,\cQ}^m(\cH)$
  and let
  $\cS:=\overline{\text{\rm  span}} \{ B_\alpha B_\beta^*;\
\alpha,\beta\in \FF_n^+\}$.
   Then there is
    a unital completely contractive linear map
$ \Psi_{f,T,\cQ}: \cS\to B(\cH) $ such that
\begin{equation}\label{Po-tran3}
\Psi_{f,T,\cQ}(g)=\lim_{r\to 1} {\bf B}^c_{rT}[g],\qquad g\in \cS,
\end{equation}
where the limit exists in the norm topology of $B(\cH)$, and \ $
\Psi_{f,T,\cQ}(B_\alpha B_\beta^*)=T_\alpha T_\beta^*, \quad
\alpha,\beta\in \FF_n^+. $
 If, in addition,
 $T$ is a  pure $n$-tuple of operators,
then  $$\lim_{r\to 1} {\bf B}^c_{rT}[g]=\widetilde{\bf
B}^c_T[g],\qquad g\in \cS,$$ where the limit exists in the norm
topology of $B(\cH)$.

 \end{theorem}
\begin{proof} Let  $\delta\in (0,1)$  be such that
 $rT:=(rT_1,\ldots,
 rT_n)\in {\bf D}_{f}^m(\cH)$  for any $ r\in (\delta, 1)$. Since $\cQ$
 consists of homogenous polynomials we also have  $rT\in
 \cV_{f,\cQ}^m(\cH)$.
Moreover, we can show, as in the proof of Theorem \ref{vN1-variety},
that $\text{\rm range}\, K_{f,rT}^{(m)}\subseteq \cN_\cQ\otimes \cH$
for any $r\in (\delta,1)$, where $K_{f,rT}^{(m)}$ is the Berezin
kernel associated with $rT\in{\bf D}^m_f(\cH)$. Moreover,
$$
K_{f,rT,\cQ}^{(m)} r^{|\alpha|} T_\alpha^*=(B_\alpha^*\otimes
I_\cH)K_{f,rT,\cQ}^{(m)}, \quad \alpha\in \FF_n^+,
$$
where $K_{f,rT,\cQ}^{(m)}:=(P_{\cN_\cQ}\otimes I_\cH)
K_{f,rT}^{(m)}$ is the constrained Berezin kernel and
$B_i:=P_{\cN_\cQ} W_i|_{\cN_\cQ} $, \ $i=1,\ldots,n$. Since $rT$ is
 pure, $K_{f,rT,\cQ}^{(m)}$ is an isometry.
Consequently, as in the proof of Theorem \ref{Poisson-C*}, we deduce
that
   there is a unique unital
completely contractive linear map $\Psi_{p,T,\cQ}:\cS\to B(\cH) $
such that $\Psi_{p,T,\cQ}(B_\alpha B_\beta^*)=T_\alpha T_\beta^*$, \
$\alpha,\beta\in \FF_n^+$. The rest of the proof is similar to that
of Theorem \ref{Poisson-C*}. We shall omit it.
\end{proof}

 Assume now  that $p$ is a positive regular noncommutative  polynomial
 and let ${\bf D}^m_p$ be the noncommutative domain  it generates.
 The next result will play an important role in Section  6,  where we    develop  a  model
 theory  on noncommutative subvarieties of ${\bf D}^m_p$.

\begin{theorem}\label{compact}
Let $\cQ$ be a set of noncommutative polynomials such that
$1\in\cN_\cQ$, and let $(B_1,\ldots, B_n)$ be the universal model
associated with the noncommutative variety $\cV_{p,\cQ}^m$. Then all
the compact operators in $B(\cN_\cQ)$ are contained in the operator
space
$$\overline{\text{\rm span}}\{B_\alpha B_\beta^*:\ \alpha,\beta\in \FF_n^+\}.
$$
Moreover,   the $C^*$-algebra  $C^*(B_1,\ldots, B_n)$ is
irreducible.
\end{theorem}

\begin{proof}
Since $1\in \cN_\cQ$ and  $\cN_\cQ$ is an  invariant subspace
$W_i^*$, \ $i=1,\ldots,n$,   we use Theorem \ref{prop-shif}  to obtain
\begin{equation*}
 (id-\Phi_{p,B})^m(I_{\cN_\cQ})=P_{\cN_\cQ}\left[(id-\Phi_{p,W})^m(I_{F^2(H_n)})\right]
 |_{\cN_\cQ}=P_{\cN_\cQ} P_\CC|_{\cN_\cQ}=P^{\cN_\cQ}_\CC,
\end{equation*}
where $ P^{\cN_\cQ}_\CC$ is the orthogonal projection of $\cN_\cQ$
onto $\CC$.
Fix
$$g(W_1,\ldots, W_n):=\sum\limits_{|\alpha|\leq m} d_\alpha W_\alpha\quad
\text{ and } \quad \xi:=\sum\limits_{\beta\in \FF_n^+} c_\beta
e_\beta\in \cN_J\subset F^2(H_n),
$$
and note  that
\begin{equation*}
P^{\cN_\cQ}_\CC g(B_1,\ldots, B_n)^*\xi= \left< \xi,g(B_1,\ldots,
B_n)(1)\right>.
\end{equation*}
Consequently, we have
\begin{equation}\label{rankone}
q(B_1,\ldots, B_n)P^{\cN_J}_\CC g(B_1,\ldots, B_n)^*\xi= \left<
\xi,g(B_1,\ldots, B_n)(1)\right>q(B_1,\ldots, B_n)(1)
\end{equation}
for any polynomial $q(B_1,\ldots, B_n)$.  Hence, we deduce that the
operator $q(B_1,\ldots, B_n)P^{\cN_J}_\CC g(B_1,\ldots, B_n)^*$
 has rank one  and, since
$P_\CC^{\cN_\cQ}= (id-\Phi_{p,B})^m(I_{\cN_\cQ})$, it is   in the
operator space $\overline{\text{\rm span}}\{B_\alpha B_\beta^*:\
\alpha,\beta\in \FF_n^+\}$. On the other hand, due to the fact that
the set of all vectors  of the form
 $ \sum\limits_{|\alpha|\leq m}d_\alpha
B_\alpha (1)$ with  $m\in \NN$, $d_\alpha\in \CC$, is  dense in
$\cN_\cQ$, relation \eqref{rankone} implies  that all compact
operators  in $B(\cN_\cQ)$ are included in the operator space
$\overline{\text{\rm span}}\{B_\alpha B_\beta^*:\ \alpha,\beta\in
\FF_n^+\}$.

To prove the last part of this theorem, let $\cM\neq\{0\}$ be a
subspace of $\cN_\cQ\subseteq F^2(H_n)$, which is jointly reducing
for each operator $B_i$, $i=1,\ldots,n$. Let $\varphi\in \cM$,
$\varphi\neq 0$, and assume that
$\varphi=c_0+\sum\limits_{|\alpha|\geq 1} c_\alpha e_\alpha.$  If
 $c_\beta$  is a nonzero coefficient of $\varphi$,
then
 $
P_\CC B_\beta^*\varphi= \frac{1}{\sqrt{b_\beta^{(m)}}} c_\beta$.
Indeed, since $1\in \cN_\cQ$, one can use relation \eqref{WbWb} to
deduce that
\begin{equation*}
\left< P_\CC
B_\beta^*\varphi,1\right>=\left<P_{\cN_J}W_\beta^*\varphi,1\right>
=\left<W_\beta^*\varphi, 1\right>= \frac{1}{\sqrt{b_\beta^{(m)}}}
c_\beta.
\end{equation*}
Since $\left< P_\CC B_\beta^*\varphi,e_\gamma\right>=0$ for any
$\gamma\in \FF_n^+$ with $|\gamma|\geq 1$, our assertion  follows.
On the other hand, since
 $P_\CC^{\cN_\cQ}= (id-\Phi_{p,B})^m(I_{\cN_\cQ})$ and $\cM$ is reducing for $B_1,\ldots, B_n$,
we deduce that $c_\beta\in \cM$, so $1\in \cM$. Using once again
that $\cM$ is invariant under the operators $B_1,\ldots, B_n$, we
have
 $\cE\subseteq \cM$. On the other hand, since $\cE$ is dense in $\cN_\cQ$, we deduce that
 $\cN_\cQ\subset \cM$. Therefore $\cN_\cQ= \cM$.
This completes the proof.
\end{proof}

We say that two $n$-tuples of operators $(T_1,\ldots, T_n)$, \
$T_i\in B(\cH)$, and $(T_1',\ldots, T_n')$, \ $T_i'\in B(\cH')$, are
unitarily equivalent if there exists a unitary operator $U:\cH\to
\cH'$ such that
$$T_i=U^* T_i' U\ \text{  for  any } \ i=1,\ldots, n.
$$
 If $(B_1,\ldots, B_n)$ is the universal model
  associated with the noncommutative variety $\cV_{p,\cQ}^m$,
then the $n$-tuple $(B_1\otimes I_\cH,\ldots, B_n\otimes I_\cH)$ is
called constrained weighted shift with multiplicity $\dim \cH$.
Using Theorem \ref{compact}, one can easily  prove  that two
constrained weighted shifts associated with the noncommutative
variety $\cV_{p,\cQ}^m$ are unitarily equivalent if and only if
their multiplicities are equal.

We remark that all the results of this section are true in the
commutative case, i.e., when
$$\cQ_c:=\{Z_iZ_j-Z_jZ_i:\
i,j=1,\ldots,n\}.
$$
According to the results of Section 4 (see Theorem \ref{symm-Fock} and
the remarks preceding Theorem \ref{reflexivity}), the space
$\cN_{\cQ_c}$ coincides with
 the symmetric
weighted Fock space $F_s^2({\bf D}^m_f)$, which  can be identified
with the Hilbert space $H^2({\bf D}_{f,\circ}^1(\CC))$. Moreover,
under this identification, the operators $B_i$, $i=1,\ldots,n$,
become the multipliers $M_{\lambda_i}$ by the coordinate functions
on the Hilbert space $H^2({\bf D}_{f,\circ}^1(\CC))$.

\bigskip

\section{ Model theory  on  Noncommutative  Varieties }

  In this section, we obtain dilation and model theorems for the elements of
  the noncommutative variety $\cV_{f,\cQ}^m(\cH)\subset {\bf D}_f^m(\cH)$
  generated by a set $\cQ$ of noncommutative polynomials.

We recall  that $\cN_\cQ:=F^2(H_n)\ominus \cM_\cQ$, where the subspace $\cM_\cQ$ is defined
by \eqref{MQ}. We keep the notations of the previous sections.
Our first dilation result   on noncommutative varieties
  is the following.

\begin{theorem}\label{dil1} Let $f$ be a positive regular free
holomorphic function on $[B(\cH)^n]_\rho$, $\rho>0$,  and  let $\cQ$
  be a family of  noncommutative polynomials such that  $\cN_\cQ\neq
 \{0\}$.
  If $T:=(T_1,\ldots, T_n)$  is an
$n$-tuple of operators in the noncommutative variety
$\cV_{f,\cQ}^m(\cH)$, then there exists a Hilbert space $\cK$  and
and $n$-tuple $(U_1\ldots, U_n)\in \cV_{f,\cQ}^m(\cK)$ with
$\Phi_{f, U}(I_\cK)=I_\cK$ and such that
\begin{enumerate}
\item[(i)]
 $\cH$ can be identified with a  co-invariant subspace of \
   $\widetilde\cK:=(\cN_\cQ\otimes \overline{\Delta_{f,m,T}\cH})\oplus \cK$ under the operators
$$
V_i:=\left[\begin{matrix} B_i\otimes
I_{\overline{\Delta_{f,m,T}\cH}}&0\\0&U_i
\end{matrix}\right],\quad i=1,\ldots,n,
$$
where $\Delta_{f,m,T}:=[\left(id- \Phi_{f,T}\right)^m(I)]^{1/2};$
\item[(ii)]
$T_i^*=V_i^*|\cH$ for  $i=1,\ldots,n$.
\end{enumerate}
Moreover, $\cK=\{0\}$ if and only if $(T_1,\ldots, T_n)$ is  pure
$n$-tuple of operators in $\cV_{f,\cQ}^m(\cH)$,
i.e., $\Phi_{f,T}^k(I)\to 0$ strongly, as $k\to 0$.
\end{theorem}

\begin{proof}
We recall that the operator   $Q_{f,T}:=\text{\rm
SOT-}\lim\limits_{k\to\infty} \Phi_{f,T}^k(I)$  is well-defined. We use it  to define
$$Y:\cH\to \cK:=\overline{Q_{f,T}^{1/2}\cH}\quad \text{  by }\quad
Yh:=Q_{f,T}^{1/2}h, \ h\in \cH.
$$
For each $i=1,\ldots, n$,  let $L_i:Q_{f,T}^{1/2}\cH\to
\cK$ be  given by
\begin{equation}\label{Z_i}
L_i Yh:=YT_i^*h,\quad h\in \cH.
\end{equation}
Note that $L_i$, $i=1,\ldots, n$, are well-defined  due to the fact
that
\begin{equation*}
\begin{split}
\|L_i Yh\|^2&=\left<   T_i Q_{f,T}T_i^*h,h \right>  \leq
 \frac{1}{a_{g_i}}\left< \Phi_{f,T}(Q_{f,T})h,h\right>\\
 &=\frac{1}{a_{g_i}} \|Q_{f,T}^{1/2}h\|^2=\frac{1}{a_{g_i}}\|Yh\|^2.
\end{split}
\end{equation*}
Since $f$ is positive regular free holomorphic
function,  we have  $a_{g_i}\neq 0$ for any $i=1,\ldots,n$.
 Consequently,  $L_i$ can be extended to a bounded operator on $\cK$, which will
also be denoted by $L_i$. Now, setting $U_i:=L_i^*$, $i=1,\ldots,
n$, relation  \eqref{Z_i} implies
\begin{equation}
\label{YZT} Y^*U_i=T_iY^*,\quad i=1,\ldots, n.
\end{equation}
 Using this relation and the fact that $\Phi_{f,T}(Q_{f,T})=Q_{f,T}$, we deduce that
\begin{equation*}
Y^* \Phi_{f,U}(I_\cK)Y=\Phi_{f,T}(YY^*)=YY^*.
\end{equation*}
Hence,
$$\left< \Phi_{f,U}(I_\cK)Yh,Yh\right>=
\left< Yh,Yh\right>,\quad h\in \cH,
$$
which implies $\Phi_{f,U}(I_\cK)=I_\cK$.
 Now, using relation \eqref{YZT}, we obtain
$$Y^*q(U_1,\ldots, U_n)=q(T_1,\ldots, T_n)Y^*=0, \quad q\in\cQ.
$$
Since $Y^*$ is injective on $\cK=\overline{Y\cH}$, we have
$q(U_1,\ldots, U_n)=0$ for any $q\in \cQ$.
Let $V:\cH\to [\cN_\cQ\otimes \cH]\oplus \cK$ be defined by
$$V:=\left[\begin{matrix}
K_{f,T,\cQ}^{(m)}\\ Y
\end{matrix}\right].
$$
Notice  that $V$ is an isometry. Indeed, due to relations
\eqref{K*K}  and \eqref{range}, we have
\begin{equation*}\begin{split}
\|Vh\|^2&=\|K_{f,T,\cQ}^{(m)}h\|^2+\|Yh\|^2\\
&=\|h\|^2-\text{\rm SOT-}\lim_{k\to\infty}\left<
\Phi_{f,T}^k(I)h,h\right>+\|Yh\|^2\ =\|h\|^2
\end{split}
\end{equation*}
for any $h\in \cH$. Now, using relations \eqref{KJ}, \eqref{Z_i},
and \eqref{YZT},  we obtain
\begin{equation*}
\begin{split}
VT_i^*h&= K_{f,T,\cQ}^{(m)}T_i^*h\oplus YT_i^*h\\
&=(B_i^*\otimes I_\cH)K_{f,T,\cQ}^{(m)}h\oplus U_i^*Yh\\
&=\left[\begin{matrix} B_i^*\otimes
I_{\overline{\Delta_{f,m,T}\cH}}&0\\0&U_i^*
\end{matrix}\right]Vh
\end{split}
\end{equation*}
for any $h\in \cH$ and $i=1,\ldots, n$.
 Identifying   $\cH$  with $V\cH$ we
complete the proof of (i) and (ii). The  last part  of the theorem
is obvious.
\end{proof}

We need the following result concerning power bounded positive linear maps on $B(\cH)$.

\begin{lemma}\label{=0}
Let $\varphi: B(\cH)\to B(\cH)$ be a power bounded positive
 linear map and let $D\in B(\cH)$be a positive operator
 such that $\varphi(D)\leq D$.
 If  $m\geq 1$, then
 $$(id-\varphi)^m(D)=0\quad \text{  if and only if }\quad \varphi(D)=D.
 $$
In particular, if $\varphi$  is a positive linear map such that
  $\varphi(I)\leq I$ and $(id-\varphi)^m(I)=0$, then $\varphi(I)=I$.
\end{lemma}
\begin{proof}
According to Lemma \ref{identity}, we have
\begin{equation*}
\sum_{p=0}^q\left(\begin{matrix} p+m-1\\m-1
\end{matrix}\right) \varphi^p(id-\varphi)^m(D)=D-\sum_{j=0}^{m-1}
 \left(\begin{matrix} q+j\\ j
\end{matrix}\right) \varphi^{q+1} (id-\varphi)^j(D)
\end{equation*}
for any $q\in \NN$. Consequently, if $(id-\varphi)^m(D)=0$,  then
$$D=\lim_{q\to \infty} \sum_{j=0}^{m-1}
 \left(\begin{matrix} q+j\\ j
\end{matrix}\right) \varphi^{q+1} (id-\varphi)^j(D).
$$
Using Lemma \ref{ineq-lim}, we deduce that $D=\lim_{q\to \infty} \varphi^q(D)$.
Since $\varphi$ is a positive linear map and $\varphi(D)\leq D$, we have
$$
D=\lim_{q\to \infty} \varphi^q(D)\leq \ldots \leq \varphi^2(D)\leq \varphi(D)\leq D.
$$
Hence, we deduce that $\varphi(D)=D$. The converse is obvious.
\end{proof}

Let $C^*(\Gamma)$ be the $C^*$-algebra generated by  a set of operators $\Gamma\subset B(\cK)$ and the identity.
A subspace $\cH\subset \cK$ is called $*$-cyclic for $\Gamma$ if
$\cK=\overline{\text{\rm span}}\{Xh, X\in C^*(\Gamma), h\in \cH\}$.
The main result of this section is the following model theorem  for the elements of
a noncommutative variety $\cV_{p,\cQ}^m(\cH)$.

\begin{theorem}\label{dil2} Let
 $p$ be  a positive regular noncommutative
polynomial and let $\cQ$ be a set of
    homogenous polynomials. Let $\cH$ be a separable Hilbert space, and
 $T:=(T_1,\ldots, T_n) $  be an $n$-tuple of operators  in the noncommutative variety
 $\cV_{p,\cQ}^m(\cH)$ with the  radial property, i.e.,
 $$rT:=(rT_1,\ldots,
 rT_n)\in \cV_{p,\cQ}^m(\cH)\quad \text{ for any } \ r\in (\delta, 1)
 $$
 and some  $\delta\in (0,1)$.

Then there exists  a $*$-representation $\pi:C^*(B_1,\ldots, B_n)\to
B(\cK_\pi)$  on a separable Hilbert space $\cK_\pi$,  which
annihilates the compact operators and
$$
\Phi_{p,\pi(B)}(I_{\cK_\pi})=I_{\cK_\pi},
$$
such that
\begin{enumerate}
\item[(i)]
$\cH$ can be identified with a $*$-cyclic co-invariant subspace of
$\tilde\cK:=(\cN_\cQ\otimes \overline{\Delta_{p,m,T}\cH})\oplus
\cK_\pi$ under  each operator
$$
V_i:=\left[\begin{matrix} B_i\otimes
I_{\overline{\Delta_{p,m,T}\cH}}&0\\0&\pi(B_i)
\end{matrix}\right],\quad i=1,\ldots,n,
$$
where $\Delta_{p,m,T}:=[\left( id-\Phi_{p,T}\right)^m(I)]^{1/2};$
\item[(ii)]
$ T_i^*=V_i^*|\cH$ for $ i=1,\ldots, n.$
\end{enumerate}
  \end{theorem}
 \begin{proof}
 Applying Arveson extension theorem
\cite{Arv-acta} to the map $\Psi_{p,T,\cQ}$ of Theorem \ref{vN2-variety}, we find a unital
completely positive linear map $\Psi_{p,T,\cQ}:C^*(B_1,\ldots,
B_n)\to B(\cH)$ such that $\Psi_{p,T,\cQ}(B_\alpha B_\beta^*)=T_\alpha T_\beta^*$
 for $\alpha.\beta\in \FF_n^+$.
 Let $\tilde\pi:C^*(B_1,\ldots, B_n)\to
B(\tilde\cK)$ be a minimal Stinespring dilation \cite{St} of
$\Psi_{p,T,\cQ}$. Then
$$\Psi_{p,T,\cQ}(X)=P_{\cH} \tilde\pi(X)|\cH,\quad X\in C^*(B_1,\ldots, B_n),
$$
and $\tilde\cK=\overline{\text{\rm span}}\{\tilde\pi(X)h:\ h\in
\cH\}.$  Now, one can easily see that
  that $P_\cH \tilde\pi(B_i)|_{\cH^\perp}=0$, $i=1,\ldots, n$.
   Consequently,
  $\cH$ is an
invariant subspace under each $\tilde\pi(B_i)^*$, \ $i=1,\ldots, n$,
and
\begin{equation}\label{coiso}
\tilde\pi(B_i)^*|\cH=\Psi_{p,T,\cQ}(B_i^*)=T_i^*,\quad i=1,\ldots,
n.
\end{equation}

Since  $1\in \cN_\cQ$,   Theorem \ref{compact} implies that all the
compact operators $ \cC(\cN_\cQ)$ in $B(\cN_\cQ)$ are contained in the
$C^*$-algebra $C^*(B_1,\ldots, B_n)$.
 Due to  standard theory of
representations of  $C^*$-algebras \cite{Arv-book},
representation $\tilde\pi$ decomposes into a direct sum
$\tilde\pi=\pi_0\oplus \pi$ on $\tilde \cK=\cK_0\oplus \cK_\pi$,
where $\pi_0$, $\pi$  are disjoint representations of
$C^*(B_1,\ldots, B_n)$ on the Hilbert spaces
$$\cK_0:=\overline{\text{\rm span}}\{\tilde\pi(X)\tilde\cK:\ X\in \cC(\cN_\cQ)\}
\quad \text{ and  }\quad  \cK_\pi:=\cK_0^\perp,
$$
respectively,
such that
 $\pi$ annihilates  the compact operators in $B(\cN_\cQ)$, and
  $\pi_0$ is uniquely determined by the action of $\tilde\pi$ on the
  ideal $\cC(\cN_\cQ)$ of compact operators.
Since every representation of $\cC(\cN_\cQ)$ is equivalent to a
multiple of the identity representation, we deduce
 that
\begin{equation}\label{sime}
\cK_0\simeq\cN_\cQ\otimes \cG, \quad  \pi_0(X)=X\otimes I_\cG, \quad
X\in C^*(B_1,\ldots, B_n),
\end{equation}
 for some Hilbert space $\cG$.
 Using Theorem \ref{compact} and its proof, one can
easily see
that
\begin{equation*}\begin{split}
\cK_0&:=\overline{\text{\rm span}}\{\tilde\pi(X)\cK:\ X\in \cC(\cN_\cQ)\}\\
&=\overline{\text{\rm span}}\{\tilde\pi(B_\beta P_\CC^{\cN_\cQ} B_\alpha^*)\cK:\
 \alpha, \beta\in \FF_n^+\}\\
&= \overline{\text{\rm span}}\left\{\tilde\pi(B_\beta) \left[(id-\Phi_{p,\tilde\pi(B)})^m
(I_\cK)\right] \cK:\ \beta\in
\FF_n^+\right\}.
\end{split}
\end{equation*}
According to Theorem \ref{compact}, the operator
$(id-\Phi_{p,B})^m(I_{\cN_\cQ})= P_\CC^{\cN_\cQ}$  is a
 projection  of rank one in $C^*(B_1,\ldots, B_n)$.
 Hence, we deduce  that
$
 (id-\Phi_{p,\pi(B)})^m(I_{\cK_\pi})= 0$
  and
$$
\dim \cG=\dim \left[\text{\rm range}\,\pi(P_\CC^{\cN_\cQ})\right].
$$
Since  the  Stinespring representation $\tilde\pi$  is minimal, we
can use the proof of Theorem \ref{compact} to  deduce that
\begin{equation*}
\text{\rm range}\,\tilde\pi(P_\CC^{\cN_\cQ})=
 \overline{\text{\rm
span}}\{\tilde\pi(P_\CC^{\cN_\cQ})\tilde\pi(B_\beta^*)h:\ \beta\in
\FF_n^+, h\in \cH\}.
\end{equation*}
 On the other hand, it is easy to see that
\begin{equation*}
\left<\tilde\pi(P_\CC^{\cN_\cQ})\tilde\pi(B_\alpha^*)h,
\tilde\pi(P_\CC^{\cN_\cQ})\tilde\pi(B_\beta^*)k\right>
 = \left<h,T_\alpha\left[(id-\Phi_{p,T})^m(I_\cH) \right]T_\beta^*h\right> =
\left<\Delta_{p,m,T}T_\alpha^*h,\Delta_{p,m,T}T_\beta^*k\right>
\end{equation*}
for any $h, k \in \cH$ and $\alpha,\beta\in \FF_n^+$. This implies the existence of      a unitary
operator $\Lambda:\text{\rm range}\,\tilde\pi(P_\CC^{\cN_\cQ})\to
\overline{\Delta_{p,m,T}\cH}$ defined by
$$
\Lambda[\tilde\pi(P_\CC^{\cN_\cQ})\tilde\pi(B_\alpha^*)h]:=\Delta_{p,m,T}
T_\alpha^*h,\quad h\in \cH, \,\alpha\in \FF_n^+.
$$
 This shows that
$$
\dim[\text{\rm range}\,\pi(P_\CC^{\cN_\cQ})]= \dim
\overline{\Delta_{p,m,T}\cH}=\dim \cG.
$$
 Using relations \eqref{coiso} and
\eqref{sime},  and identifying    $\cG$ with
$\overline{\Delta_{p,m,T}\cH}$, we obtain the required dilation.
On the other hand, due to the fact that $(id-\Phi_{p,\pi(B)})^m(I_{\cK_\pi})= 0$,
  we can use  Lemma \ref{=0}  to deduce that
    $ \Phi_{p,\pi(B)}(I_{\cK_\pi})=I_{\cK_\pi}$.
 The proof is complete.
\end{proof}

A few remarks are needed. A closer look at Theorem \ref{dil2} reveals that
one can replace the polynomial
$p$ with a positive regular free holomorphic function $f$ and obtain  a
 model theorem for any $n$-tuple $(T_1,\ldots, T_n) \in \cV_{f,\cQ}^m(\cH)$
  with the  radial property. More precisely, one can show that there is a
  $*$-representation $\tilde\pi:C^*(B_1,\ldots, B_n)\to
B(\cK_\pi)$  such that  $\cH$ is an invariant subspace  under each operator
$\tilde\pi(B_i)^*$  and $T_i^*=\tilde\pi(B_i)^*|_{\cH}$ for $i=1,\ldots,n$.

 On the other hand, notice that using  the proof of Theorem \ref{dil2}
  and due to  the  standard theory of
representations of   $C^*$-algebras,   one can deduce
 the following   Wold type
decomposition for non-degenerate $*$-representations of the
$C^*$-algebra $C^*(B_1,\ldots, B_n)$, generated by the the
constrained weighted shifts associated with $\cV_{p,\cQ}^m$,  and
the identity.

\begin{corollary}\label{wold} Let
$p$ be a positive regular noncommutative polynomial and let  $\cQ$
be a set of noncommutative polynomials such that $1\in \cN_\cQ$.
Let $(B_1,\ldots, B_n)$ be the universal model  associated with the
noncommutative variety $\cV_{p,\cQ}^{(m)}$. If
\ $\pi:C^*(B_1,\ldots, B_n)\to B(\cK)$ is  a nondegenerate
$*$-representation  of \ $C^*(B_1,\ldots, B_n)$ on a separable
Hilbert space  $\cK$, then $\pi$ decomposes into a direct sum
$$
\pi=\pi_0\oplus \pi_1 \  \text{ on  } \ \cK=\cK_0\oplus \cK_1,
$$
where $\pi_0$ and  $\pi_1$  are disjoint representations of \
$C^*(B_1,\ldots, B_n)$ on the Hilbert spaces
\begin{equation*}
\cK_0:=\overline{\text{\rm span}}\left\{\pi(B_\beta) \left[\left(
id-\Phi_{p,\pi(B)}\right)^m(I_\cK)\right] \cK:\ \beta\in
\FF_n^+\right\}\quad \text{ and }\quad
 \cK_1:=\cK_0^\perp,
\end{equation*}
 respectively, where $\pi(B):=(\pi(B_1),\ldots, \pi(B_n))$. Moreover,
  up to an isomorphism,
\begin{equation*}
\cK_0\simeq\cN_\cQ\otimes \cG, \quad  \pi_0(X)=X\otimes I_\cG \quad
\text{ for } \  X\in C^*(B_1,\ldots, B_n),
\end{equation*}
 where $\cG$ is  a Hilbert space with
$
\dim \cG=\dim \left\{\text{\rm range}\, \left[\left(
id-\Phi_{p,\pi(B)}\right)^m(I_\cK)\right]\right\},
$
 and $\pi_1$ is a $*$-representation  which annihilates the compact operators   and
$$
\Phi_{p,\pi_1(B)}(I_{\cK_1})=I_{\cK_1}.
$$
If  $\pi'$ is another nondegenerate  $*$-representation of
$C^*(B_1,\ldots, B_n)$ on a separable  Hilbert space  $\cK'$, then
$\pi$ is unitarily equivalent to $\pi'$ if and only if
$\dim\cG=\dim\cG'$ and $\pi_1$ is unitarily equivalent to $\pi_1'$.
\end{corollary}
We remark that under
 the hypotheses and notations of Corollary $\ref{wold}$, and setting
$V_i:=\pi(B_i)$, \ $i=1,\ldots, n$,   the following statements are
equivalent:
\begin{enumerate}
\item[(i)]
$V:=(V_1,\ldots, V_n)$ is a constrained  weighted shift in the
noncommutative variety $\cV_{p,\cQ}^m(\cK)$;
\item[(ii)] \text{\rm SOT-}$\lim\limits_{k\to\infty}
 \Phi^k_{p, V}(I)=0$;
 \item[(iii)]
$ \cK=\overline{\text{\rm span}}\left\{V_\beta \left[(id- \Phi_{p,
V})^m (I)\right] \cK:\  \beta\in \FF_n^+\right\}; $
\item[(iv)]
$\sum\limits_{\beta\in \FF_n^+} b_\beta^{(m)} V_\beta \left[(id-
\Phi_{p, V})^m (I)\right] V_\beta^*=I_\cK$, where $b_\beta^{(m)}$
are the coefficients defined by \eqref{b-al}.
\end{enumerate}

 We mention that,  under the
additional  condition that
\begin{equation*}
\overline{\text{\rm span}}\,\{B_\alpha B_\beta^*:\ \alpha,\beta\in
\FF_n^+\}=C^*(B_1,\ldots, B_n),
\end{equation*}
  the map $\Psi_{p,T,\cQ}$ in the
proof of  Theorem \ref{dil2} is unique. The uniqueness of the
minimal Stinespring representation  \cite{St} and the  the above-mentioned
Wold type decomposition  imply  the
uniqueness of the minimal dilation of Theorem \ref{dil2}.

\begin{corollary} \label{part-cases}
Let $V:=(V_1,\ldots, V_n)\in {\cV}^m_{p,\cQ}(\cK)$ be the dilation
 of $T:=(T_1,\ldots,T_n)\in \cV_{p,\cQ}^m(\cH)$, given by Theorem $\ref{dil2}$.
 Then,
\begin{enumerate}
\item[(i)]
 $V$ is a constrained  weighted shift if and only if
$T$ is  a pure $n$-tuple of operators;
\item[(ii)]
   $\Phi_{p,V}(I_{\widetilde \cK})=I_{\widetilde \cK}$ \
if and only if \  $\Phi_{p,T}(I_\cH)=I_\cH$.
\end{enumerate}
\end{corollary}
\begin{proof} According to Theorem \ref{dil2}, we have
$$
\Phi_{p,T}^k(I_\cH)=P_\cH \left[\begin{matrix}
\Phi_{p,B}^k(I_{\cN_\cQ})\otimes
I_{\overline{\Delta_{p,m,T}\cH}}&0\\0&
I_{\cK_\pi}\end{matrix}\right]|\cH\quad \text{ for } \ k=1,2,\ldots,
$$
which implies
$$
\text{\rm SOT-}\lim_{k\to\infty} \Phi_{p,T}^k(I_\cH)=P_\cH
\left[\begin{matrix}
 0&0\\0& I_{\cK_\pi}\end{matrix}\right]|\cH.
$$
Consequently,   $T$ is pure   if and only if $P_\cH P_{\cK_\pi}
|\cH=0$. The latter condition is equivalent to $\cH\perp (0\oplus
\cK_\pi)$, which, according to Theorem \ref{dil2}, is equivalent to
  $\cH\subset \cN_\cQ\otimes\overline{\Delta_{p,m,T}\cH}$. On
the other hand, since $\cN_\cQ\otimes\overline{\Delta_{p,m,T}\cH}$
is reducing for $V_1,\ldots, V_n$, and $\widetilde \cK$ is the
smallest reducing subspace for   $V_1,\ldots, V_n$, which contains
$\cH$, we must have
$\widetilde\cK=\cN_\cQ\otimes\overline{\Delta_{p,m,T}\cH}$.
Therefore, item (i) holds.

   To prove part (ii), note that
$$\left( id-\Phi_{p,V}\right)^m(I_{\widetilde
\cK})= \left[\begin{matrix}  \left[\left(
id-\Phi_{p,B}\right)^m(I_{\cN_\cQ})\right]\otimes
I_{\overline{\Delta_{p,m,T}\cH}}&0\\0&
 0\end{matrix}\right].
$$
Hence, we  deduce that $\left( id-\Phi_{p,V}\right)^m(I_{\widetilde
\cK})=0$ if and only if  $\left[\left(
id-\Phi_{p,B}\right)^m(I_{\cN_\cQ})\right]\otimes
I_{\overline{\Delta_{p,m,T}\cH}}=0$. On the other hand,  we know
that   $\left( id-\Phi_{p,B}\right)^m(I_{\cN_\cQ})=P^{\cN_\cQ}_\CC$.
Consequently, $\left( id-\Phi_{p,V}\right)^m(I_{\widetilde \cK})=0$
if and only if $\Delta_{p,m,T}=0$.
 Now, using Lemma \ref{=0},  we obtain the
equivalence in part (ii). The proof is complete.
\end{proof}

We mention now  a few  remarkable particular cases, when Theorem \ref{dil2} applies.

\begin{remark}
\begin{enumerate}
\item[(i)] In the particular case when $m=1$, $n=1$,
$p=X$, and $\cQ=0$, we obtain the classical isometric dilation
theorem for contractions obtained by Sz.-Nagy $($see \cite{Sz1},
\cite{SzF-book}$)$.
\item[(ii)]
When $m=1$, $n\geq 2$,   $p=X_1+\cdots +X_n$,  and $\cQ=0$ we obtain the
noncommutative dilation theorem for  row contractions (see \cite{F},
\cite{B}, \cite{Po-isometric}).
\item[(iii)]
In
   the single variable case, when  $m\geq 2$,  $n=1$,  $p=X$, and $\cQ=0$,  the
    corresponding domain coincides with the set of all
    $m$-hypercontractions  studied by Agler  in \cite{Ag1}, \cite{Ag2},
    and  recently by
    Olofsson \cite{O1}, \cite{O2}.
\item[(iv)]
 When $m\geq 2$, $n\geq2$, $p=X_1+\cdots +X_n$,  and  $\cQ=0$, the elements
    of the corresponding domain ${\bf D}_p^m(\cH)$ can be seen as
    multivariable noncommutative analogues of Agler's
    $m$-hypercontractions.
\item[(v)]
In the particular case when $\cQ_c$  consists of  the polynomials
  $Z_iZ_j-Z_jZ_i$,  $i,j=1,\ldots, n$, we recover  several results concerning
   model theory for commuting $n$-tuples of operators.
   The case $n\geq 2$, $m\geq 2$, $p=X_1+\cdots + X_n$, and $\cQ=\cQ_c$, was studied
   by Athavale \cite{At}, M\" uller \cite{M}, M\" uller-Vasilescu \cite{MV},
   Vasilescu \cite{Va}, and Curto-Vasilescu \cite{CV1}.
\item[(vi)]
When  $p$ is a positive regular noncommutative polynomial and  $\cQ$
consists of  the polynomials
$$
W_iW_j-W_jW_i,\qquad  i,j=1,\ldots, n, $$ we obtain the dilation
theorem of  S. Pott \cite{Pot}.
\item[(vii)]  When $m=1$, $n\geq 1$, and $p$ is
   any positive regular noncommutative  polynomial
     we find  the dilation theorem
   obtained in   \cite{Po-domains}.
\end{enumerate}
\end{remark}

 We expect to use the results of
the present  paper to obtain functional models for the elements of
the noncommutative domain ${\bf D}_f^m(\cH)$ (resp. subvariety
$\cV_{f,\cQ}^m(\cH)$), based on characteristic functions.

      \bigskip

       %

      \end{document}